\documentclass[a4paper, reqno]{amsart} 

\usepackage[T1]{fontenc}
\usepackage[latin1]{inputenc}
\usepackage{newlfont,verbatim, indentfirst,enumerate}

\usepackage{amsmath, amsthm, amssymb, amscd}
\usepackage{latexsym, amsfonts, amsbsy, mathrsfs}
\usepackage{array, pst-all, graphicx, rotating}

\usepackage{tikz} \usetikzlibrary{arrows} \usetikzlibrary{patterns}

\usepackage[unicode]{hyperref}
\hypersetup{unicode=true, pdftoolbar=true, pdfmenubar=true, pdffitwindow=false, pdfstartview={FitH},
  pdftitle={Orbifolds and groupoids}, pdfauthor={Matteo Tommasini},
  pdfkeywords={reduced effective complex orbifolds, proper etale effective groupoids, Morita equivalences},
  pdfnewwindow=true, colorlinks=false, linkcolor=red, citecolor=red, filecolor=magenta, urlcolor=cyan}

\numberwithin{equation}{section}

\renewenvironment{itemize}{\begin{list}{$\bullet$}{\leftmargin=0.5cm}\parindent=0pt}{\end{list}}

\theoremstyle{plain}
\newtheorem{theorem}{Theorem}[section]
\newtheorem{teo}[theorem]{Theorem}
\newtheorem{prop}[theorem]{Proposition}
\newtheorem{lem}[theorem]{Lemma}   
\newtheorem{cor}[theorem]{Corollary}
\newtheorem{definprop}[theorem]{Definition-proposition}
\newtheorem{definlem}[theorem]{Definition-lemma}

\theoremstyle{definition}
\newtheorem{defin}[theorem]{Definition}
\newtheorem{rem}[theorem]{Remark}

\def \tA {\widetilde{A}} \def \tB {\widetilde{B}} \def \tC {\widetilde{C}}
\def \tD {\widetilde{D}}  
  \def \tU {\widetilde{U}}
\def \tV {\widetilde{V}} \def \tW {\widetilde{W}} \def \tZ {\widetilde{Z}}
          
     \def \tf {\tilde{f}}     \def \tg {\tilde{g}}
\def \th {\tilde{h}}     \def \tk {\tilde{k}}     
     \def \ts {\tilde{s}}     \def \tx {\tilde{x}}
\def \ty {\tilde{y}}          

\def \AND {\textrm{\quad and\quad}}
\def \groupR {R\rightrightarrows^{\hspace{-0.25 cm}^{s}}_{\hspace{-0.25 cm}_{t}}U}
\def \groupRR {R'\rightrightarrows^{\hspace{-0.28 cm}^{s'}}_{\hspace{-0.28 cm}_{t'}}U'}
\def \groupRRR {R''\rightrightarrows^{\hspace{-0.30 cm}^{s''}}_{\hspace{-0.28 cm}_{t''}}U''}

\newcommand {\CAT} [1] {(\textbf{#1})}
\newcommand {\ca} [1] {\mathscr{#1}}
\newcommand {\LA} [1] {\lambda_{#1}}
\newcommand {\fibra} [3] {#1\,\times_{#2}\, #3}
\newcommand {\fibre} [4] {{#1}\,_{#2}\times_{#3}\, {#4}}
\newcommand {\unif} [1] {(\tU_{#1},G_{#1},\pi_{#1})}
\newcommand {\uniff} [1] {(\tV_{#1},H_{#1},\phi_{#1})}
\newcommand {\unifff} [2] {(\tV_{#1}^{#2},H_{#1}^{#2},\phi_{#1}^{#2})}
\newcommand {\fre} [1] {\stackrel{#1}{\rightarrow}}
\newcommand {\sinistra} [1] {\stackrel{#1}{\leftarrow}}
\newcommand {\group} [1] {R^{#1}\rightrightarrows^{\hspace{-0.30 cm}^{s^{#1}}}_{\hspace{-0.28 cm}_{t^{#1}}}U^{#1}}

\begin{document}

\title{Orbifolds and groupoids\\
Preliminary version - comments welcome}

\author{Matteo Tommasini}

\address{Mathematical Physics Sector\\
SISSA - International School for Advanced Studies\\
Via Bonomea 256, 31014 Trieste, Italy}

\email{matteo.tommasini2@gmail.com}

\date{August 2010}



\thanks{I would like to acknowledge my advisors Barbara Fantechi and Emilia Mezzetti, for suggesting to me the
topic and for all the help they gave to me during this work.
I would also like to acknowledge Dorette Pronk for her help during the work for the thesis and for this article.\\
This work was supported by the International School for Advanced Studies (SISSA-ISAS) in Trieste.}

\begin{abstract}
We define a 2-category structure $\CAT{Pre-Orb}$ on the category of reduced complex orbifold atlases. We construct a
2-functor $F$ from $\CAT{Pre-Orb}$ to the 2-category $\CAT{Grp}$ of proper \'etale effective groupoid objects over the
complex manifolds. Both on $\CAT{Pre-Orb}$ and $\CAT{Grp}$ there are natural equivalence relations on objects: (a natural
extension of) equivalence of orbifold atlases on $\CAT{Pre-Orb}$ and Morita equivalence in $\CAT{Grp}$. We prove that $F$
induces a bijection between the equivalence classes of its source and target.
\end{abstract}

\maketitle

\tableofcontents

\section*{Introduction}

A well known issue in mathematics is that of modeling geometric objects where points have non-trivial groups of
automorphisms (for example, the set of all Riemann surfaces of a given genus $g\geq 2$,
together with their isomorphisms).
In topology and differential geometry the standard approach to these objects (when the groups associated to every point
are finite) is through orbifolds. This concept was
formalized for the first time by I. Satake in 1956 in \cite{Sa} with some different hypothesis than the current
ones, although the informal idea dates back at least to H. Poincar\'e (for example, see \cite{Poi}, 1882). The name
``orbifold'' was introduced for the first time by W. Thurston (see \cite{Thu}, 1979), and comes from the
contraction of the words ``orbit'' and ``manifold''. Indeed the first known examples of these objects where
obtained as quotients of manifolds via the action of finite groups of automorphisms; analogous to the definition
of manifold, a complex orbifold atlas is locally modeled on open subsets of $\mathbb{C}^n$ modulo finite groups of 
biholomorphisms acting
on it, such that a suitable condition of compatibility is satisfied in the intersection of any pair of ``charts''.
There is a well defined notion of ``map'' between orbifolds (for example one can adapt the definition of the appendix
of \cite{CR} from the real to the complex case) and composition of them, so orbifolds form a category.\\

On the other hand, in algebraic complex geometry, objects that have non-trivial group of automorphisms arise frequently
from moduli problems and are usually studied as (Deligne-Munford) algebraic stacks, that form in a natural way
 a non-trivial 2-category. A third approach, intermediate between the previous two, is the one that
uses smooth groupoid objects, which also form a 2-category; for an introduction to these objects, see for example 
\cite{M}. There exist strong relations between groupoid objects over the category of schemes and algebraic
stacks, as shown in the appendix of \cite{Vis} by A. Vistoli (1989); analogous strong relations were found by
D. Pronk in \cite{Pr} and \cite{Pr2} for the topological and differentiable case.\\

Moreover, there is a very good reason to
think of orbifolds as groupoid objects (at least in the smooth case) because of a construction due to D. Pronk
(\cite{Pr}) that allows to associate to every smooth reduced orbifold atlas a proper \'etale and effective
groupoid object over smooth manifolds. Hence it seems natural to try to give a richer structure to complex
orbifolds, i.e. to make them into a non trivial 2-category and to find if there exists any relation between
morphisms and 2-morphisms for orbifolds and the corresponding ones for groupoid objects over complex manifolds.\\

For simplicity in this work we will always restrict
our attention to the complex case and to effective actions, i.e. all the orbifold atlases will be reduced and all
the groupoid objects will be effective (see remark \ref{reduced-orb} and definition \ref{effective-groupoid}).
This article is divided in 4 section as follows:

\begin{enumerate}[(1)]
\item we review the basic facts about 2 categories and objects and morphisms in the category of reduced complex
orbifold atlases; then we prove that this category can be embedded into a non-trivial 2-category, called
$\CAT{Pre-Orb}$ by defining suitable 2-morphisms (that will be called natural transformations), 2-identities and
vertical and horizontal compositions of them. In addition, we review the definition of equivalence of orbifold atlases,
analogous to the notion of compatibility between manifold atlases;

\item a straightforward calculation proves that there exists a natural 2-category $\CAT{Grp}$ where the objects are
proper \'etale and effective groupoid objects over complex manifolds and morphisms and 2-morphisms are the usual
morphisms and 2-morphisms between any pair of groupoid objects;

\item the construction due to D.Pronk can be adapted from the smooth to the complex case in order to associate to
every object of $\CAT{Pre-Orb}$ an object of $\CAT{Grp}$. Moreover, we prove that this construction can be extended to
morphisms and 2-morphisms in order to get a 2-functor $F:\CAT{Pre-Orb}\rightarrow\CAT{Grp}$;

\item in this section we review briefly the notion of Morita equivalence on groupoid objects. Then we
prove that $F$ induces a bijection between classes of orbifold atlases (as described in \S 1) and classes of
Morita equivalent groupoid objects. Moreover, it is implicitly proved that if one considers
the 2-functor $U$ (described in \cite{Pr2}) of localization (up to Morita equivalences) of $\CAT{Grp}$, we get that
$U\circ F$ is essentially surjective.
\end{enumerate}

Note that we don't call the first 2-category $\CAT{Orb}$ or $\CAT{Orbifolds}$ because its objects would have been
orbifolds (i.e. equivalence classes of orbifold atlases, see \S \ref{section-equivalences}). Indeed, there remains
the following 2 open problems:

\begin{enumerate}[(a)]
 \item What is the natural extension of the equivalence in \S \ref{section-equivalences} to morphisms and
2-morphisms of orbifolds in order to describe a 2-category (or a bicategory) $\CAT{Orb}$? Is it possible to use the
calculus of fractions of D. Pronk in order to formally invert the (functorial) refinements of atlases?

\item Is this extension compatible with $F$? In other words, consider the bicategory $\CAT{Grp}[W^{-1}]$ obtained
by inverting the class $W$ of Morita equivalences using the calculus of fractions. Then suppose that
(a) is solved; is it possible to use the results of \S 4 in order to induce a 2-functor $\tilde{F}$ from
$\CAT{Orb}$ to $\CAT{Grp}[W^{-1}]$?
\end{enumerate}

\section{The 2-category of complex reduced orbifolds}

\subsection{2-categories and 2-functors}
We assume the standard notions of categories, (covariant) functors, fiber products in a fixed category and
natural transformations (see, for example \cite{B}). In this work we will also use the notions
of 2-categories and 2 functors, that we recall briefly:

\begin{defin}\label{2-cat}
(\cite{B}, def. 7.1.1) A \emph{$2$-category} $\ca{A}$ consists of the following data:

\begin{enumerate}[(1)]\parindent=0pt
\item a \emph{class} $\ca{A}_0$, whose elements are called objects;

\item for every pair of objects $A,B$, a \emph{small category} $\ca{A}(A,B)$ (i.e. its objects form a set and not
a class); the objects of this category are called morphisms or 1-morphisms and will be denoted by $f:A\rightarrow
B$. The morphisms of this category between any pair of 1-morphisms $f$ and $g$ are called 2-morphisms and are
denoted by $\alpha:f\Rightarrow g$. The composition of 2 composable morphisms
 $\alpha,\beta$ in the category $\ca{A}(A,B)$ will be called \emph{vertical} composition and
denoted with $\beta\odot\alpha$;

\item for each triple $A,B,C$ of objects of $\ca{A}$, a functor:

$$c_{ABC}:\ca{A}(A,B)\times\ca{A}(B,C)\rightarrow\ca{A}(A,C);$$

the composition  $c_{ABC}(f,g)$ of two objects $f:A\rightarrow B$ and $g:B\rightarrow C$ will be denoted by $g
\circ f$. The composition $c_{A,B,C}(\alpha,\beta)$ of two morphisms $\alpha:f\Rightarrow f'$ in $\ca{A}(A,B)$ and
$\beta:g\Rightarrow g'$ in $\ca{A}(B,C)$ will be called \emph{horizontal} composition and denoted by $\beta\ast
\alpha$;

\item for each object $A$ of $\ca{A}$, a morphism $1_A:A\rightarrow A$ and a 2-morphism $i_A:1_A\Rightarrow 1_A$.
\end{enumerate}

We require that these data satify the following axioms (which are not the original axioms of \cite{B}, but are
equivalent to them):

\begin{enumerate}[(a)]\parindent=0pt
\item for every triple of 1-morphisms of the form $A\stackrel{f}{\longrightarrow}B\stackrel{g}{\longrightarrow}C
\stackrel{h}{\longrightarrow}D$ we have $(h\circ g)\circ f=h\circ(g\circ f);$

\item for every diagram:

\[\begin{tikzpicture}[scale=0.8]
    \def\x{1.5}
    \def\y{-1.2}
    \node (A0_0) at (0*\x, 0*\y) {$A$};
    \node (A0_1) at (1.1*\x, 0*\y) {$\Downarrow\alpha$};
    \node (A0_2) at (2*\x, 0*\y) {$B$}; 
    \node (A0_3) at (3.1*\x, 0*\y) {$\Downarrow\beta$};
    \node (A0_4) at (4*\x, 0*\y) {$C$};
    \node (A0_5) at (5.1*\x, 0*\y) {$\Downarrow\gamma$};
    \node (A0_6) at (6*\x, 0*\y) {$D$};
    \path (A0_0) edge [->,bend left=25] node [auto] {$\scriptstyle{f}$} (A0_2);
    \path (A0_0) edge [->,bend right=25] node [auto,swap] {$\scriptstyle{f'}$} (A0_2);
    \path (A0_2) edge [->,bend left=25] node [auto] {$\scriptstyle{g}$} (A0_4);
    \path (A0_2) edge [->,bend right=25] node [auto,swap] {$\scriptstyle{g'}$} (A0_4);
    \path (A0_4) edge [->,bend left=25] node [auto] {$\scriptstyle{h}$} (A0_6);
    \path (A0_4) edge [->,bend right=25] node [auto,swap] {$\scriptstyle{h'}$} (A0_6);
\end{tikzpicture}\]

we have $(\gamma\ast\beta)\ast\alpha=\gamma\ast(\beta\ast\alpha);$

\item for each 1-morphism $A\fre{f}B$, we have $f\circ 1_A=f=1_B\circ f;$

\item for each 2-morphism $\alpha:(f:A\rightarrow B)\Rightarrow (g:A\rightarrow B)$ we require that $\alpha\ast
i_A=\alpha=i_B\ast\alpha.$
\end{enumerate}
\end{defin}

\begin{rem}\label{interchage}
Let us consider the following diagram in a 2-category $\ca{A}$:

\[\begin{tikzpicture}[scale=0.8]
    \def\x{2.0}
    \def\y{-1.2}
    \node (A0_0) at (0*\x, 0*\y) {$A$};
    \node (A0_1) at (1.1*\x, 0.6) {$\Downarrow\alpha$};
    \node (A0_1) at (1.1*\x, -0.8) {$\Downarrow\beta$};
    \node (A0_2) at (2*\x, 0*\y) {$B$};
    \node (A0_3) at (3.1*\x, 0.6) {$\Downarrow\gamma$};
    \node (A0_3) at (3.1*\x, -0.8) {$\Downarrow\delta$};
    \node (A0_4) at (4*\x, 0*\y) {$C;$};

    \path (A0_0) edge [->,bend left=50] node [auto] {$\scriptstyle{f}$} (A0_2);
    \path (A0_0) edge [->] node [auto,swap] {$\scriptstyle{f'}$} (A0_2);
    \path (A0_0) edge [->,bend right=50] node [auto,swap] {$\scriptstyle{f''}$} (A0_2);
    \path (A0_2) edge [->,bend left=50] node [auto] {$\scriptstyle{g}$} (A0_4);
    \path (A0_2) edge [->] node [auto,swap] {$\scriptstyle{g'}$} (A0_4);
    \path (A0_2) edge [->,bend right=50] node [auto,swap] {$\scriptstyle{g''}$} (A0_4);
\end{tikzpicture}\]

then we get that:

$$(\delta\ast\beta)\odot(\gamma\ast\alpha)=c_{ABC}(\beta,\delta)\odot c_{ABC}(\alpha,\gamma)=$$
$$=c_{ABC}((\beta,\delta)\odot(\alpha,\gamma))=c_{ABC}(\beta\odot\alpha,\delta\odot\gamma)=(\beta\odot\alpha)\ast
(\delta\odot\gamma).$$

This formula is known as \emph{interchagenge law}(see \cite{B}, proposition 1.3.5).
\end{rem}

For some basic examples of 2-categories, see \cite{B}, example 7.1.4 (the most simple example is the 2-category of
small categories, functors and natural transformations between them).

\begin{defin}\label{2-func}(equivalent to \cite{B}, def. 7.2.1)
Given two 2-categories $\ca{A}$ and $\ca{B}$, a \emph{(covariant) 2-functor} 
$F:\ca{A}\rightarrow\ca{B}$ consists of the following data:

\begin{enumerate}[(1)]\parindent=0pt
\item for each object $A$ in $\ca{A}$, an object $F(A)$ in $\ca{B}$;

\item for each pair of objects $A,A'$, a functor $F_{A,A'}:\ca{A}(A,A')\rightarrow\ca{B}(F(A),F(A'))$; with a
little abuse of notation, sometimes we will denote this functor only with $F$. These data must satisfy the
following axioms:
\end{enumerate}

\begin{enumerate}[(a)]\parindent=0pt
\item for every pair of morphisms $f:A\rightarrow A'$ and $g:A'\rightarrow A''$ we have 
$F(g\circ f)=F(g)\circ F(f)$;

\item for every diagram in $\ca{A}$ of the form:

\[\begin{tikzpicture}[scale=0.8]
    \def\x{1.5}
    \def\y{-1.2}
    \node (A0_0) at (0*\x, 0*\y) {$A$};
    \node (A0_1) at (1*\x, 0*\y) {$\Downarrow\alpha$};
    \node (A0_2) at (2*\x, 0*\y) {$A'$};
    \node (A0_3) at (3*\x, 0*\y) {$\Downarrow\beta$};
    \node (A0_4) at (4*\x, 0*\y) {$A''$};
    \path (A0_2) edge [->,bend left=25] node [auto] {$\scriptstyle{g}$} (A0_4);
    \path (A0_2) edge [->,bend right=25] node [auto,swap] {$\scriptstyle{g'}$} (A0_4);
    \path (A0_0) edge [->,bend left=25] node [auto] {$\scriptstyle{f}$} (A0_2);
    \path (A0_0) edge [->,bend right=25] node [auto,swap] {$\scriptstyle{f'}$} (A0_2);
\end{tikzpicture}\]

we have that $F(\beta\ast\alpha)=F(\beta)\ast F(\alpha)$;
 
\item for every object $A$ of $\ca{A}$ we require that $F(1_A)=1_{F(A)}$ and $F(i_A)=i_{F(A)}$.
\end{enumerate}
\end{defin}

\subsection{Uniformizing systems, embeddings and atlases}
Let us review some basic definitions about complex orbifolds. All the definitions of this section are just
translations to the complex case of the corresponding definitions for the smooth case (see, for example, the
appendix of \cite{CR}). Since we will work only in the holomorphic case, in general we will use the word
``orbifold'' instead of ``complex orbifold''.

\begin{defin}\label{unif-sys}
Let $X$ be a paracompact second countable Hausdorff topological space and let $U\subseteq X$ be open and
non-empty. Then a \emph{(complex) uniformizing system} (also known as \emph{orbifold chart}) \emph{of dimension}
$n$ for $U$ is the datum of:

\begin{itemize}
\item a \emph{connected} and non-empty open set $\tU\subseteq\mathbb{C}^n$;

\item a \emph{finite} group $G$ of holomorphic automorphisms of $\tU$;

\item a continuous, surjective and $G$-invariant map $\pi:\tU\rightarrow U$, which induces an homeomorphism
between $\tU/G$ and $U$, where we give to $\tU/G$ the quotient topology.
\end{itemize}
\end{defin}

\begin{rem}\label{reduced-orb}
In this work we will always assume that $G$ \emph{is a set of maps, which is also a group}; the
orbifolds which have this property are usually called \emph{reduced} or \emph{effective} (the precise definitions
of orbifold and orbifold atlas will be given in the following pages). Some authors don't use this restriction: in
this case $G$ is a priori a group together with a representation $\psi: G\rightarrow 
\textrm{Aut}(\tU)$ which is not necessarily \emph{faithful} (i.e. injective).
\end{rem}

\begin{rem}\label{unif-sys-alernat-1}
Some articles (see, for example, \cite{LU}, def.  2.1.1) don't require that $\tU$ is a connected \emph{open set} of
$\mathbb{C}^n$, but only that it is a connected complex \emph{manifold}, while all the other properties are exactly
the same. This is very useful in order to define orbifold atlases for global quotients, but this definition is
not equivalent to the previous one; however, it is not so difficult to prove that the equivalence classes of
orbifold atlases (as described in \S \ref{section-equivalences}) will be the same with this alternative definition.
We prefer to use the previous definition just because it is more common in literature.
\end{rem}

\begin{defin}
Let $(\tU,G,\pi)$ be a uniformizing system and let $\tx\in\tU$. Then we define the \emph{isotropy subgroup} (also
known as \emph{stabilizer group}) at $\tx$ as the subgroup of $G$:

$$G_{\tx}:=\{g\in G\textrm{ s.t. } g(\tx)=\tx\}.$$
\end{defin}

\begin{lem}\label{radius-lemma}
Let $(\tU,G,\pi)$ be a uniformizing system, let $\tx\in\tU$ and $g\in G\smallsetminus G_{\tx}$. Then there exists a
positive radius $r=r(\tx,g)$ such that if we call $B_r$ the open ball with radius $r$ and centered in $\tx$, we
have:

$$g(B_r)\cap B_r=\varnothing.$$

\end{lem}

\begin{defin}\label{complex-embedding}
Let us fix two uniformizing systems $(\tU,G,\pi)$ and $(\tV,H,\phi)$ for open sets $U,V$ in $X$ with $U\subseteq
V$. Then a \emph{(complex) embedding} $\lambda$ from the first uniformizing system to the second one is given by
an holomorphic embedding $\lambda:\tU\rightarrow\tV$ such that $\phi\circ\lambda=\pi$.
\end{defin}

\begin{lem}\label{unif-sys-induced}
Let us fix a uniformizing system $(\tU,G,\pi)$ and any point $\tx$ in $\tU$, together with an open neighborhood
$\tA$ of it in $\tU$. Then there exist a uniformizing system of the form $(\tU',G',\pi')$ and an embedding
$\lambda$ of it into the previous one, such that:

\begin{itemize}
\item $\tU'$ is an open connected neighborhood of $\tx$, completely contained in $\tA$;

\item $\lambda$ is just the inclusion map of $\tU'$ in $\tU$;

\item $G'$ is the set of elements of the isotropy subgroup $G_{\tx}$, restricted to $\tU'$.

\item up to a biholomorphic change of coordinates all the elements of $G'$ act linearily on $\tU'$.
\end{itemize}
\end{lem}

The proof of this lemma uses lemma \ref{radius-lemma} for every $g$ in the finite set $G\smallsetminus G_{\tx}$, 
together with Cartan's linearization lemma (see \cite{Ca}, lemma 1) in order to find the open neighborhood $\tU'$ of 
$\tx$ in $\tU$; then it suffices to use the definition of quotient topology for the open set $U$ homeomorphic to 
$\tU/G$.

\begin{defin}\label{trivial-stabilizer}
Let $(\tU,G,\pi)$ be a uniformizing system for an open set $U$ of $X$; then for every $g\in G\smallsetminus
\{1_{\tU}\}$ we define the sets $\tU_g:=\{\tx\in\tU\textrm{ s.t. } g(\tx)\neq \tx \}$ and $\tU_G:=\bigcap_{g\in G
\smallsetminus\{1_{\tU}\}}\tU_g$ i.e. the set of points of $\tU$ with trivial stabilizer.
\end{defin}

Now  if we fix any $g\in G\smallsetminus\{1_{\tU}\}$, we get that the map $g-1_{\tU}:\tU\rightarrow
\mathbb{C}^n$ is continuous, and the set $\tU_g$ is the preimage via this function of the open set $\mathbb{C}^n
\smallsetminus\{0\}$, hence $\tU_g$ is open in $\tU$. Moreover, it is dense in $\tU$, indeed if it was not dense, this
would imply that there exists an open subset where $g=id_{\tU}$; since $g$ is holomorphic, this implies that $g$ is
the identity on all $\tU$, contradiction. Now every locally compact Hausdorff space is a Baire space, hence $\tU$
is a Baire space, so every countable intersection of open dense subsets of it is still dense. In particular,
$\tU_G$ is a \emph{finite} intersection of open dense sets, so we get:

\begin{lem}\label{density}
The set $\tU_G$ is open and dense in $\tU$.
\end{lem}

The following is a very useful technical result. It was proved for the first time by I. Satake (\cite{Sa2}) with an
extra assumption, and by I. Moerdijk and D. Pronk (\cite{MP}, appendix, proposition A.1) in the general case for
smooth orbifolds. The following is an analogous result proved with the same technique in the case
of complex orbifolds.

\begin{lem}\label{lemma-moerdijk}
Let $\lambda$ and $\mu$ be two embeddings: $(\tU,G,\pi)\rightarrow (\tV,H,\phi)$ between uniformizing systems of
the same dimension $n$. Then there exists a unique $h\in H$ such that $\mu=h\circ\lambda$.
\end{lem}

As a consequence of lemma \ref{lemma-moerdijk} we have the following corollary, whose proof is analogous to the one
given in the smooth case in \cite{ALR}, \S 1.1.

\begin{cor}\label{induced-map}
Any embedding $\lambda :(\tU,G,\pi)\rightarrow(\tV,H,\phi)$ induces an injective group homomorphism $\Lambda:G
\rightarrow H$ such that $\lambda\circ g=\Lambda(g)\circ \lambda$ for all $g\in G$.
\end{cor}

\begin{rem}
When we don't assume that the orbifolds we use are reduced (see, for example, \cite{LU}, def. 2.1.2)
one has to define an embedding from $(\tU,G.\pi)$ to $(\tV,H,\phi)$ as a pair $(\lambda,\Lambda)$ where:

\begin{itemize}
\item $\lambda:\tU\rightarrow\tV$ is an holomorphic embedding such that $\phi\circ\lambda=\pi$;

\item $\Lambda:G\rightarrow H$ is an injective group homomorphism, such that for all $g\in G$ we have $\lambda
\circ g=\Lambda(g)\circ\lambda$.
\end{itemize}

The first condition is just definition \ref{complex-embedding} and the second one is just the previous corollary,
so is not necessary in the reduced case. However, \emph{this is true only if we use reduced orbifolds because
lemma} \ref{lemma-moerdijk}\emph{ only applies in this case}. Indeed, the proof of the lemma consists in defining
a \emph{unique} element $h$ in $H$ such that $\mu=h\circ\lambda$ holds, but this is an identity between holomorphic
functions; so if the map $\psi$ defined in remark \ref{reduced-orb} is not injective, the existence part of the
lemma is still true, but in general we can't prove uniqueness, so also the corollary is no more true. In this work
we will only use reduced orbifolds, so we don't bother about this problem.
\end{rem}

\begin{lem}\label{lemma-intersection}
Let $\lambda:(\tU,G,\pi)\rightarrow(\tV,H,\phi)$ be an embedding and let $h\in H$. If $h(\lambda(\tU))\cap\lambda
(\tU)\neq\varnothing$, then $h(\lambda(\tU))=\lambda(\tU)$ and $h$ belongs to the image of the induced injective
group homomorphism $\Lambda:G\rightarrow H$.
\end{lem}

This is proved for the smooth case, in \cite{MP}, appendix, lemma A.2, but the proof works also in the holomorphic
case, with some small changes, so we omit it.

\begin{defin}\label{orb-atlas}
Let $X$ be a paracompact and second countable Hausdorff topological space; a (\emph{complex})\emph{ reduced orbifold
atlas of dimension }$n$ on $X$ is a family $\mathcal{U}=\{(\tU_i,G_i,\pi_i)\}_{i\in I}$ of reduced uniformizing
systems of dimension $n$, such that:

\begin{enumerate}[(i)]
\item the family $\{\pi_i(\tU_i)\}_{i\in I}$ is an open cover of $X$;

\item if $\unif{i}$, $\unif{j}\in\mathcal{U}$ are uniformizing systems for $U_i$ and $U_j$ respectively, then for
every point $x\in U_i\cap U_j$ there exists an open neighborhood $U_k \subseteq U_i\cap U_j$ of $x$ in $X$, a
uniformizing system $\unif{k}\in\mathcal{U}$ for $U_k$ and embeddings:

\begin{equation}\label{eq-2}
\unif{i}\sinistra{\LA{ki}}\unif{k}\fre{\LA{kj}}\unif{j}. 
\end{equation}

\end{enumerate}
\end{defin}

\begin{rem}\label{orb-atlas-2}
To be more precise, an orbifold atlas is the datum of a family $\mathcal{U}$ of uniformizing systems that satisfy
(i) and (ii), \emph{together with the family of all possible embeddings between charts of }$\mathcal{U}$, so the
set of embeddings is uniquely determined by the set of uniformizing systems. Hence with a little abuse of notation
we will always write $\mathcal{U}=\{\unif{i}\}_{i\in I}$ to denote both the family of uniformizing systems and the
family of embeddings between them. So every atlas can be considered as a category, with objects given by its 
uniformizing systems and morphisms given by embeddings between them.
\end{rem}

\begin{rem}\label{useful-remark}
Let us fix any pair of uniformizing systems $\unif{i},\unif{j}\in\mathcal{U}$ and a pair of points
$\tx_i\in\tU_i$ and $\tx_j\in\tU_j$ such that $\pi_i(\tx_i)=\pi_j(\tx_j)$. Then by using (\ref{eq-2}) and
eventually by composing $\LA{ki}$ and $\LA{kj}$ with elements of $G_i$ and $G_j$ respectively, without loss
of generality we can assume that we have fixed a point $\tx_k\in\tU_k$ such that the following diagram of sets and
marked points is commutative:

\[\begin{tikzpicture}[scale=0.8]
    \def\x{1.5}
    \def\y{-1.2}
    \node (A0_0) at (0*\x, 1*\y) {$\tx_i$};
    \node (A0_2) at (2*\x, 1*\y) {$\tx_k$};
    \node (A0_4) at (4*\x, 1*\y) {$\tx_j$};
    \node (A1_0) [rotate=270] at (0*\x, 1.5*\y) {$\in$};
    \node (A1_2) [rotate=270] at (2*\x, 1.5*\y) {$\in$};
    \node (A1_4) [rotate=270] at (4*\x, 1.5*\y) {$\in$};
    \node (A2_0) at (0*\x, 2*\y) {$\tU_i$};
    \node (A2_2) at (2*\x, 2*\y) {$\tU_k$};
    \node (A2_4) at (4*\x, 2*\y) {$\tU_j$};
    \node (A3_1) at (1*\x, 3*\y) {$\curvearrowright$};
    \node (A3_3) at (3*\x, 3*\y) {$\curvearrowright$};
    \node (A4_0) at (0*\x, 4*\y) {$U_i$};
    \node (A4_2) at (2*\x, 4*\y) {$U_k$};
    \node (A4_4) at (4*\x, 4*\y) {$U_j.$};
    \node (A5_2) [rotate=90] at (2*\x, 4.5*\y) {$\in$};
    \node (A6_2) at (2*\x, 5*\y) {$x$};
    \path (A4_2) edge [right hook->] node [auto] {$\scriptstyle{}$} (A4_4);
    \path (A2_2) edge [->] node [auto,swap] {$\scriptstyle{\LA{ki}}$} (A2_0);
    \path (A2_2) edge [->] node [auto] {$\scriptstyle{\LA{kj}}$} (A2_4);
    \path (A2_0) edge [->] node [auto,swap] {$\scriptstyle{\pi_i}$} (A4_0);
    \path (A4_2) edge [right hook->] node [auto,swap] {$\scriptstyle{}$} (A4_0);
    \path (A2_4) edge [->] node [auto] {$\scriptstyle{\pi_j}$} (A4_4);
    \path (A2_2) edge [->] node [auto] {$\scriptstyle{\pi_k}$} (A4_2);
\end{tikzpicture}\]
\end{rem}

\begin{rem}\label{local-group}
Let us fix any point $x\in X$, let us choose $\unif{i}$ and $\unif{j}$ together with $\tx_i\in\tU_i$, $\tx_j\in
\tU_j$ such that $\pi_i(\tx_i)=x=\pi_j(\tx_j)$. Then using remark \ref{useful-remark} and corollary \ref{induced-map}
we get injective group homomorphisms $\Lambda_{ki}:G_k\rightarrow G_i$ and $\Lambda_{kj}:G_k\rightarrow G_j$. It is
easy to see that for every $g\in (G_k)_{\tx_k}$ we have $\Lambda_{ki}(g)\in (G_i)_{\tx_i}$, moreover, using lemma
\ref{lemma-intersection} we get that for every $g\in(G_i)_{\tx}$ there exists a unique $g'\in (G_k)_{\tx_k}$ such
that $\Lambda_{ki}(g')=g$. Hence we have proved that $\Lambda_{ki}$ restricts to a group isomorphism from 
$(G_k)_{\tx_k}$ to $(G_i)_{\tx_i}$; analogously we get a group isomorphism between $(G_k)_{\tx_k}$ and 
$(G_j)_{\tx_j}$, so $(G_i)_{\tx_i}\simeq(G_j)_{\tx_j}$. Hence one can give a notion of local group at a point
of $X$, \emph{well defined up to isomorphisms}.
\end{rem}

\begin{rem}
Some articles (see, for example, \cite{LU}, \S2) use a slight different definition of orbifold atlas that can
be restated as follows: an atlas is a family that satisfies condition (i) of our definition, together with
condition (ii) \emph{without} the assumption that $\unif{k}$ belongs to the atlas $\mathcal{U}$.
We preferred to use definition \ref{orb-atlas} instead of this one because it is more common in literature
and because in this way we can construct more easily a groupoid object associated to every orbifold atlas,
as described in \S 3.1. Clearly this alternative definition is weaker than the previous one, but it is
easy to see that any orbifold atlas
with respect to this alternative definition is equivalent (using \S \ref{section-equivalences}) to an orbifold
atlas with respect to our definition.
\end{rem}

\subsection{Local liftings and compatible systems}
Now our aim is to make orbifold atlases into a category, i.e. we want to define what a morphism between orbifolds atlases is. In
order to do that, we have first of all to define a continuous map between the underlying topological spaces, but
differently from the case of morphisms between manifolds, this will not be sufficent in general. The idea to keep
in mind in the following definitions is that a morphism between orbifolds is essentially a continuous function
which can be \emph{locally lifted to a holomorphic} function between uniformizing systems in source and target.

\begin{rem}
\emph{The following definition of morphism (compatible system) between orbifold atlases is} almost \emph{the usual
one of good/strong map between orbifolds (see, for example, \cite{CR}, \S 4.1) with one important difference.
Indeed it is quite evident from the following definitions that after passing to equivalence classes of atlases
(see \S \ref{section-equivalences}) our definition coincides with the definition of strong/good map. We are
forced to use this definition, instead of the usual one, since the objects we are dealing with are orbifold
atlases and not} classes of equivalence of \emph{orbifold atlases}.
\end{rem}

\begin{defin}
Let $\mathcal{U}$ and $\mathcal{V}$ be atlases for $X$ and $Y$ respectively and let $U\subseteq X$ and $V\subseteq
Y$ be open sets with uniformizing systems $(\tU,G,\pi)\in\mathcal{U}$ and $(\tV,H,\phi)\in\mathcal{V}$
respectively. Let $f:U\rightarrow V$ be a continuous function; then a \emph{lifting of }$f$ from $(\tU,G,\pi)$
\emph{to} $(\tV,H,\phi)$ is an holomorphic function $\tf_{\tU,\tV}:\tU\rightarrow \tV$ such that:

\begin{equation}\label{eq-3}\phi\circ\tf_{\tU,\tV} = f\circ \pi.\end{equation}  
\end{defin}

\begin{defin}\label{compatible-system}
Let $\mathcal{U}=\{\unif{i}\}_{i\in I}$ and $\mathcal{V}=\{\uniff{j}\}_{j\in J}$ be atlases (not necessarily of the
same dimension) for $X$ and $Y$ respectively and let $f:X\rightarrow Y$ be a continuous map. Then a
\emph{compatible system} for $f$ is the datum of:

\begin{enumerate}[(1)]\parindent=0pt
\item a \emph{functor} $\tf:\mathcal{U}\rightarrow\mathcal{V}$ between the associated categories (see remark 
\ref{orb-atlas-2}) such that if we call
$\uniff{i}\in\mathcal{V}$ the image of any element $\unif{i}\in\mathcal{U}$ via $\tf$, we have $f(\pi_i(\tU_i))
\subseteq\phi_i(\tV_i)$;

\item a collection $\{\tf_{\tU_i,\tV_i}\}_{\unif{i}\in\mathcal{U}}$ where for every $\unif{i}\in\mathcal{U}$ we 
have that $\tf_{\tU_i,\tV_i}$ is a lifting for the continuous function $f|_{U_i}:U_i\rightarrow f(U_i)\subseteq V_i$
from $\unif{i}$ to $\uniff{i}$;
\end{enumerate}

such that for every embedding $\LA{ij}$ from $\unif{i}$ to $\unif{j}$ in $\mathcal{U}$ we have:

\begin{equation}\label{eq-4}
\tf_{\tU_j,\tV_j}\circ \LA{ij}=\tf(\LA{ij})\circ\tf_{\tU_i,\tV_i}
\end{equation}
 
i.e. we are in the following situation:

\[\begin{tikzpicture}[scale=0.8]
    \def\x{1.9}
    \def\y{-1.4}
    \node (A0_0) at (0*\x, 0*\y) {$\tU_i$};
    \node (A0_2) at (2*\x, 0*\y) {$\tV_i$};
    \node (A1_1) at (1*\x, 1*\y) {$\tU_j$};
    \node (A1_3) at (3*\x, 1*\y) {$\tV_j$};
    \node (A2_0) at (0*\x, 2*\y) {$U_i$};
    \node (A2_2) at (2*\x, 2*\y) {$V_i$};
    \node (A3_1) at (1*\x, 3*\y) {$U_j$};
    \node (A3_3) at (3*\x, 3*\y) {$V_j$};
    \node (A2_1) at (1*\x, 2*\y) {$ $};
    \node (A2_3) at (3*\x, 2*\y) {$ $};
    \node (A1_2) at (2*\x, 1*\y) {$ $};

    \path (A0_0) edge [->] node [auto,swap] {$\scriptstyle{\LA{ij}}$} (A1_1);
    \path (A0_0) edge [->] node [auto] {$\scriptstyle{\tf_{\tU_i,\tV_i}}$} (A0_2);
    \path (A2_0) edge [right hook->] node [auto] {$\scriptstyle{}$} (A3_1);
    \path (A0_2) edge [->] node [auto] {$\scriptstyle{}$} (A2_2);
    \path (A1_2) edge [-] node [auto] {$\scriptstyle{\phi_i}$} (A2_2);
    \path (A3_1) edge [->] node [auto,swap] {$\scriptstyle{f_{|U_j}}$} (A3_3);
    \path (A2_0) edge [->] node [auto,swap] {$\scriptscriptstyle{ }$} (A2_2);
    \path (A2_1) edge [-] node [auto,swap] {$\scriptstyle{f_{|U_i}}$} (A2_2);
    \path (A1_1) edge [->] node [auto,swap] {$\scriptstyle{}$} (A3_1);
    \path (A2_2) edge [right hook->] node [auto] {$\scriptstyle{}$} (A3_3);
    \path (A0_2) edge [->] node [auto] {$\scriptstyle{\tf(\LA{ij})}$} (A1_3);
    \path (A1_1) edge [->] node [auto] {$\scriptstyle{}$} (A1_3);
    \path (A1_1) edge [-] node [auto] {$\scriptstyle{\tf_{\tU_j,\tV_j}}$} (A1_2);
    \path (A0_0) edge [->] node [auto,swap] {$\scriptstyle{\pi_i}$} (A2_0);
    \path (A1_3) edge [->] node [auto] {$\scriptstyle{\phi_j}$} (A3_3);
    \path (A1_1) edge [-] node [auto,swap] {$\scriptstyle{\pi_j}$} (A2_1);
\end{tikzpicture}\]

where \emph{all the faces of the cube are commutative}; indeed:

\begin{itemize}
\item the lower face is commutative because we are just restricting the function $f$ from $U_j$ to $U_i$;

\item the left and right sides are commutative by definition of embedding in $\mathcal{U}$ and $\mathcal{V}$ 
respectively;

\item the front and back sides are both commutatives because of (\ref{eq-3});

\item the top side is commutative because of (\ref{eq-4}).
\end{itemize}
\end{defin}

With a little abuse of notation, we will always write $\tf:\mathcal{U}\rightarrow\mathcal{V}$ to denote a 
compatible system for $f$, i.e. with $\tf$ we will usually mean not only the functor which satisfies (1), but also
the collection of local liftings described in (2).

\begin{rem}
Note that in (1) \emph{it is sufficient to require that $\tilde{f}$ preserves compositions} (and actually, this
is the standard definition in most articles). Indeed, suppose that we have fixed any $i\in I$ and let us call
$h:=\tilde{f}(1_{\tU_i}) \in H_i$; then we have:

$$h=\tilde{f}(1_{\tU_i})=\tilde{f}(1_{\tU_i}^2)=\tilde{f}(1_{\tU_i})^2=h^2;$$

since $H_i$ is a group by hypothesis, this implies that $h=1_{\tV_i}$, i.e. $\tilde{f}$ preserves
all the identities.
\end{rem}

\begin{defin}\label{composition-compatible-sys} 
Now let us consider 3 fixed orbifolds atlases $\mathcal{U},\mathcal{V},\mathcal{W}$  for $X$, $Y$ and $Z$
respectively, together with 2 continuous functions $f:X\rightarrow Y$, $g:Y\rightarrow Z$ and compatible
systems $\tf:\mathcal{U}\rightarrow\mathcal{V}$ and $\tg:\mathcal{V}\rightarrow\mathcal{W}$. For every
uniformizing system $\unif{i}\in\mathcal{U}$, let us call:

$$\uniff{i}:=\tf\unif{i}\AND(\tW_i,K_i,\xi_i):=\tg\uniff{i}.$$

Then we define the compatible system $\tg\circ\tf$ for $g\circ f$ as the functor
$\tg\circ\tf:\mathcal{U}\rightarrow\mathcal{W}$ together with the collection of liftings:

$$\left\{(\tg\circ\tf)_{\tU_i,\tW_i}:=\tg_{\tV_i,\tW_i}\circ\tf_{\tU_i,\tV_i}\right\}_{\unif{i}\in\mathcal{U}}.$$
\end{defin}

\subsection{Natural transformations between compatible systems}
With the previous definitions we get a category, but we recall that we wanted to make orbifol atlases 
into a 2-category, so we give the following definition (a slight change of def. 1.3.6 in \cite{Pe}, which 
I think was too much restrictive for our purposes).

\begin{defin}\label{nat-tran-orb}
Let us fix atlases $\mathcal{U}$ and $\mathcal{V}$ for $X$ and $Y$ respectively and let $\tf_1,\tf_2:\mathcal{U}
\rightarrow\mathcal{V}$ be compatible systems for \emph{the same} continuous map $f: X\rightarrow Y$. For
simplicity, for every uniformizing system $\unif{i}\in\mathcal{U}$ and for every embedding $\LA{ij}$, let us call:

$$\unifff{i}{m}:=\tf_m\unif{i}\AND\lambda_{ij}^m:=\tf_m(\LA{ij})\textrm{\quad for\quad} m=1,2.$$

Then a \emph{natural transformation of compatible systems from} $\tf_1$ \emph{to} $\tf_2$ is a family:

$$\left\{\delta_{\tU_i}=\delta_{\unif{i}}:\unifff{i}{1}\rightarrow\unifff{i}{2}\right\}_{\unif{i}\in\mathcal{U}}$$

of embeddings in $\mathcal{V}$, such that:

\begin{enumerate}[(i)]\parindent=0pt
\item for every $\unif{i}\in\mathcal{U}$ we have $(\tf_2)_{\tU_i,\tV_i^2}=\delta_{\tU_i}\circ
(\tf_1)_{\tU_i,\tV_i^1}$;

\item for every embedding $\LA{ij}$ in $\mathcal{U}$ we have a commutative diagram in $\mathcal{V}$:

\begin{equation}\label{eq-6}
   \begin{tikzpicture}[scale=0.8]
    \def\x{2.0}
    \def\y{-1.2}
    \node (A0_0) at (0*\x, 0*\y) {$\unifff{i}{1}$};
    \node (A0_2) at (2*\x, 0*\y) {$\unifff{i}{2}$};
    \node (A2_2) at (2*\x, 2*\y) {$\unifff{j}{2}.$};
    \node (A2_0) at (0*\x, 2*\y) {$\unifff{j}{1}$};
    \node (A1_1) at (1*\x, 1*\y) {$\curvearrowright$};

    \path (A0_0) edge [->] node [auto] {$\scriptstyle{\delta_{\tU_i}}$} (A0_2);
    \path (A0_2) edge [->] node [auto] {$\scriptstyle{\lambda_{ij}^2}$} (A2_2);
    \path (A2_0) edge [->] node [auto,swap] {$\scriptstyle{\delta_{\tU_j}}$} (A2_2);
    \path (A0_0) edge [->] node [auto,swap] {$\scriptstyle{\lambda_{ij}^1}$} (A2_0);
   \end{tikzpicture}
\end{equation}
\end{enumerate}

Whenever we have a natural transformation as before, we will denote it as $\delta:\tf_1\Rightarrow\tf_2.$
Note that if we ignore the additional properties of the compatible systems $\tf_1$ and $\tf_2$ and we consider
them just as functors, we get that condition (ii) is just the description of a natural transformation from
the functor $\tf_1$ to the functor $\tf_2$ (so the following horizontal and vertical compositions of natural
transformations will be modeled on the corresponding constructions described, for example, in \cite{B}, \S 1.3).
\end{defin}

\begin{rem}
\emph{Condition} (ii) \emph{is not overabundant}; indeed let us consider the following diagram:

\[\begin{tikzpicture}[scale=0.8]
    \def\x{2.0}
    \def\y{-1.7}
    \node (A0_2) at (2*\x, 0.6*\y) {$\curvearrowright$};
    \node (A1_0) at (0*\x, 1*\y) {$(\tV_i^1,H_i^1,\phi_i^1)$};
    \node (A1_2) at (2*\x, 1*\y) {$(\tU_i,G_i,\pi_i)$};
    \node (A1_4) at (4*\x, 1*\y) {$(\tV_i^2,H_i^2,\phi_i^2)$};
    \node (A2_1) at (1*\x, 2*\y) {$\curvearrowright$};
    \node (A2_3) at (3*\x, 2*\y) {$\curvearrowright$};
    \node (A3_0) at (0*\x, 3*\y) {$(\tV_j^1,H_j^1,\phi_j^1)$};
    \node (A3_2) at (2*\x, 3*\y) {$(\tU_j,G_j,\pi_j)$};
    \node (A3_4) at (4*\x, 3*\y) {$(\tV_j^2,H_j^2,\phi_j^2)$};
    \node (A4_2) at (2*\x, 3.4*\y) {$\curvearrowright$};
    \path (A1_4) edge [->] node [auto] {$\scriptstyle{\lambda_{ij}^2}$} (A3_4);
    \path (A1_0) edge [->,bend left=30] node [auto] {$\scriptstyle{\delta_{\tU_i}}$} (A1_4);
    \path (A3_2) edge [->] node [auto] {$\scriptstyle{(\tf_2)_{\tU_j,\tV_j^2}}$} (A3_4);
    \path (A3_2) edge [->] node [auto,swap] {$\scriptstyle{(\tf_1)_{\tU_j,\tV_j^1}}$} (A3_0);
    \path (A3_0) edge [->,bend right=30] node [auto,swap] {$\scriptstyle{\delta_{\tU_j}}$} (A3_4);
    \path (A1_0) edge [->] node [auto,swap] {$\scriptstyle{\lambda_{ij}^1}$} (A3_0);
    \path (A1_2) edge [->] node [auto,swap] {$\scriptstyle{(\tf_2)_{\tU_i,\tV_i^2}}$} (A1_4);
    \path (A1_2) edge [->] node [auto] {$\scriptstyle{\lambda_{ij}}$} (A3_2);
    \path (A1_2) edge [->] node [auto] {$\scriptstyle{(\tf_1)_{\tU_i,\tV_i^1}}$} (A1_0);
\end{tikzpicture}\]

where the two squares are commutative because of (\ref{eq-4}) applied to $\tf_1$ and $\tf_2$ 
respectively, and the upper and lower parts are commutative because of part (i) of the previuos definition
(applied to $\tU_i$ and $\tU_j$ respectively). Then a diagram chase only proves that
$\lambda_{ij}^2\circ\delta_{\tU_i}=\delta_{\tU_j}\circ\lambda_{ij}^1$ on the set 
$(\tf_1)_{\tU_i,\tV_i^1}(\tU_i)$, which in general is neither open nor dense in the whole  $\tV_i$.
\end{rem}

\begin{defin}
Let us fix orbifold atlases $\mathcal{U}$ for $X$ and $\mathcal{V}$ for $Y$, a continuous map $f:X\rightarrow Y$, 
3 compatible systems $\tf_m:\mathcal{U}\rightarrow\mathcal{V}$ for $m=1,2,3$ and natural transformations 
$\delta:\tf_1\Rightarrow\tf_2$ and $\sigma:\tf_2\Rightarrow\tf_3$. Then we define the \emph{vertical composition} 
$\sigma\odot\delta:\tf_1\Rightarrow\tf_3$ as follows: for any $\unif{i}\in\mathcal{U}$ we set:

$$(\sigma\odot\delta)_{\tU_i}:=\sigma_{\tU_i}\circ\delta_{\tU_i}$$

which is clearly an embedding in $\mathcal{V}$ between the images of $\unif{i}$ via $\tf_1$ and $\tf_3$ 
respectively. Moreover, a direct check proves that properties (i) and (ii) of definition \ref{nat-tran-orb} are
satisfied, hence $\sigma\odot\delta$ is actually a natural transformation from $\tf_1$ to $\tf_3$.
\end{defin}

\begin{defin}
For every uniformizing system $\tf:\mathcal{U}\rightarrow\mathcal{V}$, we define 
the natural transformation $i_{\tf}$ as follows: for any uniformizing system $\unif{i}\in\mathcal{U}$ we set as
usual $(\tV_i,H_i,\phi_i):=\tf(\tU_i,G_i,\pi_i)$ and we define $(i_{\tf})_{\tU_i}:=1_{\tV_i}$. Clearly $i_{\tf}$ is
a natural transformation from $\tf$ to itself; moreover, for any  $\alpha:\tf \Rightarrow\tg$ and for any $\beta:
\th\Rightarrow\tf$ we have:

\begin{equation}\label{eq-7}
\alpha\odot i_{\tf}=\alpha\AND i_{\tf}\odot\beta=\beta. 
\end{equation}
\end{defin}

\begin{defin}
Let $\mathcal{U},\mathcal{V},\mathcal{W}$ be ordifold atlases for $X,Y$ and $Z$ respectively; let $\tf_m$ and 
$\tg_m$  be compatible systems for $f:X\rightarrow Y$ and $g:Y\rightarrow Z$ respectively, for $m=1,2$. Moreover,
assume that we have natural transformations $\delta:\tf_1\Rightarrow\tf_2$ and $\eta:\tg_1\Rightarrow\tg_2$. Then
we define a \emph{horizontal composition} $\eta\ast\delta:(\tg_1\circ\tf_1)\Rightarrow(\tg_2\circ\tf_2)$
as follows: for any $\unif{i}\in\mathcal{U}$ we set:

$$(\eta\ast\delta)_{\tU_i}:=\eta_{\tV_i^2}\circ\tg_1(\delta_{\tU_i}).$$

Every map of this form is actually an embedding between uniformizing systems because composition of embeddings:
indeed $\eta_{\tV_i^2}$ is so
by definition and $\tg_1(\delta_{\tU_i})$ is an embedding because the functor $\tg_1$ maps embeddings to
embeddings. Again a very simple check proves that properties (i) and (ii) of definition \ref{nat-tran-orb} are
satisfied.
\end{defin}

\subsection{The 2-category \CAT{Pre-Orb}}

\begin{prop}\label{2-cat-orb}
The definitions of orbifold atlases, compatible systems, natural transformations and compositions $\circ,\odot,
\ast$ give rise to a 2-category, that we will denote with $\CAT{Pre-Orb}$.
\end{prop}

Note that we cannot call this 2-category $\CAT{Orb}$ because we will see in \S \ref{section-equivalences} that
orbifolds are equivalence classes of orbifold atlases.

\begin{proof}
In order to construct a 2-category, we have to define data (1)-(4) and to verify axioms (a)-(d) of definition
\ref{2-cat}.

\begin{enumerate}[(1)]\parindent=0pt
\item First of all, the class of objects is just the set of all orbifold atlases for every
topological space $X$ (if any).

\item If $\mathcal{U}$ and $\mathcal{V}$ are atlases over $X$ and $Y$, we
define a (small) category $\CAT{Pre-Orb}(\mathcal{U},\mathcal{V})$ as follows: the space of objects
is the set of all compatible systems $\tf:\mathcal{U}\rightarrow
\mathcal{V}$ (if any) for all continuous maps $f:X\rightarrow Y$; for any pair of compatible systems $\tf$ and $\tg$
for $f$ and $g$ respectively, we define:

$$\CAT{Pre-Orb}(\mathcal{U},\mathcal{V})(\tf,\tg):=
\left\{\begin{array}{c c}\textrm{natural transformations }
  \tf\Rightarrow\tg & \textrm{if } f=g\\ 
  \varnothing       & \textrm{else.} 
\end{array}\right.$$

The vertical composition $\odot$ is clearly associative; moreover using (\ref{eq-7}) we get that the identity over
any object $\tf$ is just $i_{\tf}$.

\item For every triple $\mathcal{U},\mathcal{V},\mathcal{W}$ of objects, we define the functor ``composition'':

$$c_{\mathcal{U},\mathcal{V},\mathcal{W}}:\CAT{Pre-Orb}(\mathcal{U},\mathcal{V})\times\CAT{Pre-Orb}(\mathcal{V},
\mathcal{W})\rightarrow\CAT{Pre-Orb}(\mathcal{U},\mathcal{W}):$$

\begin{itemize}
\item for every $\tf:\mathcal{U}\rightarrow\mathcal{V}$ and $\tg:\mathcal{V}\rightarrow\mathcal{V}$ we set
$c_{\mathcal{U},\mathcal{V},\mathcal{W}}(\tf,\tg):=\tg\circ \tf$;

\item for every $\delta:\tf_1\Rightarrow\tf_2$ in $\CAT{Pre-Orb}(\mathcal{U},\mathcal{V})$ and for every $\eta:
\tg_1\Rightarrow\tg_2$ in $\CAT{Pre-Orb}(\mathcal{V},\mathcal{W})$ we set:

\begin{equation}\label{eq-8}
c_{\mathcal{U},\mathcal{V},\mathcal{W}}(\delta,\eta):=\eta\ast\delta:(\tg_1\circ\tf_1)\Rightarrow(\tg_2\circ
\tf_2). \end{equation}

\end{itemize}

We want to prove that $c_{\mathcal{U},\mathcal{V},\mathcal{W}}$ is a functor. It is easy to see that it
preserves identities, so let us only prove that it preserves compositions. For every diagram of the form:

\[\begin{tikzpicture}[scale=0.8]
    \def\x{2.0}
    \def\y{-1.2}
    \node (A0_0) at (0*\x, 0*\y) {$\mathcal{U}$};
    \node (A0_1) at (1.1*\x, 0.6) {$\Downarrow\delta$};
    \node (A0_1) at (1.1*\x, -0.8) {$\Downarrow\sigma$};
    \node (A0_2) at (2*\x, 0*\y) {$\mathcal{V}$};
    \node (A0_3) at (3.1*\x, 0.6) {$\Downarrow\eta$};
    \node (A0_3) at (3.1*\x, -0.8) {$\Downarrow\mu$};
    \node (A0_4) at (4*\x, 0*\y) {$\mathcal{W}$};

    \path (A0_0) edge [->,bend left=50] node [auto] {$\scriptstyle{\tf_1}$} (A0_2);
    \path (A0_0) edge [->] node [auto,swap] {$\scriptstyle{\tf_2}$} (A0_2);
    \path (A0_0) edge [->,bend right=50] node [auto,swap] {$\scriptstyle{\tf_3}$} (A0_2);
    \path (A0_2) edge [->,bend left=50] node [auto] {$\scriptstyle{\tg_1}$} (A0_4);
    \path (A0_2) edge [->] node [auto,swap] {$\scriptstyle{\tg_2}$} (A0_4);
    \path (A0_2) edge [->,bend right=50] node [auto,swap] {$\scriptstyle{\tg_3}$} (A0_4);
\end{tikzpicture}\]

we want to prove that:

\begin{equation}\label{eq-9}
c_{\mathcal{U},\mathcal{V},\mathcal{W}}\Big((\sigma\odot\delta),(\mu\odot\eta)\Big)\stackrel{?}{=}
c_{\mathcal{U},\mathcal{V},\mathcal{W}}(\sigma,\mu)\odot c_{\mathcal{U},\mathcal{V},\mathcal{W}}(\delta,\eta). 
\end{equation}

Using (\ref{eq-8}) we get that to prove (\ref{eq-9}) is equivalent to prove that $(\mu\odot\eta)\ast(\sigma\odot\delta)
\stackrel{?}{=}(\mu\ast\sigma)\odot(\eta\ast\delta)$.
In other words, we have to prove that the \emph{interchagenge law} (see remark \ref{interchage}) is satisfied. So
let us verify that this last identity is true: for any uniformizing system $\unif{i}\in\mathcal{U}$ we have:

\begin{eqnarray*}
& \Big((\mu\odot\eta)\ast(\sigma\odot\delta)\Big)_{\tU_i}=(\mu\odot\eta)_{\tV_i^3}\circ\tg_1\Big((\sigma\odot
  \delta)_{\tU_i}\Big)= &\\
& =\mu_{\tV_i^3}\circ\eta_{\tV_i^3}\circ\tg_1(\sigma_{\tU_i})\circ\tg_1(\delta_{\tU_i})\stackrel{*}{=}
  \mu_{\tV_i^3}\circ\tg_2(\sigma_{\tU_i})\circ\eta_{\tV_i^2}\circ\tg_1(\delta_{\tU_i})= &\\
& =(\mu\ast\sigma)_{\tU_i}\odot(\eta\ast\delta)_{\tU_i}=\Big((\mu\ast\sigma)\odot(\eta\ast\delta)\Big)_{\tU_i} &
\end{eqnarray*}

where the passage denoted with $\stackrel{*}{=}$ is just (\ref{eq-6}) for the natural transformation $\eta:\tg_1
\Rightarrow\tg_2$ and for the the embedding $\lambda:=\sigma_{\tU_i}$. Hence (\ref{eq-9}) is proved, so 
$c_{\mathcal{U},\mathcal{V},\mathcal{W}}$ preserves compositions.

\item It remains to define the ``identities'' of $\CAT{Pre-Orb}$, so for every atlas $\mathcal{U}$ we define 
$1_{\mathcal{U}}:\mathcal{U}\rightarrow\mathcal{U}$ to be a compatible system over the identity on $X$, described
as the identity functor from the category associated to $\mathcal{U}$ to itself, together with the collection of
liftings for the identity map on $X$:

$$\left\{(1_{\mathcal{U}})_{\tU_i,\tU_i}:=1_{\tU_i}\right\}_{\unif{i}\in\mathcal{U}};$$

moreover, we define $i_{\mathcal{U}}$ as the natural transformation $i_{1_{\mathcal{U}}}$. So \emph{we have
defined all the data of a 2-category}; now we have only to verify axioms (a)-(d).
\end{enumerate}

\begin{enumerate}[(a)]\parindent=0pt
\item For every triple of compatible systems: $\mathcal{U}\stackrel{\tf}{\longrightarrow}\mathcal{V}\stackrel{\tg}
{\longrightarrow}\mathcal{W}\stackrel{\th}{\longrightarrow}\mathcal{Z}$ we have that $(\th\circ\tg)\circ\tf=\th
\circ(\tg\circ\tf)$ as functors; moreover, for every uniformizing system $\unif{i}\in\mathcal{U}$, if we call
$(\tV_i,H_i,\phi_i)$, $(\tW_i,K_i,\xi_i)$ and $(\tZ_i,L_i,\psi_i)$ the images of $(\tU_i,G_i,\pi_i)$ via $\tf$,
$\tg\circ\tf$ and $\th\circ\tg\circ\tf$ respectively, we get:

\begin{eqnarray*}
& \Big((\th\circ\tg)\circ\tf\Big)_{\tU_i,\tZ_i}=(\th\circ\tg)_{\tV_i,\tZ_i}\circ\tf_{\tU_i,\tV_i}= &\\
& =\th_{\tW_i,\tZ_i}\circ\tg_{\tV_i,\tW_i}\circ\tf_{\tU_i,\tV_i}=\Big(\th\circ(\tg\circ\tf)\Big)_{\tU_i,
  \tZ_i} &
\end{eqnarray*}

hence $(\th\circ\tg)\circ\tf=\th\circ(\tg\circ\tf)$ as compatible systems.

\item Let us fix any diagram of compatible systems and natural transformations of the form:

\[\begin{tikzpicture}
    \def\x{1.5}
    \def\y{-1.2}
    \node (A0_0) at (0*\x, 0*\y) {$\mathcal{U}$};
    \node (A0_1) at (1.1*\x, 0*\y) {$\Downarrow\delta$};
    \node (A0_2) at (2*\x, 0*\y) {$\mathcal{V}$}; 
    \node (A0_3) at (3.1*\x, 0*\y) {$\Downarrow\eta$};
    \node (A0_4) at (4*\x, 0*\y) {$\mathcal{W}$};
    \node (A0_5) at (5.1*\x, 0*\y) {$\Downarrow\omega$};
    \node (A0_6) at (6*\x, 0*\y) {$\mathcal{Z}$};

    \path (A0_0) edge [->,bend left=25] node [auto] {$\scriptstyle{\tf_1}$} (A0_2);
    \path (A0_0) edge [->,bend right=25] node [auto,swap] {$\scriptstyle{\tf_2}$} (A0_2);
    \path (A0_2) edge [->,bend left=25] node [auto] {$\scriptstyle{\tg_1}$} (A0_4);
    \path (A0_2) edge [->,bend right=25] node [auto,swap] {$\scriptstyle{\tg_2}$} (A0_4);
    \path (A0_4) edge [->,bend left=25] node [auto] {$\scriptstyle{\th_1}$} (A0_6);
    \path (A0_4) edge [->,bend right=25] node [auto,swap] {$\scriptstyle{\th_2}$} (A0_6);
\end{tikzpicture}\]

and any uniformizing system $(\tU_i,G_i,\pi_i)\in\mathcal{U}$; then let us define:

$$(\tW_i^{mn},K_i^{mn},\xi_i^{mn}):=\tg_n\circ\tf_m\unif{i}\textrm{\quad for\quad}m,n=1,2.$$

If we use also the notations introduced in definition \ref{nat-tran-orb} we have:

\begin{eqnarray*}
& \Big((\omega\ast\eta)\ast\delta\Big)_{\tU_i}=(\omega\ast\eta)_{\tV_i^2}\circ\Big((\th_1\circ\tg_1)
  (\delta_{\tU_i})\Big)= &\\
& =\omega_{\tW_i^{22}}\circ\th_1(\eta_{\tV_i^2})\circ\Big(\th_1\circ\tg_1(\delta_{\tU_i})\Big)=
  \omega_{\tW_i^{22}}\circ\th_1\Big(\eta_{\tV_i^2}\circ\tg_1(\delta_{\tU_i})\Big)= &\\
& =\omega_{\tW_i^{22}}\circ\th_1\Big((\eta\ast\delta)_{\tU_i}\Big)=\Big(\omega\ast(\eta\ast\delta)
  \Big)_{\tU_i}. &
\end{eqnarray*}

Hence we have proved that $(\omega\ast\eta)\ast\delta=\omega\ast(\eta\ast\delta)$.

\item For every compatible system $\tf:\mathcal{U}\rightarrow\mathcal{V}$ we have $\tf\circ 1_{\mathcal{U}}=\tf$
as functors. Moreover, for every $\unif{i}\in\mathcal{U}$ we have $(\tf\circ 1_{\mathcal{U}})_{\tU_i,\tV_i}=
\tf_{\tU_i,\tV_i}\circ 1_{\tU_i}=\tf_{\tU_i,\tV_i}$. Hence $\tf\circ 1_{\mathcal{U}}=\tf$ in the sense of
compatible systems. In the same way one can check that $1_{\mathcal{V}}\circ\tf=\tf$.

\item A direct check proves also that for every natural transformation $\delta:\tf_1\Rightarrow\tf_2$ in
$\CAT{Pre-Orb}(\mathcal{U},\mathcal{V})$ we have that $\delta\ast i_{\mathcal{U}}=\delta=i_{\mathcal{V}}\ast
\delta$. This concludes the proof that $\CAT{Pre-Orb}$ is a 2-category.
\end{enumerate}
\end{proof}

\subsection{Equivalent orbifold atlases}\label{section-equivalences}

It is well known that a manifold is an equivalence class of compatible manifold atlases; in literature there is
an analogous notion in the framework of orbifolds. 

\begin{defin}\label{equivalence-CR} (equivalent to \cite{LU}, \S 2.1 and \cite{CR}, def. 4.1.2)
Two atlases $\mathcal{U}=\{\unif{i}\}_{i\in I}$ and $\mathcal{V}=\{\uniff{j}\}_{j\in J}$ on the same space $X$
are \emph{equivalent at a point} $x\in X$ iff there exists a uniformizing system $(\tW,K,\xi)$ around $x$,
together with 2 embeddings:

$$\unif{i}\sinistra{\lambda}(\tW,K,\xi)\fre{\lambda'}\uniff{j}$$

for some uniformizing systems $\unif{i}\in\mathcal{U}$ and $\uniff{j}\in\mathcal{V}$.
Note that we don't require that $(\tW,K,\xi)$ belongs to $\mathcal{U}$ and/or $\mathcal{V}$.
Two atlases on a space $X$ are \emph{equivalent} iff they are equivalent at every point of $X$.
\end{defin}

\begin{lem}\label{equivalence-CR-proof}
This is an equivalence relation (see the appendix for the proof).
\end{lem}

Actually, in literature one can also find these definitions:

\begin{defin}\label{refinements}(\cite{ALR}, \S 1) 
An orbifold atlas $\mathcal{U}$ on $X$ is said to \emph{refine} another orbifold atlas $\mathcal{V}$ on the 
same topological space iff
for every uniformizing system in $\mathcal{U}$ there exists an embedding of it into some uniformizing system of
$\mathcal{V}$. Equivalenty, $\mathcal{U}=\{(\tU_i,G_i,\pi_i)\}_{i\in I}$ is a refinement of $\mathcal{V}=\{(\tV_j,
H_j,\phi_j)\}_{j\in J}$ iff there exists a set map $\gamma:I\rightarrow J$ and embeddings $\lambda_i:(\tU_i,G_i,
\pi_i)\rightarrow (\tV_{\gamma(i)},H_{\gamma(i)},\phi_{\gamma(i)})$ for every $i\in I$.
In a some sense, we can consider this as a compatible system $\mathcal{U}\rightarrow\mathcal{V}$ for the identity
on $X$, except for the fact that in general this will not be a functor (actually, it is not even defined on
embeddings). However, we will use the same abuse of notation that we used for compatible system, i.e. we will
write $(\tV_i,H_i,\phi_i)$ instead of $(\tV_{\gamma(i)},H_{\gamma(i)},\phi_{\gamma(i)})$.
\end{defin}

\begin{defin}(\cite{ALR}, \S 1)\label{equivalence-ALR}
Two orbifold atlases on the same space $X$ are \emph{equivalent} if they have a common refinement.
\end{defin}

\begin{prop}
Definitions \ref{equivalence-CR} and \ref{equivalence-ALR} coincide.
\end{prop}

\begin{proof}
Let us fix a space $X$ and two orbifold atlases $\mathcal{U}_1$ and $\mathcal{U}_2$ which are equivalent w.r.t.
definition \ref{equivalence-CR}; then we define a family $\mathcal{W}$ whose elements are \emph{all} the
uniformizing systems (for some open set of $X$) that have an embedding in at least one chart of $\mathcal{U}_1$ and
one chart of $\mathcal{U}_2$. Then we consider $\mathcal{W}$ as an orbifold atlas by adding all the possible
embedddings between its charts and it is easy to see that actually $\mathcal{W}$ is an orbifold atlas on $X$ and it
refines both $\mathcal{U}_1$ and $\mathcal{U}_2$.

Conversely, if there exists a common refinement $\mathcal{W}$ of two atlases $\mathcal{U}_1$ and $\mathcal{U}_2$, then 
we get that they are equivalent with respect to definition \ref{equivalence-CR} by a direct application of definition
\ref{refinements}.
\end{proof}

\begin{cor}
The notion of having a common refinement is a relation of equivalence.
\end{cor}

So it makes sense to give the following definition:

\begin{defin}(\cite{ALR}, def. 1.2)
A \emph{complex orbifold structure} on a second countable paracompact Hausdorff topological space $X$ is an
equivalence class of orbifold atlases on $X$. We will denote such an object by $\mathcal{X}$ or $[X]$. We will
call \emph{orbifold} the pair $(X,\mathcal{X})$, or, by abuse of notation, just the orbifold \mbox{structure
$\mathcal{X}$.} We say that $\mathcal{X}$ has \emph{dimension} $n$ if there is an atlas $\mathcal{U}$ 
of dimension $n$ in the class $\mathcal{X}$. This is equivalent to say that \emph{every} atlas of the class
has the same dimension $n$.
\end{defin}

\begin{defin}\label{maximal-atlas}
For every orbifold $\mathcal{X}$ on $X$ we define the \emph{maximal atlas} associated to it as the family
of all the uniformizing systems of all the atlases of the class $\mathcal{X}$. If one consider also
all the possible embeddings between the charts of this family, one can easily prove that it is actually an orbifold
atlas for $X$, that it belongs to the class $\mathcal{X}$ and that it is refined by every atlas of $\mathcal{X}$.
\end{defin}

At the end of this paper we will also need the following new equivalence relation on the set of orbifold atlases.

\begin{defin}\label{new-equivalence}
Suppose that we have an orbifold atlas $\mathcal{U}=\{\unif{i}\}_{i\in I}$ on $X$ and an homeomorphism
$\varphi:X\rightarrow X'$; then we can define the following family:

$$\varphi_{\ast}(\mathcal{U}):=\{(\tU_i,G_i,\varphi\circ\pi_i)\}_{i\in I}$$

which is an orbifold atlas on $X'$. Now suppose that we have two orbifold atlases $\mathcal{U}$ on $X$ and
$\mathcal{U}'$ on $X'$; we say that they are \emph{equivalent} if and only if the following 2 conditions hold:

\begin{enumerate}[(a)]\parindent=0pt
\item there exists an homeomorphism $\varphi:X\rightarrow X'$;

\item the orbifold atlases $\varphi_{\ast}(\mathcal{U})$ and $\mathcal{U}'$ on $X'$ are equivalent with
respect to the previous definition.
\end{enumerate}

This relation is obviously reflexive and simmetric; moreover, one can easily prove transitivity
using the transitivity of the equivalence relation on a fixed topological space.
Hence \emph{it is an equivalence relation on the set of all orbifold atlases 
over any topological space, i.e. over the objects of} $\CAT{Pre-Orb}$.
\end{defin}

\section{Internal groupoids in a category \texorpdfstring{$\ca{C}$}{C}}
We will use this section in order to recall the basic notions of internal groupoids in 
any category $\ca{C}$ and then we will specialize to
$\ca{C}=\CAT{Manifolds}$.

\subsection{Groupoid objects in a fixed category}

Let us fix a category $\ca{C}$ and two objects $R$ and $U$ in it. In the case where both $R$ and $U$ are sets (or
sets with additional properties), we would like to think of $R$ as a set of ``arrows'' or ``identifications''
between points of $U$. Hence
we have to define two morphisms ``source'' and ``target'' from $R$ to $U$ which associate to every point in $R$
a pair of points in $U$ (their source and target). Moreover, we would like to think to the pair $(R,U)$ as a 
``groupoid'' in the sense of a category where all the arrows are invertible. Hence we have to define a binary
operation of ``multiplication'' between composable arrows of $R$ and an operation of ``inversion'' from $R$ to
itself. Moreover, in order to have a category we have to associate to every point $x$ of $U$ an arrow ``identity''
in $R$ with source and target coinciding with $x$. Furthermore, we would like that all these five maps (source,
target, multiplication, inverse and identity) are morphisms in the category $\ca{C}$ we have fixed. Finally,
these morphisms will have to satisfy some compatibility axioms, so we get the following definition:

\begin{defin}\label{groupoid-object}
(\cite{BCE+}, \S 3.1) A \emph{groupoid object} or \emph{internal groupoid} in a category $\ca{C}$ is the datum
of two objects $R,U$ and five morphisms of $\ca{C}$:

\begin{itemize}
\item $s,t: R \rightrightarrows U$ such that the fiber product $\fibre{R}{t}{s}{R}$ \emph{exists} in $\ca{C}$; 
these two maps are usually called \emph{source} and \emph{target} of the groupoid object;

\item $m: \fibre{R}{t}{s}{R}\rightarrow R$, called \emph{multiplication};

\item $i: R \rightarrow R$, known as \emph{inverse} of the groupoid object;

\item $e: U \rightarrow R$, called \emph{identity};
\end{itemize}

which satisfy the following axioms:

\begin{enumerate}[(i)]\parindent=0pt
\item $s\circ e=1_U=t\circ e$;

\item if we call $pr_1$ and $pr_2$ the two projections from the fibered product $\fibre{R}{t}{s}{R}$ to $R$, then 
we have $s\circ m=s\circ pr_1$ and $t\circ m=t\circ pr_2$, i.e:

\[\begin{tikzpicture}[scale=0.8]
    \def\x{1.5}
    \def\y{-1.2}
    \node (A0_0) at (0*\x, 0*\y) {$\fibre{R}{t}{s}{R}$};
    \node (A0_2) at (2*\x, 0*\y) {$R$};
    \node (A0_5) at (4*\x, 0*\y) {$\fibre{R}{t}{s}{R}$};
    \node (A0_7) at (6*\x, 0*\y) {$R$};
    \node (A1_1) at (1*\x, 1*\y) {$\curvearrowright$};
    \node (A1_6) at (5*\x, 1*\y) {$\curvearrowright$};
    \node (A2_0) at (0*\x, 2*\y) {$R$};
    \node (A2_2) at (2*\x, 2*\y) {$U$};
    \node (A2_5) at (4*\x, 2*\y) {$R$};
    \node (A2_7) at (6*\x, 2*\y) {$U;$};
    \path (A0_5) edge [->] node [auto,swap] {$\scriptstyle{pr_2}$} (A2_5);
    \path (A2_5) edge [->] node [auto,swap] {$\scriptstyle{t}$} (A2_7);
    \path (A2_0) edge [->] node [auto,swap] {$\scriptstyle{s}$} (A2_2);
    \path (A0_2) edge [->] node [auto] {$\scriptstyle{s}$} (A2_2);
    \path (A0_7) edge [->] node [auto] {$\scriptstyle{t}$} (A2_7);
    \path (A0_0) edge [->] node [auto] {$\scriptstyle{m}$} (A0_2);
    \path (A0_0) edge [->] node [auto,swap] {$\scriptstyle{pr_1}$} (A2_0);
    \path (A0_5) edge [->] node [auto] {$\scriptstyle{m}$} (A0_7);
\end{tikzpicture}\]

\item (associativity) the two morphisms $m\circ (1_R\times m)$ and $m\circ(m\times 1_R)$ are equal:

\[\begin{tikzpicture}[scale=0.8]
    \def\x{2.0}
    \def\y{-1.2}
    \node (A0_0) at (0*\x, 0*\y) {$\fibre{R}{t}{s}{R\,_t\times_s R}$};
    \node (A0_2) at (2*\x, 0*\y) {$\fibre{R}{t}{s}{R}$};
    \node (A1_1) at (1*\x, 1*\y) {$\curvearrowright$};
    \node (A2_0) at (0*\x, 2*\y) {$\fibre{R}{t}{s}{R}$};
    \node (A2_2) at (2*\x, 2*\y) {$R;$};
    \path (A0_0) edge [->] node [auto,swap] {$\scriptstyle{m\times 1_R}$} (A2_0);
    \path (A0_0) edge [->] node [auto] {$\scriptstyle{1_R\times m}$} (A0_2);
    \path (A0_2) edge [->] node [auto] {$\scriptstyle{m}$} (A2_2);
    \path (A2_0) edge [->] node [auto,swap] {$\scriptstyle{m}$} (A2_2);
\end{tikzpicture}\]

\item (unit) the two morphisms $m\circ(e\circ s,1_R)$ and $m\circ(1_R,e\circ t)$ from $R$ to $R$ are both equal to 
the identity of $R$:

\[\begin{tikzpicture}[scale=0.8]
    \def\x{1.5}
    \def\y{-1.2}
    \node (A0_0) at (0*\x, 0*\y) {$R$};
    \node (A0_2) at (2*\x, 0*\y) {$\fibre{R}{t}{s}{R}$};
    \node (A0_4) at (4*\x, 0*\y) {$R$};
    \node (A0_6) at (6*\x, 0*\y) {$\fibre{R}{t}{s}{R}$};
    \node (A1_1) at (1.4*\x, 0.8*\y) {$\curvearrowright$};
    \node (A1_5) at (5.4*\x, 0.8*\y) {$\curvearrowright$};
    \node (A2_2) at (2*\x, 2*\y) {$R$};
    \node (A2_6) at (6*\x, 2*\y) {$R;$};
    \path (A0_6) edge [->] node [auto] {$\scriptstyle{m}$} (A2_6);
    \path (A0_4) edge [->] node [auto] {$\scriptstyle{(1_R,e\circ t)}$} (A0_6);
    \path (A0_0) edge [->] node [auto,swap] {$\scriptstyle{1_R}$} (A2_2);
    \path (A0_4) edge [->] node [auto,swap] {$\scriptstyle{1_R}$} (A2_6);
    \path (A0_2) edge [->] node [auto] {$\scriptstyle{m}$} (A2_2);
    \path (A0_0) edge [->] node [auto] {$\scriptstyle{(e\circ s,1_R)}$} (A0_2);
\end{tikzpicture}\]

\item (inverse) $i\circ i=1_R$, $s\circ i=t$ (and therefore $t\circ i=s$). Moreover, we require that $m\circ (1_R,i)
=e\circ s$ and $m\circ (i,1_R)=e\circ t$:
\end{enumerate}

\end{defin}

Using this list of axioms one can deduce the following useful property:

\begin{lem}\label{groupoid-property-1}(\cite{BCE+}, exercise 3.1)
$m\circ(i\circ pr_2,i\circ pr_1)=i\circ m$.
\end{lem}

If we assume that the category $\ca{C}$ is fixed, we will denote any groupoid object as before by $\groupR$.
In some articles one can also find the following notations:

\begin{itemize}
\item $(U,R,s,t,m,e,i)$;

\item $\fibre{R}{s}{t}{R}\fre{m}R\fre{i}\groupR\fre{e}R$;

\item $G=(\phantom{G}_1\hspace{-0.42 cm}G^{\phantom{,,,}^{d_0}}_{\phantom{,,,}_{d_1}}\hspace{-0.45 cm}
\rightrightarrows{G_0})$ (we prefer not to use this notation since otherwise a lot of the constructions of sections
3 and 4 will have too much indexes and will not be simply manageable).
\end{itemize}

\subsection{Morfisms and 2-morphisms between groupoid objects}

\begin{defin}\label{morphism-groupoid}(\cite{M}, \S 2.1)
Given two groupoid objects $\groupR$ and $\groupRR$ in a fixed category $\ca{C}$, a \emph{morphism} between them is
a pair $(\psi,\Psi)$, where $\psi:U\rightarrow U'$ and $\Psi:R\rightarrow R'$ are both \emph{morphisms in }$\ca{C}$,
which together commute with all structure morphisms of the two groupoid objects. In other words, we ask that the
following five identities are satisfied:

\begin{equation}\label{eq-9-bis}
s'\circ\Psi=\psi\circ s,\quad t'\circ\Psi=\psi\circ t,\quad \Psi\circ e=e'\circ\psi,
\end{equation}

\begin{equation}
\Psi\circ m=m'\circ (\Psi\times\Psi) \AND \Psi\circ i=i'\circ\Psi.
\end{equation}

\end{defin}

\begin{rem}\label{groupoid-property-2}
Note that if we fix the morphism $\Psi$, then $\psi$ is uniquely determined by these properties: indeed using axiom
(i) for the second groupoid object and (\ref{eq-9-bis}), we get that $\psi=s'\circ\Psi\circ e$.
\end{rem}

\begin{defin}Let us consider 3 groupoid objects in $\ca{C}$ and 2 morphisms:

$$(\groupR)\stackrel{(\psi,\Psi)}{\longrightarrow}(\groupRR)\stackrel{(\phi,\Phi)}{\longrightarrow}(\groupRRR).$$

It is easy to see that if we define the \emph{composition} $(\phi,\Phi)\circ(\psi,\Psi)$ as
$(\phi\circ\psi,\Phi\circ\Psi)$, then this is again a morphism of groupoid objects from $(\groupR)$
to $(\groupRR)$.
\end{defin}

Now we want to make groupoid objects into a 2-category, i.e. we want to define 2-morphisms, which will be called 
``natural transformations''.

\begin{defin}\label{nat-tran-group}(\cite{PS}, def. 2.3) Suppose we have fixed two morphisms of groupoid
objects in $\ca{C}$:

$$(\psi,\Psi),(\phi,\Phi):\,\,(\groupR)\rightarrow(\groupRR),$$

then a \emph{natural transformation} $\alpha:(\psi,\Psi)\Rightarrow(\phi,\Phi)$ is the datum of a \emph{morphism}
$\alpha:U\rightarrow R'$ in $\ca{C}$ such that the following 2 conditions hold:

\begin{enumerate}[(i)]\parindent=0pt
\item $s'\circ\alpha=\psi\AND t'\circ\alpha=\phi$;

\item $m'\circ(\alpha\circ s,\Phi)=m'\circ(\Psi,\alpha\circ t).$
\end{enumerate}

Note that using (i) together with the definition of morphism between groupoid objects, we get that:

$$t'\circ(\alpha \circ s)=\phi\circ s=s'\circ\Phi\AND t'\circ\Psi=\psi\circ t=s'\circ(\alpha\circ t);$$
  
hence we can consider both $(\alpha\circ s,\Phi)$ and $(\Psi,\alpha\circ t)$ as morphisms in the category $\ca{C}$
from $R$ to $\fibre{R'}{t'}{s'}{R'}$, so both the L.H.S. and the R.H.S. of (ii) are well defined.
\end{defin}

\begin{definlem}\label{vertical-composition-groupoid}
Let us consider a diagram as follows:

\[\begin{tikzpicture}[scale=0.8]
    \def\x{2.5}
    \def\y{-1.2}
    \node (A0_0) at (0*\x, 0*\y) {$(\groupR)$};
    \node (A0_1) at (1.1*\x, 0.6) {$\Downarrow\alpha$};
    \node (A0_1) at (1.1*\x, -0.8) {$\Downarrow\beta$};
    \node (A0_2) at (2*\x, 0*\y) {$(\groupRR);$};
    \path (A0_0) edge [->,bend left=50] node [auto] {$\scriptstyle{(\psi_1,\Psi_1)}$} (A0_2);
    \path (A0_0) edge [->] node [auto,swap] {$\scriptstyle{(\psi_2,\Psi_2)}$} (A0_2);
    \path (A0_0) edge [->,bend right=50] node [auto,swap] {$\scriptstyle{(\psi_3,\Psi_3)}$} (A0_2);
\end{tikzpicture}\]

using definition \ref{nat-tran-group} for $\alpha$ and $\beta$ we get that it makes sense to consider the morphism
$(\alpha,\beta):U\rightarrow\fibre{R'}{t'}{s'}{R'}$ and we can define:

$$\beta\odot\alpha:=m'\circ(\alpha,\beta):U\rightarrow R'.$$

In this way we have defined a natural transformation (for the proof, see the appendix), called \emph{vertical
composition} of $\alpha$ and $\beta$ and denoted with $\beta\odot\alpha:(\psi_1,\Psi_1)\Rightarrow(\psi_3,\Psi_3)$.
\end{definlem}

\begin{defin} For every morphism $(\psi,\Psi):(\groupR)\rightarrow(\groupRR)$ we define its identity as the natural
transformation:

$$i_{(\psi,\Psi)}:=e'\circ\psi=\Psi\circ e:\quad(\psi,\Psi)\Rightarrow(\psi,\Psi)$$

A direct check proves that this is actually a natural transformation and that for every $\alpha:(\psi,\Psi)
\Rightarrow(\phi,\Phi)$ and for every $\beta:(\theta,\Theta)\Rightarrow(\phi,\Phi)$ we have:

\begin{equation}\label{eq-12}
\alpha\odot i_{(\psi,\Psi)}=\alpha\AND i_{(\psi,\Psi)}\odot\beta=\beta. 
\end{equation}
\end{defin}

\begin{definlem}\label{horizontal-composition-groupoid}
Let us consider a diagram of the form:

\begin{equation}
\begin{tikzpicture}[scale=0.8]
    \def\x{2.0}
    \def\y{-1.2}
    \node (A0_0) at (0*\x, 0*\y) {$(\groupR)$};
    \node (A0_1) at (1*\x, 0*\y) {$\Downarrow\alpha$};
    \node (A0_2) at (2*\x, 0*\y) {$(\groupRR)$};
    \node (A0_3) at (3*\x, 0*\y) {$\Downarrow\beta$};
    \node (A0_4) at (4*\x, 0*\y) {$(\groupRRR).$};
    \path (A0_2) edge [->,bend left=25] node [auto] {$\scriptstyle{(\phi_1,\Phi_1)}$} (A0_4);
    \path (A0_2) edge [->,bend right=25] node [auto,swap] {$\scriptstyle{(\phi_2,\Phi_2)}$} (A0_4);
    \path (A0_0) edge [->,bend left=25] node [auto] {$\scriptstyle{(\psi_1,\Psi_1)}$} (A0_2);
    \path (A0_0) edge [->,bend right=25] node [auto,swap] {$\scriptstyle{(\psi_2,\Psi_2)}$} (A0_2);
\end{tikzpicture}
\end{equation}

In particular, we get that:

\begin{eqnarray}
\label{eq-14} & s'\circ\alpha=\psi_1\AND t'\circ\alpha=\psi_2;&\\
\label{eq-15} & s''\circ\beta=\phi_1\AND t''\circ\beta=\phi_2;&\\
& t''\circ(\Phi_1\circ\alpha)=\phi_1\circ t'\circ\alpha\stackrel{(\ref{eq-14})}{=}
  \phi_1\circ\psi_2\stackrel{(\ref{eq-15})}{=}s''\circ(\beta\circ\psi_2).&
\end{eqnarray}

Hence $(\Phi_1\circ\alpha,\beta\circ\psi_2):U\rightarrow\fibre{R''}{t''}{s''}{R''}$, so we can define:

$$\beta\ast\alpha:=m''\circ(\Phi_1\circ\alpha,\beta\circ\psi_2):U\rightarrow R''$$

and we get that $\beta\ast\alpha$ is a natural transformation from $(\phi_1,\Phi_1)\circ(\psi_1,\Psi_1)$ to 
$(\phi_2,\Phi_2)\circ(\psi_2,\Psi_2)$ (see the appendix for the proof), called 
\emph{horizontal composition} of $\alpha$ and $\beta$.
\end{definlem}

Using all the previous data it is not so difficult to prove the following result:

\begin{prop}\label{2-cat-groupoid}
Let us fix a category $\ca{C}$; the definitions of groupoid objects, morphisms of groupoid objects,
natural transformations and compositions
$\circ,\odot,\ast$ give rise to a 2-category, that we will denote with $(\ca{C}-\textbf{Groupoids})$.
\end{prop}

We omit this proof, since this is just a direct check of all the axioms of definition \ref{2-cat}. In any case,
we will just be interested in a special case of this result, that will be recalled below.

\subsection{The 2-category \CAT{Grp}}

Given any pair of morphisms $s,t:R\rightarrow U$ in a fixed category $\ca{C}$, in order to define a groupoid
object from this data, we must be sure that the fibered product $\fibre{R}{t}{s}{R}$ exists.
This is always ensured e.g. when we work in the categories $\CAT{Sets}$, $\CAT{Groups}$ or $\CAT{Schemes}$,
but in general it is no more true in the category $\CAT{Manifolds}$ (complex manifolds and holomorphic
maps between them). In this last category (which will be used in this section) it is known from category
theory that \emph{if} the fibered product of any pair of morphisms $f:X\rightarrow Y$ and $g:Z\rightarrow Y$ exists,
 it is obtained adding a natural structure of manifold on the set-theoretical fibered product.
The problem is that such a structure of manifold exists only if we add some additional hypothesis
on $f$ and/or $g$. One of the most useful condition is about submersions:

\begin{prop}\label{fiber-products-1}
Let us fix any pair of holomorphic maps between complex manifolds $f:X\rightarrow Y$ and $g:Z\rightarrow Y$.
If $f$ is a submersion, then the set-theoretical fiber product:

$$\fibra{X}{Y}{Z}=\{(x,z)\in X\times Z\textrm{ s.t. }f(x)=g(z)\}$$

is a complex submanifold of $X\times Z$, with complex dimension equal to $dim(X)+dim(Z)-dim(Y)$. Moreover,
the map $pr_2:\fibra{X}{Y}{Z}\rightarrow Z$ is again a submersion and if $f$ is \'etale (i.e. a local
biholomorphism), so is $pr_2$.
\end{prop}

For a hint of the proof, see \cite{FG}, chapter IV.1, exercise 6 and (for the smooth case) \cite{D},
chapter XVI.8, exercise 10.

\begin{defin}
(Adapted from \cite{Ler}, def. 2.11) A groupoid object in $\CAT{Manifolds}$ is called \emph{Lie groupoid}
if both its source and target maps are holomorphic \emph{submersions}.
\end{defin}

\begin{rem}\label{fiber-products-2}
Using the previous proposition, we get that the fiber product used in the first point of the definition of groupoid
objects exists. Moreover, the resulting maps $pr_1$ and $pr_2$ are again both submersions. Hence also the fiber
product $\fibre{R}{t}{s}{R\,_t\times_s R}$ exists in $\CAT{Manifolds}$. Indeed:

\begin{eqnarray*}
& \fibre{R}{t}{s}{R\,_t\times_s R}=\{(r,r',r'')\in R\times R\times R \textrm{ s.t. }s(r')=t(r)\,
   \textrm{and}\,s(r'')=t(r')\}= &\\
& =(\fibre{R}{t}{s}{R)_{\hat{t}}\times_s R}&
\end{eqnarray*}

where $\hat{t}:=t\circ pr_2$. Here the fiber product behind parenthesis exists, $\hat{t}$ is a submersion (because
$t$ is so by hypothesis and $pr_2$ is so using the previous proposition) and also $s$ is a submersion, hence the
whole fiber product exists in $\CAT{Manifolds}$ and again the projection maps are submersions, so by induction we can
prove that there exists fiber products of the form $\fibre{R}{t}{s}{\cdots\phantom{R}\,_t\times_s R}$ for finitely
many terms.
\end{rem}

\begin{defin} (\cite{M}, \S 1.5)
A groupoid object in $\CAT{Manifolds}$ is \emph{proper} if the map $(s,t):R\rightarrow U\times U$ (called
\emph{relative diagonal}) is proper, i.e. if the pre-image of any compact set in $U\times U$ is compact in $R$.
\end{defin}

\begin{defin}
(\cite{M}, \S 1.2) An \emph{\'etale groupoid} is a groupoid object $\groupR$ in $\CAT{Manifolds}$ such that the maps
$s$ and $t$ are both \'etale (i.e. local biholomorphisms). Using remark \ref{fiber-products-2}, the fiber products
$\fibre{R}{t}{s}{R}$ and $\fibre{R}{t}{s}{R\,_t\times_s R}$ exist, so all definition \ref{groupoid-object} still
makes sense (clearly every \'etale groupoid is also a Lie groupoid).
\end{defin}

\begin{defin}\label{effective-groupoid}
(\cite{M}, example 1.5) Let $\groupR$ be an \'etale proper groupoid, let us fix any point $\tx\in U$
and let us define $R_{\tx}:=(s,t)^{-1}\{(\tx,\tx)\}$, called the \emph{isotropy subgroup of} $\tx$. This set is
naturally a group (using the multiplication $m$) and it is compact because $(s,t)$ is proper; moreover its points are
all isolated because $s$ is \'etale, hence locally invertible; hence $R_{\tx}$ is a \emph{finite} set of points of $R$.
Since both $s$ and $t$ are \'etale, for every point $g$ in this set, we can find a sufficiently small open
neighborhood $W_g$ of $g$ where both $s$ and $t$ are invertible. Then to every point $g$ we can associate the set map:

$$\tilde{g}:=t\circ(s|_{W_g})^{-1}:\,s(W_g)\rightarrow t(W_g)$$

which is a biholomorphism between two open neighborhoods of $\tx$. Then we can define the set map:

\begin{equation}\label{eq-27}
f_{\tx}:R_{\tx}\rightarrow Diff_{\tx}(U) 
\end{equation}

(where $Diff_{\tx}(U)$ is the group of germs of holomorphic maps defined on an open neighborhood of $\tx$ in $U$
and which fix $\tx$), that to every $g\in R_{\tx}$ associates the germ of the function $\tilde{g}$ at the point
$\tx$. Then we say that the groupoid $\groupR$ is \emph{effective} (or
\emph{reduced}) if $f_{\tx}$ is injective for every $\tx$ in $U$. This notion will correspond to the notion of
reduced orbifolds via the 2-functor $F$ that we'll define in section 3.
\end{defin}

We have this useful result:

\begin{lem}\label{lemma-a}
The following identities hold on any open connected neighborhood of $s(g)$ where both the L.H.S. and the R.H.S.
are defined:
\begin{enumerate}[(a)]
\item for every pair of points $(g,h)\in\fibre{R}{t}{s}{R}$ if we call $k:=m(g,h)$, we have that $\th\circ\tg=\tk$;
\item for every $g\in R$ we have that $\widetilde{i(g)}=\tg^{-1}$.
\end{enumerate}
\end{lem}

\begin{proof}
(a) Let us consider the points $g,h$ and $k:=m(g,h)$; since $s$ and $t$ are \'etale, there exist open neighborhoods
$W_g,W_h,W_k$ of them where both $s$ and $t$ are invertible. Now let us consider the open neighborhood $t(W_h)\cap
s(W_g)$ of $t(h)=s(g)=:\tx$; then for every point $\ty$ in it we can define:

$$f_{g,h}(\ty):=m\Big((t|_{W_g})^{-1}(\ty),(s|_{W_h})^{-1}(\ty)\Big);$$

this map is well defined and holomorphic because composition of holomorphic maps. Moreover, using the properties
of groupoid objects we get that:

$$t\circ f_{g,h}=\th\AND s\circ f_{g,h}=\tg^{-1};$$ 

now $f_{g,h}(\tx)=m(g,h)=k$, hence we can apply $(s|_{W_k})^{-1}$ to the second identity and we get that:

$$f_{g,h}=(s|_{W_k})^{-1}\circ \tg^{-1}.$$

Clearly this happens
only for points $\ty$ sufficiently near $\tx$, but there is no loss of generality in restricting $W_g$ or $W_k$ in
order to get that this is true for every point $\ty$ of the open neighborhood of $\tx$. So we get:

$$\th=t\circ f_{g,h}=t\circ (s|_{W_k})^{-1}\circ\tg^{-1}=\tk\circ\tg^{-1}$$

or, equivalently,

$$\th\circ\tg=\tk.$$

We have proved this identity only for points sufficiently near $s(h)$ in $U$; since we are working with
holomorphic functions, we get that this identity is still true for every open \emph{connected} neighborhood of
$s(h)$ where both the left and the right hand sides are defined.\\

(b) Using axiom (v) of definition \ref{groupoid-object} and part (a) we have that $\tg\circ\widetilde{i(g)}=
\widetilde{m(i(g),g)}=\widetilde{e(t(g))}=id$ and analogously $\widetilde{i(g)}\circ\tg=id$ on sufficiently small
open neighborhoods of $t(g)$ and $s(g)$ respectively. Hence $\widetilde{i(g)}=\tg^{-1}$.
\end{proof}

\begin{cor}
The set map of (\ref{eq-27}) is a group homomorphism.
\end{cor}

Now let us recall the following result (stated in the smooth case, but also true in the holomorphic case).

\begin{prop}(\cite{PS}, \S 2.1)
The data of Lie groupoids, morphisms and natural transformations between them form a 2-category, known as
$\CAT{LieGpd}$.
\end{prop}

\begin{defin}
Let us define the following 2-category, that we will denote with $\CAT{Grp}$:

\begin{itemize}
\item the objects are all the \emph{proper, \'etale and effective groupoid objects} in the category 
$\CAT{Manifolds}$ of complex manifolds and holomorphic maps between them;

\item the morphisms between these objects are \emph{all} the morphisms between them as groupoid objects in
$\CAT{Manifolds}$;

\item the 2-morphisms are \emph{all} the natural transformations between the morphisms of the previous point.
\end{itemize}
\end{defin}

Since every \'etale groupoid object is in particular a Lie groupoid, we are just considering a full
sub-2-category of $\CAT{LieGpd}$, hence we have for free that:

\begin{prop}
$\CAT{Grp}$ \emph{is actually a 2-category}.
\end{prop}

\section{From orbifolds to groupoids}

The main aim of this work is to describe a 2-functor $F$ from $\CAT{Pre-Orb}$ to $\CAT{Grp}$. This will be done on the
level of objects by adapting to the complex case a construction due to D. Pronk (\cite{Pr}, \S 4.4) in the smooth case:
first of all
for every orbifold atlas $\mathcal{U}$ we define two \emph{sets} $U$ and $R$ and five \emph{set maps} $s,t,m,i,e$
that make the pair $(R,U)$ into a groupoid object in $\CAT{Sets}$. Then we describe the topology on $R$ and $U$ and
we prove that actually they are both complex manifolds such that the five structure maps are holomorphic. Finally,
we prove that the groupoid object is proper, \'etale and effective, hence we obtain an object in $\CAT{Grp}$.\\

The original part of this section consists in the extension of this construction to the level of morphisms
and 2-morphisms. The last part of this section will be devoted to prove that actually this gives
rise to a 2-functor $F$ from $\CAT{Pre-Orb}$ to $\CAT{Grp}$.

\subsection{Objects}

Let us fix any orbifold atlas $\mathcal{U}=\{(\tU_i,G_i,\pi_i)\}_{i\in I}$ of dimension $n$ on a paracompact and
second countable Hausdorff topological space $X$; we want to associate to it a groupoid object $\groupR$
which ``encodes'' the informations about the underlying topological space $X$ and the atlas $\mathcal{U}$.
First of all, we define:

$$U:=\coprod_{i\in I}\tU_i$$

with the topology of the disjoint union. Since all the $\tU_i$'s are open subsets of $\mathbb{C}^n$ and $U$ is
their disjoint union, then $U$ is a Hausdorff paracompact complex manifold of dimension $n$. The points of
this manifold will be always denoted as $(\tx_i,\tU_i)$ if $\tx_i\in\tU_i\subseteq U$. In the following
constructions we will tacictly assume that if we take a generic point $\tx_i$, then this point belongs to
$\tU_i$.\\

Now the idea is that whenever we have $U$ defined in this way, we would like to recover both the underlying
topological space $X$ (up to homemorphisms, see proposition \ref{quasi-injectivity} below) and the atlas 
$\mathcal{U}$; how to do this must be encoded in $R$, that we are going to define. If we want to recover $X$
\emph{set-theoretically}, we have to identify on $U$ any two pair of points $(\tx_i,\tU_i)$ and $(\tx_j,\tU_j)$ such that
$\pi_i(\tx_i)=\pi_j(\tx_i)=:x$ on $X$. Now let us suppose that we have fixed such a pair of points and let us apply
remark \ref{useful-remark}; so we get that there exists an open set $U_k\subseteq U_i\cap U_j$ on $X$ which
contains $x$, a uniformizing system $\unif{k}\in\mathcal{U}$ for it, a point $\tx_k\in\tU_k$ and embeddings:

\begin{equation}\label{eq-28}
\unif{i}\sinistra{\lambda_{ki}}\unif{k}\fre{\lambda_{kj}}\unif{j};
\end{equation}

such that:

\begin{equation}\label{eq-29}
\LA{ki}(\tx_k)=\tx_i\AND\LA{kj}(\tx_k)=\tx_j.
\end{equation}

Conversely, if there exists a pair of embeddings (\ref{eq-28}) such that (\ref{eq-29}) holds, we get immediately
that $(\tx_i,\tU_i)$ and $(\tx_j,\tU_j)$ will be identified in $X$. So the first idea could be to define a set of
``relations'' $\hat{R}$ whose elements are of the form $(\LA{ki},\tx_k,\LA{kj})$, as follows:

$$\hat{R}:=\coprod\tU_k^{ij}$$

where $\tU_k^{ij}:=\tU_k^{\LA{ki},\LA{kj}}$ denotes a copy of $\tU_k$ indexed by a pair of embeddings 
$\LA{ki},\LA{kj}$, and where the disjoint
union is taken \emph{over all triples of uniformizing systems of the form }(\ref{eq-28}). In
other words, the disjoint union is taken over all the uniformizing systems $\unif{k}\in\mathcal{U}$ and over all
the possible embeddings of them into any pair of uniformizing systems as in (\ref{eq-28}); hence 
$(\LA{ki},\tx_k,\LA{kj})$ just means the point $\tx_k\in\tU_k$ considered as in the set $\tU_k^{ij}
\subseteq\hat{R}$. Now we can easily define 4 of the 5 morphisms of a groupoid as follows:

\begin{eqnarray*}
& \hat{s}(\LA{ki},\tx_k,\LA{kj}):=(\LA{ki}(\tx_k),\tU_i),\quad\quad
   \hat{t}(\LA{ki},\tx_k,\LA{kj}):=(\LA{kj}(\tx_k),\tU_j), &\\
& \hat{e}(\tx_i,\tU_i):=(1_{\tU_i},\tx_i,1_{\tU_i}),\quad\quad
   \hat{i}(\LA{ki},\tx_k,\LA{kj}):=(\LA{kj},\tx_k,\LA{ki}).&
\end{eqnarray*}

$\hat{R}$ is \emph{almost} the manifold we want to use; the only problem is that we can't have a well defined
notion of multiplication $m$ on it, so we have to define the set of relations $R$ as a quotient of $\hat{R}$ by
a locally trivial relation of equivalence as follows. The necessity of such a relation will be clear in
the proof of lemma \ref{m-well-defined}.

\begin{defin}\label{equivalence}(\cite{Pr}, \S 4.4)
Two points $(\LA{ki},\tx_k,\LA{kj})$ and $(\LA{li},\tx_l,\LA{lj})$ in $\hat{R}$ are said to be \emph{equivalent}
and we write:

\begin{equation}\label{eq-30}(\LA{ki},\tx_k,\LA{kj})\sim(\LA{li},\tx_l,\LA{lj})\end{equation}

iff there exist a uniformizing system $\unif{m}\in\mathcal{U}$, a point $\tx_m\in\tU_m$ and two embeddings:

$$\unif{k}\sinistra{\LA{mk}}\unif{m}\fre{\LA{ml}}\unif{l}$$

such that the following diagrams are commutative:

\begin{equation}\label{eq-31}
\begin{tikzpicture}[scale=0.8]
    \def\x{1.5}
    \def\y{-1.2}
    \node (A2_0) at (2*\x, 0*\y) {$\tU_k$};
    \node (A0_2) at (0*\x, 2*\y) {$\tU_i$};
    \node (A2_2) at (2*\x, 2*\y) {$\tU_m$};
    \node (A4_2) at (4*\x, 2*\y) {$\tU_j$};
    \node (A2_4) at (2*\x, 4*\y) {$\tU_l$};
    \node (A1_2) at (1*\x, 2*\y) {$\curvearrowright$};
    \node (A3_2) at (3*\x, 2*\y) {$\curvearrowright$};
    \path (A2_0) edge [->] node [auto,swap] {$\scriptstyle{\LA{ki}}$} (A0_2);
    \path (A2_0) edge [->] node [auto] {$\scriptstyle{\LA{kj}}$} (A4_2);
    \path (A2_4) edge [->] node [auto] {$\scriptstyle{\LA{li}}$} (A0_2);
    \path (A2_4) edge [->] node [auto,swap] {$\scriptstyle{\LA{lj}}$} (A4_2);
    \path (A2_2) edge [->] node [auto,swap] {$\scriptstyle{\LA{mk}}$} (A2_0);
    \path (A2_2) edge [->] node [auto] {$\scriptstyle{\LA{ml}}$} (A2_4);

    \def\z{7}
    \node (B2_0) at (2*\x+\z, 0*\y) {$\tx_k$};
    \node (B0_2) at (0*\x+\z, 2*\y) {$\tx_i$};
    \node (B2_2) at (2*\x+\z, 2*\y) {$\tx_m$};
    \node (B4_2) at (4*\x+\z, 2*\y) {$\tx_j.$};
    \node (B2_4) at (2*\x+\z, 4*\y) {$\tx_l$};
    \node (B1_2) at (1*\x+\z, 2*\y) {$\curvearrowright$};
    \node (B3_2) at (3*\x+\z, 2*\y) {$\curvearrowright$};
    \path (B2_0) edge [->] node [auto,swap] {$\scriptstyle{\LA{ki}}$} (B0_2);
    \path (B2_0) edge [->] node [auto] {$\scriptstyle{\LA{kj}}$} (B4_2);
    \path (B2_4) edge [->] node [auto] {$\scriptstyle{\LA{li}}$} (B0_2);
    \path (B2_4) edge [->] node [auto,swap] {$\scriptstyle{\LA{lj}}$} (B4_2);
    \path (B2_2) edge [->] node [auto,swap] {$\scriptstyle{\LA{mk}}$} (B2_0);
    \path (B2_2) edge [->] node [auto] {$\scriptstyle{\LA{ml}}$} (B2_4);
\end{tikzpicture}\end{equation}
\end{defin}

\begin{rem}
In order to simplify the notations, \emph{here and from now on we omit the groups }$G_i$\emph{'s and the maps}
$\pi_i$\emph{'s} in every diagram; in other words from now on every map $\LA{ij}$ will be an embedding between
uniformizing systems even if we write only its source and target as open sets of $\mathbb{C}^n$ and not as
uniformizing systems.
\end{rem}

\emph{Now our aim is to prove that} (\ref{eq-30}) \emph{is an equivalence relation on} $R'$. In order to do that,
let us state and prove the following lemma:

\begin{lem}\label{useful-lemma}
Let us fix an atlas $\mathcal{U}$, a pair of embeddings $\LA{nl}:\unif{n}\rightarrow\unif{l}$, $\LA{pl}:\unif{p}
\rightarrow\unif{l}$ and a pair of points $\tx_n\in\tU_n$, $\tx_p\in\tU_p$ such that $\LA{nl}(\tx_n)=\LA{pl}
(\tx_p)=:\tx_l$. Then there exist a uniformizing system $\unif{q}\in\mathcal{U}$, a pair of embeddings $\LA{qn},
\LA{qp}$ and a point $\tx_q\in\tU_q$ such that the following diagrams are commutative:

\begin{equation}\label{eq-32}
\begin{tikzpicture}[scale=0.8]
    \def\x{1.5}
    \def\y{-1.5}
    \node (A0_1) at (1*\x, 0*\y) {$\tU_q$};
    \node (A1_0) at (0*\x, 1*\y) {$\tU_n$};
    \node (A1_1) at (1*\x, 1*\y) {$\curvearrowright$};
    \node (A1_2) at (2*\x, 1*\y) {$\tU_p$};
    \node (A2_1) at (1*\x, 2*\y) {$\tU_l$};
    \path (A1_2) edge [->] node [auto] {$\scriptstyle{\LA{pl}}$} (A2_1);
    \path (A1_0) edge [->] node [auto,swap] {$\scriptstyle{\LA{nl}}$} (A2_1);
    \path (A0_1) edge [->,dashed] node [auto] {$\scriptstyle{\LA{qp}}$} (A1_2);
    \path (A0_1) edge [->,dashed] node [auto,swap] {$\scriptstyle{\LA{qn}}$} (A1_0);

    \def\z{4.8}
    \node (B0_1) at (1*\x+\z, 0*\y) {$\tx_q$};
    \node (B1_0) at (0*\x+\z, 1*\y) {$\tx_n$};
    \node (B1_1) at (1*\x+\z, 1*\y) {$\curvearrowright$};
    \node (B1_2) at (2*\x+\z, 1*\y) {$\tx_p$};
    \node (B2_1) at (1*\x+\z, 2*\y) {$\tx_l$};
    \path (B1_2) edge [->] node [auto] {$\scriptstyle{\LA{pl}}$} (B2_1);
    \path (B1_0) edge [->] node [auto,swap] {$\scriptstyle{\LA{nl}}$} (B2_1);
    \path (B0_1) edge [->,dashed] node [auto] {$\scriptstyle{\LA{qp}}$} (B1_2);
    \path (B0_1) edge [->,dashed] node [auto,swap] {$\scriptstyle{\LA{qn}}$} (B1_0);
\end{tikzpicture}
\end{equation}
\end{lem}

\begin{proof}(\cite{Pr}, lemma 4.4.1 with some changes)
By hypothesis $\LA{nl}(\tx_n)=\LA{pl}(\tx_p)$, hence $\pi_n(\tx_n)=\pi_p(\tx_p)$, so using remark
\ref{useful-remark} we get that there exists a uniformizing system $\unif{q}\in\mathcal{U}$, embeddings
$\tilde{\lambda}_{qn}$ and $\tilde{\lambda}_{qp}$ of it into $\unif{n}$ and $\unif{p}$ respectively and a point
$\tx_q\in\tU_q$ such that $\tilde{\lambda}_{qn}(\tx_q)=\tx_n$ and $\tilde{\lambda}_{qp}(\tx_q)=\tx_p$. Now let us
consider the embeddings:

$$\alpha:=\LA{nl}\circ\tilde{\lambda}_{qn}\AND\beta:=\LA{pl}\circ\tilde{\lambda}_{qp}$$

both defined from $\unif{q}$ to $\unif{l}$. Using lemma \ref{lemma-moerdijk} we get that there exists a unique
$g\in G_l$ such that $g\circ\alpha=\beta$. Now $\alpha(\tU_q)\cap g\circ\alpha(\tU_q)\neq\varnothing$ since it
contains $\tx_l$, hence we can apply lemma  \ref{lemma-intersection}, so there exists a unique $h\in G_q$
(actually, it belongs to the stabilizer of $\tx_q$) such that $\Lambda_{nl}\circ\tilde{\Lambda}_{qn}(h)=g$. If we
define $\LA{qn}:=\tilde{\lambda}_{qn}\circ h$ and $\LA{qp}:=\tilde{\lambda}_{qp}$, then a direct check proves that
these embeddings make (\ref{eq-32}) commute.
\end{proof}

\begin{lem}\label{equivalence-proof}
$\sim$ is an equivalence relation on $\hat{R}$.
\end{lem}

\begin{proof}
(\cite{Pr}, lemma 4.4.2) The relation is clearly reflexive and symmetric. In order to prove transitivity, let us
suppose we have:

$$(\LA{ki},\tx_k,\LA{kj})\sim(\LA{li},\tx_l,\LA{lj})\sim(\LA{mi},\tx_m,\LA{mj}).$$

In other words, using definition \ref{equivalence}, there exist uniformizing systems $\unif{n}$ and $\unif{p}$ in
$\mathcal{U}$, two points $\tx_n\in\tU_n$, $\tx_p\in\tU_p$ and embeddings $\LA{nk},\LA{nl}$, $\LA{pl},\LA{pm}$
making the following diagrams commute:

\begin{equation}\label{eq-33}
\begin{tikzpicture}[scale=0.8]
    \def\x{1.5}
    \def\y{-1.7}
    \node (A0_2) at (2*\x, 0*\y) {$\tU_k$};
    \node (A1_1) at (1.3*\x, 1.3*\y) {$\curvearrowright$};
    \node (A1_2) at (2*\x, 1*\y) {$\tU_n$};
    \node (A1_3) at (2.7*\x, 1.3*\y) {$\curvearrowright$};
    \node (A2_0) at (0*\x, 2*\y) {$\tU_i$};
    \node (A2_2) at (2*\x, 2*\y) {$\tU_l$};
    \node (A2_4) at (4*\x, 2*\y) {$\tU_j$};
    \node (A3_1) at (1.3*\x, 2.7*\y) {$\curvearrowright$};
    \node (A3_2) at (2*\x, 3*\y) {$\tU_p$};
    \node (A3_3) at (2.7*\x, 2.7*\y) {$\curvearrowright$};
    \node (A4_2) at (2*\x, 4*\y) {$\tU_m$};

    \node (C_1) at (2.3*\x, 0.7*\y) {$\scriptstyle{\LA{nk}}$};
    \node (C_2) at (2.3*\x, 3.3*\y) {$\scriptstyle{\LA{pm}}$};

    \path (A0_2) edge [->] node [auto,swap] {$\scriptstyle{\LA{ki}}$} (A2_0);
    \path (A0_2) edge [->] node [auto] {$\scriptstyle{\LA{kj}}$} (A2_4);
    \path (A4_2) edge [->] node [auto,swap] {$\scriptstyle{\LA{mj}}$} (A2_4);
    \path (A2_2) edge [->] node [auto] {$\scriptstyle{\LA{lj}}$} (A2_4);
    \path (A2_2) edge [->] node [auto,swap] {$\scriptstyle{\LA{li}}$} (A2_0);
    \path (A3_2) edge [->] node [auto,swap] {$\scriptstyle{\LA{pl}}$} (A2_2);
    \path (A3_2) edge [->] node [auto] {} (A4_2);
    \path (A1_2) edge [->] node [auto,swap] {} (A0_2);
    \path (A1_2) edge [->] node [auto] {$\scriptstyle{\LA{nl}}$} (A2_2);
    \path (A4_2) edge [->] node [auto] {$\scriptstyle{\LA{mi}}$} (A2_0);

    \def\z{7}
    \node (B0_2) at (2*\x+\z, 0*\y) {$\tx_k$};
    \node (B1_1) at (1.3*\x+\z, 1.3*\y) {$\curvearrowright$};
    \node (B1_2) at (2*\x+\z, 1*\y) {$\tx_n$};
    \node (B1_3) at (2.7*\x+\z, 1.3*\y) {$\curvearrowright$};
    \node (B2_0) at (0*\x+\z, 2*\y) {$\tx_i$};
    \node (B2_2) at (2*\x+\z, 2*\y) {$\tx_l$};
    \node (B2_4) at (4*\x+\z, 2*\y) {$\tx_j.$};
    \node (B3_1) at (1.3*\x+\z, 2.7*\y) {$\curvearrowright$};
    \node (B3_2) at (2*\x+\z, 3*\y) {$\tx_p$};
    \node (B3_3) at (2.7*\x+\z, 2.7*\y) {$\curvearrowright$};
    \node (B4_2) at (2*\x+\z, 4*\y) {$\tx_m$};

    \node (D_1) at (2.3*\x+\z, 0.7*\y) {$\scriptstyle{\LA{nk}}$};
    \node (D_2) at (2.3*\x+\z, 3.3*\y) {$\scriptstyle{\LA{pm}}$};

    \path (B0_2) edge [->] node [auto,swap] {$\scriptstyle{\LA{ki}}$} (B2_0);
    \path (B0_2) edge [->] node [auto] {$\scriptstyle{\LA{kj}}$} (B2_4);
    \path (B4_2) edge [->] node [auto,swap] {$\scriptstyle{\LA{mj}}$} (B2_4);
    \path (B2_2) edge [->] node [auto] {$\scriptstyle{\LA{lj}}$} (B2_4);
    \path (B2_2) edge [->] node [auto,swap] {$\scriptstyle{\LA{li}}$} (B2_0);
    \path (B3_2) edge [->] node [auto,swap] {$\scriptstyle{\LA{pl}}$} (B2_2);
    \path (B3_2) edge [->] node [auto] {} (B4_2);
    \path (B1_2) edge [->] node [auto,swap] {} (B0_2);
    \path (B1_2) edge [->] node [auto] {$\scriptstyle{\LA{nl}}$} (B2_2);
    \path (B4_2) edge [->] node [auto] {$\scriptstyle{\LA{mi}}$} (B2_0);
\end{tikzpicture}
\end{equation}

In particular, $\LA{nl}(\tx_n)=\LA{pl}(\tx_p)$, so we can apply lemma \ref{useful-lemma} and we get that there
exist a uniformizing system $\unif{q}$, a point $\tx_q\in\tU_q$ and a pair of embeddings $\LA{qn},\LA{qp}$ as in
(\ref{eq-32}). Now using (\ref{eq-32}) and (\ref{eq-33}) a direct computation proves that we have
commutative diagrams of the form:

\begin{equation}
\begin{tikzpicture}[scale=0.8]
    \def\x{1.5}
    \def\y{-1.2}
    \def\z{7.5}
    \node (A2_0) at (2*\x, 0*\y) {$\tU_k$};
    \node (A0_2) at (0*\x, 2*\y) {$\tU_i$};
    \node (A2_2) at (2*\x, 2*\y) {$\tU_q$};
    \node (A4_2) at (4*\x, 2*\y) {$\tU_j$};
    \node (A2_4) at (2*\x, 4*\y) {$\tU_m$};
    \node (A1_2) at (1*\x, 2*\y) {$\curvearrowright$};
    \node (A3_2) at (3*\x, 2*\y) {$\curvearrowright$};
    \node (C_1) at (2.5*\x, 1.3*\y) {$\scriptstyle{\LA{nk}\circ\LA{qn}}$};
    \node (C_2) at (2.5*\x, 2.7*\y) {$\scriptstyle{\LA{pm}\circ\LA{qp}}$};

    \path (A2_0) edge [->] node [auto,swap] {$\scriptstyle{\LA{ki}}$} (A0_2);
    \path (A2_0) edge [->] node [auto] {$\scriptstyle{\LA{kj}}$} (A4_2);
    \path (A2_4) edge [->] node [auto] {$\scriptstyle{\LA{mi}}$} (A0_2);
    \path (A2_4) edge [->] node [auto,swap] {$\scriptstyle{\LA{mj}}$} (A4_2);
    \path (A2_2) edge [->] node [auto,swap] {} (A2_0);
    \path (A2_2) edge [->] node [auto] {} (A2_4);

    \node (B2_0) at (2*\x+\z, 0*\y) {$\tx_k$};
    \node (B0_2) at (0*\x+\z, 2*\y) {$\tx_i$};
    \node (B2_2) at (2*\x+\z, 2*\y) {$\tx_q$};
    \node (B4_2) at (4*\x+\z, 2*\y) {$\tx_j.$};
    \node (B2_4) at (2*\x+\z, 4*\y) {$\tx_m$};
    \node (B1_2) at (1*\x+\z, 2*\y) {$\curvearrowright$};
    \node (B3_2) at (3*\x+\z, 2*\y) {$\curvearrowright$};
    \node (C_3) at (2.5*\x+\z, 1.3*\y) {$\scriptstyle{\LA{nk}\circ\LA{qn}}$};
    \node (C_4) at (2.5*\x+\z, 2.7*\y) {$\scriptstyle{\LA{pm}\circ\LA{qp}}$};

    \path (B2_0) edge [->] node [auto,swap] {$\scriptstyle{\LA{ki}}$} (B0_2);
    \path (B2_0) edge [->] node [auto] {$\scriptstyle{\LA{kj}}$} (B4_2);
    \path (B2_4) edge [->] node [auto] {$\scriptstyle{\LA{mi}}$} (B0_2);
    \path (B2_4) edge [->] node [auto,swap] {$\scriptstyle{\LA{mj}}$} (B4_2);
    \path (B2_2) edge [->] node [auto,swap] {} (B2_0);
    \path (B2_2) edge [->] node [auto] {} (B2_4);
\end{tikzpicture}
\end{equation}

So we have proved that $(\LA{ki},\tx_k,\LA{kj})\sim(\LA{mi},\tx_m,\LA{mj})$,
hence $\sim$ is transitive.
\end{proof}

\begin{defin}
For every orbifold atlas $\mathcal{U}$ over $X$, we define the set $R:=\hat{R}/\sim$ and we call $q:\hat{R}\rightarrow
R$ the quotient map; we will denote the class of any point $(\LA{ki},\tx_k,\LA{kj})$ with $[\LA{ki},\tx_k,\LA{kj}]$.
\end{defin}

\begin{rem}\label{good-definitions-on-R}
The equivalence relation we have just described is the same as the relation $\sim$ defined in \cite{MP2} as the 
relation \emph{generated} by considering equivalent $(\LA{ki},\tx_k,\LA{kj})$ and $(\LA{li},\tx_l,\LA{lj})$
whenever there exists an embedding $\LA{lk}$ such that the
following diagrams are commutative:

\begin{equation}\label{eq-38}
\begin{tikzpicture}[scale=0.8]
    \def\x{1.5}
    \def\y{-1.2}
    \node (A0_2) at (2*\x, 0*\y) {$\tU_l$};
    \node (A1_1) at (1.4*\x, 1.2*\y) {$\curvearrowright$};
    \node (A1_3) at (2.6*\x, 1.2*\y) {$\curvearrowright$};
    \node (A2_0) at (0*\x, 2*\y) {$\tU_i$};
    \node (A2_2) at (2*\x, 2*\y) {$\tU_k$};
    \node (A2_4) at (4*\x, 2*\y) {$\tU_j$};
    \path (A0_2) edge [->] node [auto,swap] {$\scriptstyle{\LA{li}}$} (A2_0);
    \path (A2_2) edge [->] node [auto,swap] {$\scriptstyle{\LA{kj}}$} (A2_4);
    \path (A0_2) edge [->] node [auto] {$\scriptstyle{\LA{lk}}$} (A2_2);
    \path (A0_2) edge [->] node [auto] {$\scriptstyle{\LA{lj}}$} (A2_4);
    \path (A2_2) edge [->] node [auto] {$\scriptstyle{\LA{ki}}$} (A2_0);

    \def\z{7}
    \node (B0_2) at (2*\x+\z, 0*\y) {$\tx_l$};
    \node (B1_1) at (1.4*\x+\z, 1.2*\y) {$\curvearrowright$};
    \node (B1_3) at (2.6*\x+\z, 1.2*\y) {$\curvearrowright$};
    \node (B2_0) at (0*\x+\z, 2*\y) {$\tx_i$};
    \node (B2_2) at (2*\x+\z, 2*\y) {$\tx_k$};
    \node (B2_4) at (4*\x+\z, 2*\y) {$\tx_j.$};
    \path (B0_2) edge [->] node [auto,swap] {$\scriptstyle{\LA{li}}$} (B2_0);
    \path (B2_2) edge [->] node [auto,swap] {$\scriptstyle{\LA{kj}}$} (B2_4);
    \path (B0_2) edge [->] node [auto] {$\scriptstyle{\LA{lk}}$} (B2_2);
    \path (B0_2) edge [->] node [auto] {$\scriptstyle{\LA{lj}}$} (B2_4);
    \path (B2_2) edge [->] node [auto] {$\scriptstyle{\LA{ki}}$} (B2_0);
\end{tikzpicture}
\end{equation}

Hence from now on we will use without distinction the first and the second equivalence; in the following
pages we will often have to define set maps on $R$ using representatives of the
equivalence classes; in order to prove that they are well defined it will be sufficient to use different
representatives related by a diagram of the form (\ref{eq-38}).
\end{rem}

Now we recall that our first aim is to make the pair $(R,U)$ into a groupoid object in $(\textbf{Sets})$,
so for the moment we don't care about the topology on $R$ and we consider it just as a set.  The maps $s,t,i,e$
are just induced by the maps $\hat{s},\hat{t},\hat{i},\hat{e}$ respectively; the fact that the first 3 maps does
not depend on the representatives chosen is just a straightforward application of
remark \ref{good-definitions-on-R}. Now we want to define the ``multiplication'' on
$R$, so let us consider any pair of ``composable arrows'' $[\LA{ih},\tx_i,\LA{ij}]$ and $[\LA{kj},\tx_k,\LA{kl}]$
in the set-theoretical fiber product  $\fibre{R}{t}{s}{R}$. In other words, let us assume that $t([\LA{ih},\tx_i,
\LA{ij}])=s([\LA{kj},\tx_k,\LA{kl}])$ i.e. $\LA{ij}(\tx_i)=\LA{kj}(\tx_k)$; equivalently, we have a diagram of
uniformizing systems and marked points as follows:

\begin{equation}\label{eq-40}
\begin{tikzpicture}[scale=0.8]
    \def\x{2.5}
    \def\y{-0.55}
    \node (A0_1) at (1*\x, 0*\y) {$\tx_i$};
    \node (A0_2) at (2*\x, 0*\y) {$\LA{ij}(\tx_i)=\LA{kj}(\tx_k)$};
    \node (A0_3) at (3*\x, 0*\y) {$\tx_k$};
    \node[rotate=270] (A1_1) at (1*\x, 1*\y) {$\in$};
    \node[rotate=270] (A1_2) at (2*\x, 1*\y) {$\in$};
    \node[rotate=270] (A1_3) at (3*\x, 1*\y) {$\in$};
    \node (A2_0) at (0*\x, 2*\y) {$\tU_h$};
    \node (A2_1) at (1*\x, 2*\y) {$\tU_i$};
    \node (A2_2) at (2*\x, 2*\y) {$\tU_j$};
    \node (A2_3) at (3*\x, 2*\y) {$\tU_k$};
    \node (A2_4) at (4*\x, 2*\y) {$\tU_l.$};
    \path (A2_3) edge [->] node [auto] {$\scriptstyle{\LA{kl}}$} (A2_4);
    \path (A2_1) edge [->] node [auto,swap] {$\scriptstyle{\LA{ih}}$} (A2_0);
    \path (A2_1) edge [->] node [auto] {$\scriptstyle{\LA{ij}}$} (A2_2);
    \path (A2_3) edge [->] node [auto,swap] {$\scriptstyle{\LA{kj}}$} (A2_2);
\end{tikzpicture}
\end{equation}

Since $\LA{ij}(\tx_i)=\LA{kj}(\tx_k)$, we can apply lemma \ref{useful-lemma}, so we get that there exist a
uniformizing system $\unif{f}$, a point $\tx_f\in\tU_f$ and embeddings $\LA{fi},\LA{fk}$ such that we can complete
(\ref{eq-40}) to a commutative diagram (also on the level of marked points):

\begin{equation}\label{eq-41}
\begin{tikzpicture}[scale=0.8]
    \def\x{3}
    \def\y{-0.55}
    \node (A0_1) at (1*\x, 0*\y) {$\tx_i$};
    \node (A0_2) at (2*\x, 0*\y) {$\LA{ij}(\tx_i)=\LA{kj}(\tx_k)$};
    \node (A0_3) at (3*\x, 0*\y) {$\tx_k$};
    \node[rotate=270] (A1_1) at (1*\x, 1*\y) {$\in$};
    \node[rotate=270] (A1_2) at (2*\x, 1*\y) {$\in$};
    \node[rotate=270] (A1_3) at (3*\x, 1*\y) {$\in$};
    \node (A2_0) at (0*\x, 2*\y) {$\tU_h$};
    \node (A2_1) at (1*\x, 2*\y) {$\tU_i$};
    \node (A2_2) at (2*\x, 2*\y) {$\tU_j$};
    \node (A2_3) at (3*\x, 2*\y) {$\tU_k$};
    \node (A2_4) at (4*\x, 2*\y) {$\tU_l.$};
    \node (B1_1) at (2*\x, 4*\y) {$\curvearrowright$};
    \node (A3_2) at (2*\x, 2*\y-2) {$\tU_f$};
    \node[rotate=90] (A4_2) at (2*\x, 3*\y-2) {$\in$};
    \node (A5_2) at (2*\x, 4*\y-2) {$\tx_f$};

    \path (A2_3) edge [->] node [auto] {$\scriptstyle{\LA{kl}}$} (A2_4);
    \path (A2_1) edge [->] node [auto,swap] {$\scriptstyle{\LA{ih}}$} (A2_0);
    \path (A2_1) edge [->] node [auto] {$\scriptstyle{\LA{ij}}$} (A2_2);
    \path (A2_3) edge [->] node [auto,swap] {$\scriptstyle{\LA{kj}}$} (A2_2);
    \path (A3_2) edge [->] node [auto] {$\scriptstyle{\LA{fi}}$} (A2_1);
    \path (A3_2) edge [->] node [auto,swap] {$\scriptstyle{\LA{fk}}$} (A2_3);
\end{tikzpicture}
\end{equation}

Hence we give the following definition:

$$m\Big([\LA{ih},\tx_i,\LA{ij}],[\LA{kj},\tx_k,\LA{kl}]\Big):=[\LA{ih}\circ\LA{fi},\tx_f,\LA{kl}\circ\LA{fk}].$$

\begin{lem}\label{m-well-defined}
The map $m$ is well defined (see the appendix for the proof).
\end{lem}

\begin{prop}
$(\groupR)$ is a groupoid object in $\CAT{Sets}$.
\end{prop}

\begin{proof} We have to prove that all the axioms for a groupoid object given in definition \ref{groupoid-object}
are satisfied; the only non-trivial one is the third one, so let us fix any triple of composable arrows: 

$$\Big([\LA{ih},\tx_i,\LA{ij}],[\LA{kj},\tx_k,\LA{kl}],[\LA{ml},\tx_m,\LA{mn}]\Big)\in
R\,_t\times_s R\,_t\times_s R$$

i.e. $\LA{ij}(\tx_i)=\LA{kj}(\tx_k)$ and $\LA{kl}(\tx_k)=\LA{ml}(\tx_m)$. If we apply lemma \ref{useful-lemma}
twice, we get a commutative diagram (of uniformizing systems and marked points) as follows:

\[\begin{tikzpicture}[scale=0.8]
    \def\x{2.0}
    \def\y{-0.55}
    \def\z{-0.6}
    \def\w{-0.2}
    \def\o{-3.6}
    \def\a{-4.0}
    \node (A2_0) at (0*\x, 2*\y) {$\tU_h$};
    \node (A2_1) at (1*\x, 2*\y) {$\tU_i$};
    \node (A2_2) at (2*\x, 2*\y) {$\tU_j$};
    \node (A2_3) at (3*\x, 2*\y) {$\tU_k$};
    \node (A2_4) at (4*\x, 2*\y) {$\tU_l$};
    \node (A2_5) at (5*\x, 2*\y) {$\tU_m$};
    \node (A2_6) at (6*\x, 2*\y) {$\tU_n.$};
    \node (A3_2) at (2*\x, 2*\y-2) {$\tU_f$};
    \node (A3_4) at (4*\x, 2*\y-2) {$\tU_s$};
    \node (C1_1) at (2*\x, 2*\y-1) {$\curvearrowright$};
    \node (C1_2) at (4*\x, 2*\y-1) {$\curvearrowright$};
    \node[rotate=270] (E0_1) at (1*\x, \z) {$\in$};
    \node (E5_2) at (1*\x, \w) {$\tx_i$};
    \node[rotate=270] (E0_3) at (3*\x, \z) {$\in$};
    \node (E5_2) at (3*\x, \w) {$\tx_k$};
    \node[rotate=270] (E0_5) at (5*\x, \z) {$\in$};
    \node (E5_2) at (5*\x, \w) {$\tx_m$};
    \node[rotate=90] (F0_1) at (2*\x, \o) {$\in$};
    \node (F5_2) at (2*\x, \a) {$\tx_f$};
    \node[rotate=90] (F0_3) at (4*\x, \o) {$\in$};
    \node (F5_2) at (4*\x, \a) {$\tx_s$};
 
    \path (A2_3) edge [->] node [auto] {$\scriptstyle{\LA{kl}}$} (A2_4);
    \path (A2_1) edge [->] node [auto,swap] {$\scriptstyle{\LA{ih}}$} (A2_0);
    \path (A2_1) edge [->] node [auto] {$\scriptstyle{\LA{ij}}$} (A2_2);
    \path (A2_3) edge [->] node [auto,swap] {$\scriptstyle{\LA{kj}}$} (A2_2);
    \path (A3_2) edge [->] node [auto] {$\scriptstyle{\LA{fi}}$} (A2_1);
    \path (A3_2) edge [->] node [auto,swap] {$\scriptstyle{\LA{fk}}$} (A2_3);
    \path (A3_4) edge [->] node [auto] {$\scriptstyle{\LA{sk}}$} (A2_3);
    \path (A3_4) edge [->] node [auto,swap] {$\scriptstyle{\LA{sm}}$} (A2_5);
    \path (A2_5) edge [->] node [auto] {$\scriptstyle{\LA{mn}}$} (A2_6);
    \path (A2_5) edge [->] node [auto,swap] {$\scriptstyle{\LA{ml}}$} (A2_4);
\end{tikzpicture}\]

If we consider the central part of it, we can apply again lemma \ref{useful-lemma} in order to get a commutative
diagram of the form:

\begin{equation}\label{eq-46}
\begin{tikzpicture}[scale=0.8]
    \def\x{2.0}
    \def\y{-0.55}
    \node (A2_0) at (0*\x, 2*\y) {$\tU_h$};
    \node (A2_1) at (1*\x, 2*\y) {$\tU_i$};
    \node (A2_2) at (2*\x, 2*\y) {$\tU_j$};
    \node (A2_3) at (3*\x, 2*\y) {$\tU_k$};
    \node (A2_4) at (4*\x, 2*\y) {$\tU_l$};
    \node (A2_5) at (5*\x, 2*\y) {$\tU_m$};
    \node (A2_6) at (6*\x, 2*\y) {$\tU_n.$};
    \node (A3_2) at (2*\x, 2*\y-2) {$\tU_f$};
    \node (A3_4) at (4*\x, 2*\y-2) {$\tU_s$};
    \node (H1_1) at (3*\x, -4.9) {$\tU_r$};
    \node (C1_1) at (2*\x, 2*\y-1) {$\curvearrowright$};
    \node (C1_2) at (4*\x, 2*\y-1) {$\curvearrowright$};
    \node (G1_1) at (3*\x, -2.9) {$\curvearrowright$};

    \def\z{-0.6}
    \def\w{-0.2}
    \node[rotate=270] (E0_1) at (1*\x, \z) {$\in$};
    \node (E5_2) at (1*\x, \w) {$\tx_i$};
    \node[rotate=270] (E0_3) at (3*\x, \z) {$\in$};
    \node (E5_2) at (3*\x, \w) {$\tx_k$};
    \node[rotate=270] (E0_5) at (5*\x, \z) {$\in$};
    \node (E5_2) at (5*\x, \w) {$\tx_m$};

    \def\o{-3.8}
    \def\a{-3.45}
    \node[rotate=45] (F0_1) at (1.8*\x, \a) {$\in$};
    \node (F5_2) at (1.7*\x, \o) {$\tx_f$};
    \node[rotate=135] (F0_3) at (4.2*\x, \a) {$\in$};
    \node (F5_2) at (4.4*\x, \o) {$\tx_s$};

    \def\b{-5.4}
    \def\d{-5.8}
    \node[rotate=90] (M0_1) at (3*\x, \b) {$\in$};
    \node (M5_2) at (3*\x, \d) {$\tx_r$};
 
    \path (A2_3) edge [->] node [auto] {$\scriptstyle{\LA{kl}}$} (A2_4);
    \path (A2_1) edge [->] node [auto,swap] {$\scriptstyle{\LA{ih}}$} (A2_0);
    \path (A2_1) edge [->] node [auto] {$\scriptstyle{\LA{ij}}$} (A2_2);
    \path (A2_3) edge [->] node [auto,swap] {$\scriptstyle{\LA{kj}}$} (A2_2);
    \path (A3_2) edge [->] node [auto] {$\scriptstyle{\LA{fi}}$} (A2_1);
    \path (A3_2) edge [->] node [auto,swap] {$\scriptstyle{\LA{fk}}$} (A2_3);
    \path (A3_4) edge [->] node [auto] {$\scriptstyle{\LA{sk}}$} (A2_3);
    \path (A3_4) edge [->] node [auto,swap] {$\scriptstyle{\LA{sm}}$} (A2_5);
    \path (A2_5) edge [->] node [auto] {$\scriptstyle{\LA{mn}}$} (A2_6);
    \path (A2_5) edge [->] node [auto,swap] {$\scriptstyle{\LA{ml}}$} (A2_4);
    \path (H1_1) edge [->] node [auto] {$\scriptstyle{\LA{rf}}$} (A3_2);
    \path (H1_1) edge [->] node [auto,swap] {$\scriptstyle{\LA{rs}}$} (A3_4);
\end{tikzpicture}
\end{equation}

Now:

\begin{eqnarray*}
& m\Big([\LA{kj},\tx_k,\LA{kl}],[\LA{ml},\tx_m,\LA{mn}]\Big)=[\LA{kj}\circ\LA{sk}\circ\LA{rs},\tx_r,
   \LA{mn}\circ\LA{sm}\circ\LA{rs}]= &\\
& =[\LA{kj}\circ\LA{fk}\circ\LA{rf},\tx_r,\LA{mn}\circ\LA{sm}\circ\LA{rs}]; &
\end{eqnarray*}

hence using the commutative diagram:

\[\begin{tikzpicture}[scale=0.8]
    \def\x{3}
    \def\y{-0.55}
    \node (A0_1) at (1*\x, 0*\y) {$\tx_i$};
    \node (A0_2) at (2*\x, 0*\y) {$\LA{ij}(\tx_i)=\LA{kj}(\tx_k)$};
    \node (A0_3) at (3*\x, 0*\y) {$\tx_r$};
    \node[rotate=270] (A1_1) at (1*\x, 1*\y) {$\in$};
    \node[rotate=270] (A1_2) at (2*\x, 1*\y) {$\in$};
    \node[rotate=270] (A1_3) at (3*\x, 1*\y) {$\in$};
    \node (A2_0) at (0*\x, 2*\y) {$\tU_h$};
    \node (A2_1) at (1*\x, 2*\y) {$\tU_i$};
    \node (A2_2) at (2*\x, 2*\y) {$\tU_j$};
    \node (A2_3) at (3*\x, 2*\y) {$\tU_r$};
    \node (A2_4) at (4*\x, 2*\y) {$\tU_n.$};
    \node (B1_1) at (2*\x, 4*\y) {$\curvearrowright$};
    \def\z{-2}
    \node (A3_2) at (2*\x, 2*\y+\z) {$\tU_r$};
    \node[rotate=90] (A4_2) at (2*\x, 3*\y+\z) {$\in$};
    \node (A5_2) at (2*\x, 4*\y+\z) {$\tx_r$};

    \path (A2_3) edge [->] node [auto] {$\scriptstyle{\LA{mn}\circ\LA{sm}\circ\LA{rs}}$} (A2_4);
    \path (A2_1) edge [->] node [auto,swap] {$\scriptstyle{\LA{ih}}$} (A2_0);
    \path (A2_1) edge [->] node [auto] {$\scriptstyle{\LA{ij}}$} (A2_2);
    \path (A2_3) edge [->] node [auto,swap] {$\scriptstyle{\LA{kj}\circ\LA{fk}\circ\LA{rf}}$} (A2_2);
    \path (A3_2) edge [->] node [auto] {$\scriptstyle{\LA{fi}\circ\LA{rf}}$} (A2_1);
    \path (A3_2) edge [->] node [auto,swap] {$\scriptstyle{1_{\tU_r}}$} (A2_3);
\end{tikzpicture}\]

we can compute:

\begin{eqnarray*}
& m\Bigg([\LA{ih},\tx_i,\LA{ij}],m\Big([\LA{kj},\tx_k,\LA{kl}],[\LA{ml},\tx_m,\LA{mn}]\Big)\Bigg)= &\\
& =m\Big([\LA{ih},\tx_i,\LA{ij}],[\LA{kj}\circ\LA{fk}\circ\LA{rf},\tx_r,\LA{mn}\circ\LA{sm}\circ
   \LA{rs}]\Big)= &\\
& =[\LA{ih}\circ\LA{fi}\circ\LA{rf},\tx_r,\LA{mn}\circ\LA{sm}\circ\LA{rs}]. &
\end{eqnarray*}

Since diagram (\ref{eq-46}) is symmetric, we obtain the same result also if we first compute the multiplication of
the first two arrows, and then we multiply them with the third one. In other words, we have proved that
$m\circ(1_R\times m)=m\circ(m\times 1_R)$.
\end{proof}

Now we want to decribe a structure of complex manifold on $R$; in order to do that, we will use the following lemma:

\begin{lem}\label{local-triviality}
The equivalence relation $\sim$ is the trivial one whenever we restrict to any open set of $\hat{R}$ of the form
$\tU_k^{ij}$.
\end{lem}

This is just a direct consequence of the definition of $\sim$ restricted to any pair of points of the form
$(\LA{ki},\tx_k,\LA{kj}) $ and $ (\LA{ki},\tx'_k,\LA{kj})$.

\begin{prop}\label{manifold-topology}
If we give to $R=\hat{R}/\sim$ the quotient topology, then we get a natural structure of complex manifold.
\end{prop}

\begin{proof}
Let us consider any open set of $\hat{R}$ of the form $A:=\tU_k^{ij}$ and let us denote with $A^{\textrm{sat}}$
the saturated of $A$ in $\hat{R}$ with respect to $\sim$. \emph{We claim that also }$A^{\textrm{sat}}$\emph{ is
open in} $\hat{R}$. Indeed, let us consider any point in $A^{\textrm{sat}}$, i.e. a point which is equivalent to a
point $(\LA{ki},\tx_k,\LA{kj})$ of $\tU_k^{ij}$. By definition of $\sim$, this must be necessarily of the form
$(\LA{li},\tx_l,\LA{lj})$; moreover, there must exist a uniformizing system $\unif{m}$, a point $\tx_m$ and
embeddings $\LA{mk},\LA{ml}$ as in (\ref{eq-31}).\\

Now let us consider the set $\tB:=\LA{ml}(\tU_m)\subseteq\tU_l$, which is an open neighborhood of $\tx_l$ (because
$\LA{ml}$ is an embedding between open sets of $\mathbb{C}^n$, where $n$ is the dimension of the orbifold atlas
$\mathcal{U}$). If we fix any other point $\tx'_m$ in $\tU_m$ we
get a diagram similar to the second one of (\ref{eq-31}), so the set $\tB$ (considered as an open subset of
$\tU_l^{ij}$, hence also
as an open set of $\hat{R}$) contains only points equivalent to points  of $A$, so is completely contained in
$A^{\textrm{sat}}$; hence $A^{\textrm{sat}}$ is open in $\hat{R}$.\\

So if $q:\hat{R}\rightarrow R$ is the quotient map, the set $q(A^{\textrm{sat}})$ is open in $R$ by definition
of quotient topology; moreover, by definition of
saturated, it coincides with $q(A)$. Since this holds for every choice of $A=\tU_k^{ij}$ we get that the family
$\{(\tW_k^{ij}:=q(\tU_k^{ij})=q(\tU_k^{ij\textrm{sat}})\}_{\tU_k^{ij}\subseteq \hat{R}}$ is an open covering of $R$
(in the quotient topology). Then our aim is to construct from it a complex manifold atlas on $R$. If we use the
previous lemma, we get that $\sim$ is the trivial equivalence relation on every set $\tU_k^{ij}$, so
$q(\tU_k^{ij})$ is homeomorphic to $\tU_k^{ij}$ via $q$ (which is invertible if we restrict to this set). Moreover,
we recall that by construction $\tU_k^{ij}$ is just a copy of $\tU_k$, so the set map $\phi_k^{ij}$ defined from
$\tW_k^{ij}$ to $\tU_k$ as:

$$\phi_k^{ij}([\LA{ki},\tx_k,\LA{kj}]):=\tx_k$$

is an homeomorphism (with codomain an open subset of $\mathbb{C}^n)$. So it makes sense to consider the family of
charts $\mathcal{F}:=\{(\tW_k^{ij},\phi_k^{ij})\}_{\tU_k^{ij}\subseteq \hat{R}}$. Since the domains of these charts
cover all $R$, it remains only to prove the compatibility condition on the intersection of any pair of charts; so
let us fix any pair of domains $\tW_k^{ij}$ and $\tW_{l}^{i'j'}$ such that $\tW_k^{ij}\cap\tW_{l}^{i'j'}$ is
non-empty and let us fix any point $P=[\LA{ki},\tx_k,\LA{kj}]=[\LA{li'},\tx_l,\LA{lj'}]$ in the intersection. By
definition of $\sim$, we get that necessarily $i'=i$ and $j'=j$; moreover, there exist a uniformizing system
$\unif{m}$, a point $\tx_m\in\tU_m$ and a pair of embeddings $\LA{mk},\LA{ml}$ as in (\ref{eq-31}). Now the images
of the point $P$ via the coordinate functions $\phi_k^{ij}$ and $\phi_l^{ij}$ are respectively:

$$\phi_k^{ij}([\LA{ki},\tx_k,\LA{kj}])=\tx_k=\LA{mk}(\tx_m)\AND
\phi_l^{ij}([\LA{li},\tx_l,\LA{lj}])=\tx_l=\LA{ml}(\tx_m).$$

So if we call $\phi$ the transition map:

$$\phi=\phi_l^{ij}\circ(\phi_k^{ij})^{-1}:\phi_k^{ij}(\tW_k^{ij}\cap\tW_l^{ij})
\rightarrow\phi_l^{ij}(\tW_k^{ij}\cap\tW_l^{ij}),$$

we get that $\phi(\tx_k)=\tx_l=\LA{ml}(\tx_m)=\LA{ml}\circ(\LA{mk}|_{\LA{mk}(\tU_m)})^{-1}(\tx_k)$. As before,
using diagram (\ref{eq-31}) we get that this is the expression of $\phi$ not only at the point $\tx_k$, but also
in an open neighborhood of it (not necessarily coinciding with all the domain of $\phi$). Hence we have proved that
the transition map $\phi$ locally coincides with an holomorphic function. So every transition function is holomorphic,
hence we have proved that the family $\mathcal{F}$ is a complex manifold atlas for $R$.
\end{proof}

\begin{lem}\label{etale-groupoid}
$\groupR$ is an \'etale groupoid object in $\CAT{Manifolds}$.
\end{lem}

\begin{proof}
We have already proved that $\groupR$ is a groupoid object in $\CAT{Sets}$, and that both $R$ and $U$ are complex
manifolds. Hence we have only to prove the
additional properties about the five structure maps. In particular, we have to prove that $s$ and $t$ are both
\'etale (hence, in particular, holomorphic) and that $m,i$ and $e$ are holomorphic.\\

Let us prove that $s$ is \'etale (the proof for $t$ is analogous); since the property of being \'etale is a local
one, we can check it by restricting to the domains of suitable charts in source and target. So let us fix any point
$[\LA{ki},\tx_k,\LA{kj}]$ in $R$ and the chart $(\tW_k^{ij},\phi_k^{ij})$ around it. We recall that:

$$s([\LA{ki},\tx_k,\LA{kj}])=\LA{ki}(\tx_k)\in\tU_i\subseteq U$$

where $\tU_i$ means a copy of $\tU_i$ in the manifold $U$; so a chart around this point is just $(\tU_i,id)$. Hence
the map $s$ can be expressed in coordinates as:

$$\ts:=id\circ s\circ(\phi_k^{ij})^{-1}:\tU_k\rightarrow\tU_i$$

which concides with the holomorphic embedding $\LA{ki}$. So $s$ is a biholomorphism if restricted to $\tW_k^{ij}$
in domain and to $\LA{ki}(\tU_k)$ in codomain. Hence we have proved that $s$ is \'etale.\\

In order to prove that $m:\fibre{R}{t}{s}{R}\rightarrow R$ is holomorphic, let us fix any point:

$$(P,P'):=\Big([\LA{ih},\tx_i,\LA{ij}],[\LA{kj},\tx_k,\LA{kl}]\Big)\in\fibre{R}{t}{s}{R}$$

and a completion of the form (\ref{eq-41}). Then we can write:

$$(P,P')=\Big([\LA{fh},\tx_f,\LA{fj}],[\LA{fj},\tx_f,\LA{fl}]\Big)$$

(where we define $\LA{fh}:=\LA{ih}\circ\LA{fi}$ and analogously for the other 3 embeddings). Now let us define
a set map $\delta:\tU_f\rightarrow R\times R$ as:

$$\delta(\ty_f):=\Big([\LA{fh},\ty_f,\LA{fj}],[\LA{fj},\ty_f,\LA{fl}]\Big).$$

This map is clearly holomorphic because $R\times R$ has the product topology and by combining $\delta$ with the
first and second projection we get exactly inverses of holomorphic coordinates functions on $R$ (see the explicit
description in proposition \ref{manifold-topology}). Moreover, one see easily that $\delta$ has target in
$\fibre{R}{t}{s}{R}$, which is a complex submanifold of $R\times R$ (see proposition \ref{fiber-products-1}), hence
$\delta$ is holomorphic from $\tU_f$ to $\fibre{R}{t}{s}{R}$. In addition, there exists an obvious local inverse of
$\delta$, again holomorphic, hence $\delta$ is a biholomorphism if restricted in target to an open neighborhood of
$(P,P')$. Now using a diagram similar to (\ref{eq-41}) we get that $m\circ\delta(\ty_f)=[\LA{fh},\ty_f,\LA{fl}]$;
since a chart around $m(P,P')$ is just $(\tW_f^{hl},\phi_f^{hl})$, we get that in order to check whether $m$ is
holomorphic or not around $(P,P')$, it suffices to check if $\phi_f^{hl}\circ m\circ\delta$ is holomorphic near
$\tx_f$. Now this map just coincides with the identity on the whole $\tU_f$, hence $m$ is holomorphic around
$(P,P')$; since this holds for every point of $\fibre{R}{t}{s}{R}$, we are done.\\

Analogous arguments prove that both $i$ and $e$ are holomorphic maps.
\end{proof}

\begin{lem}\label{useful-lemma-2}
Suppose we have fixed 2 points $P=[\LA{ki},\tx_k,\LA{kj}]$ and $Q=[\LA{li},\tx_l,\LA{lj}]$ of $R$ with 
$s(P)=s(Q)$ and $t(P)=t(Q)$. Then there exists a unique $g\in G_j$ such that:

\begin{equation}\label{eq-80}
[\LA{li},\tx_l,\LA{lj}]=[\LA{ki},\tx_k,g\circ\LA{kj}]. 
\end{equation}

Moreover, such a $g$ belongs to the stabilizer of $\LA{kj}(\tx_k)$ in $\tU_j$, so by lemma \ref{lemma-intersection}
and lemma \ref{lemma-moerdijk} there exists a unique $g'\in (G_k)_{\tx_k}$ such that:

$$[\LA{li},\tx_l,\LA{lj}]=[\LA{ki},\tx_k,\LA{kj}\circ g'].$$
\end{lem}

\begin{proof}
The hypothesis implies that $\LA{ki}(\tx_k)=\LA{li}(\tx_l)$ and $\LA{kj}(\tx_k)=\LA{lj}(\tx_l)$. Using the first
relation, we can apply lemma \ref{useful-lemma} to the pair of embeddings $\LA{ki},\LA{li}$, so we get a pair of diagrams
of the form (\ref{eq-31}), which a priori are both commutative only in the left part. Then one can apply lemma \ref{lemma-moerdijk}
to the right part, so we get a unique $g\in G_j$ such that we have commutative diagrams of the form:

\begin{equation}
\begin{tikzpicture}[scale=0.8]
    \def\x{1.5}
    \def\y{-1.2}
    \node (A2_0) at (2*\x, 0*\y) {$\tU_k$};
    \node (A0_2) at (0*\x, 2*\y) {$\tU_i$};
    \node (A2_2) at (2*\x, 2*\y) {$\tU_m$};
    \node (A4_2) at (4*\x, 2*\y) {$\tU_j$};
    \node (A2_4) at (2*\x, 4*\y) {$\tU_l$};
    \node (A1_2) at (1*\x, 2*\y) {$\curvearrowright$};
    \node (A3_2) at (3*\x, 2*\y) {$\curvearrowright$};
    \path (A2_0) edge [->] node [auto,swap] {$\scriptstyle{\LA{ki}}$} (A0_2);
    \path (A2_0) edge [->] node [auto] {$\scriptstyle{\LA{kj}}$} (A4_2);
    \path (A2_4) edge [->] node [auto] {$\scriptstyle{\LA{li}}$} (A0_2);
    \path (A2_4) edge [->] node [auto,swap] {$\scriptstyle{g\circ\LA{lj}}$} (A4_2);
    \path (A2_2) edge [->] node [auto,swap] {$\scriptstyle{\LA{mk}}$} (A2_0);
    \path (A2_2) edge [->] node [auto] {$\scriptstyle{\LA{ml}}$} (A2_4);

    \def\z{7}
    \node (B2_0) at (2*\x+\z, 0*\y) {$\tx_k$};
    \node (B0_2) at (0*\x+\z, 2*\y) {$\tx_i$};
    \node (B2_2) at (2*\x+\z, 2*\y) {$\tx_m$};
    \node (B4_2) at (4*\x+\z, 2*\y) {$\tx_j$};
    \node (B2_4) at (2*\x+\z, 4*\y) {$\tx_l$};
    \node (B1_2) at (1*\x+\z, 2*\y) {$\curvearrowright$};
    \node (B3_2) at (3*\x+\z, 2*\y) {$\curvearrowright$};
    \path (B2_0) edge [->] node [auto,swap] {$\scriptstyle{\LA{ki}}$} (B0_2);
    \path (B2_0) edge [->] node [auto] {$\scriptstyle{\LA{kj}}$} (B4_2);
    \path (B2_4) edge [->] node [auto] {$\scriptstyle{\LA{li}}$} (B0_2);
    \path (B2_4) edge [->] node [auto,swap] {$\scriptstyle{g\circ\LA{lj}}$} (B4_2);
    \path (B2_2) edge [->] node [auto,swap] {$\scriptstyle{\LA{mk}}$} (B2_0);
    \path (B2_2) edge [->] node [auto] {$\scriptstyle{\LA{ml}}$} (B2_4);
\end{tikzpicture}\end{equation}

Hence, by definition of $\sim$ on $\hat{R}$, we get that (\ref{eq-80}) is satisfied. Moreover, it is simple
to prove that such a $g$ is also unique using again lemma \ref{lemma-moerdijk}.
\end{proof}

\begin{lem}
The \'etale groupoid object $\groupR$ is effective.
\end{lem}

\begin{proof}
Let us fix any point $(\tx_k,\tU_k)\in U$; by applying the previous lemma we get that the set of points $P$
of $R$ such that $s(P)=t(P)=(\tx_k,\tU_k)$ is in natural bijection with the stabilizer $(G_k)_{\tx_k}$; in
particular, the bijection is given by $g\mapsto [1_{\tU_k},\tx_k,g]$. Now if we restrict to any open
set of the form $\tU_k\subseteq U$ we get that every point $P=[1_{\tU_k},\tx_k,g]$ induces the set map $t\circ
(s|_{\tU_k})^{-1}=g$. Since the orbifold atlas $\mathcal{U}$ is reduced by hypothesis, the group $G_k$ acts
effectively on $\tU_k$, hence the set map $f_{(\tx_k,\tU_k)}$ (see definition \ref{effective-groupoid}) is
injective. Since this holds for every point of $U$, we have proved that $\groupR$ is effective.
\end{proof}

\begin{lem}
The relative diagonal $(s,t):R\rightarrow U\times U$ is proper.
\end{lem}

\begin{proof}(adapted from \cite{Pr}, proposition 4.4.8 and corollary 4.4.9)
Let us fix any point $(P,P')=\Big((\tx_i,\tU_i),(\tx_j,\tU_j)\Big)\in U\times U$ and let us distinguish between
2 cases: if $\pi_i(\tx_i)\neq\pi_j(\tx_j)$, then we can use the fact that $X$ is Hausdorff (by definition of
orbifold) and we get that there exist two open disjoint neighborhoods $D_i$ and $D_j$ of $\pi_i(\tx_i)$ and
$\pi_j(\tx_j)$. If we call $\tD_i:=\pi_i^{-1}(D_i)\subseteq \tU_i$ and $\tD_j:=\pi_j^{-1}(D_j)\subseteq\tU_j$, we
get that $\tD_i\times\tD_j$ is an open neighborhood of $(P,P')$ and its preimage via $(s,t)$ is empty.\\

Now let us consider the case when $\pi_i(\tx_i)=\pi_j(\tx_j)$; in this case we can use property (ii) of orbifold
atlases and remark \ref{useful-remark} in order to find a uniformizing system $\unif{k}\in\mathcal{U}$, a point
$\tx_k\in\tU_k$ and embeddings $\LA{ki},\LA{kj}$ such that $\LA{ki}(\tx_k)=\tx_i$ and $\LA{kj}(\tx_k)=\tx_j$.
Then let us consider the open sets $\tW_i:=\LA{ki}(\tU_k)\subseteq\tU_i$, $\tW_j:=\LA{kj}(\tU_k)\subseteq\tU_j$ and
the set $\tW_j\times\tW_j$, which is an open neighborhood of $(P,P')$ in $U\times U$. Now let us fix any point:

$$Q=[\LA{li},\ty_l,\LA{lj}]\in(s,t)^{-1}(\tW_i\times\tW_j)$$

and let us define $\ty_k:=\LA{ki}^{-1}(\LA{li}(\ty_l))$ (well defined by construction of $\tW_j$); then we get that:

$$\pi_j(\LA{kj}(\ty_k))=\pi_k(\ty_k)=\pi_i(\LA{li}(\ty_l))=\pi_l(\ty_l)=\pi_j(\LA{lj}(\ty_l)).$$

So by definition of uniformizing system there exists $g_1\in G_j$ such that:

$$g_1\circ\LA{kj}(\ty_k)=\LA{lj}(\ty_l);$$

then if we define $P:=[\LA{ki},\ty_k,g_1\circ\LA{kj}]$, we get that $s(P)=s(Q)$ and $t(P)=t(Q)$, so we can apply
lemma \ref{useful-lemma-2} and we get that there exists $g_2\in G_j$ such that
$Q=[\LA{ki},\ty_k,g_2\circ g_1\circ\LA{kj}]$. So we conclude that every point in $(s,t)^{-1}(\tW_i\times
\tW_j)$ is of the form $[\LA{ki},\ty_k,g\circ\LA{kj}]$ for some $\ty_k\in\tU_k$ and some $g\in G_j$.
So we have proved that:

\begin{equation}\label{eq-47}
(s,t)^{-1}(\tW_i\times\tW_j)\subseteq\bigcup_{g\in G_j}\tW_k^{\LA{ki},g\circ\LA{kj}}
\end{equation}

(where we use the notations introduced in the proof of proposition \ref{manifold-topology}).\\

Now let us fix any compact set $K\subseteq(U\times U)$ and let us fix any sequence $\{Q_n\}_{n\in\mathbb{N}}
$ in $(s,t)^{-1}(K)$. If
necessary by extracting a subsequence, we can assume that $(s,t)(Q_n)$ $=:(P_n,P'_n)$ converges to a point $(P,P')\in
K$. Hence for every open neighborhood $A$ of $(P,P')$ in $U\times U$ we get that $(s,t)^{-1}(A)$ is not empty, so
we are necessarily in the second of the previous 2 cases, hence there exists an open neighborhood $\tW_i\times
\tW_j$ of $(P,P')$ such that (\ref{eq-47}) holds. Now $(P_n,P'_n)$ converges to $(P,P')$, so for $n$ big enough we
can assume that $Q_n\in(s,t)^{-1}(\tW_i\times\tW_j)$; moreover, the union of (\ref{eq-47}) is made over a finite
set, so by passing to a subsequence we can assume that there exists $g\in G_j$ such that $Q_n\in \tW_k^{\LA{ki},
g\circ\LA{kj}}$ for all $n$ big enough. We have proved in lemma \ref{etale-groupoid} that $s$ is a biholomorphism
(hence homeomorphism) if restricted to this set, so if we call $Q$ the unique point of this set such that $s(Q)=P$,
we get that $s(Q_n)$ converges to $Q$. Hence we have proved that $(s,t)^{-1}$ of every compact set is compact,
i.e. $(s,t)$ is proper.
\end{proof}

From all the previous lemmas we get:

\begin{prop}\label{orb-to-groupoid-1}
$\groupR$ is an object of $\CAT{Grp}$.
\end{prop}

\subsection{Morphisms and 2-morphisms}
Now let us pass to morphisms: our aim is to associate to every compatible system (i.e. a morphism in $\CAT{Pre-Orb})$
a morphism in $\CAT{Grp}$.

\begin{definprop}\label{orb-to-groupoid-2}
Let $\mathcal{U}=\{\unif{i}\}_{i\in I}$ and $\mathcal{V}=\{\uniff{j}\}_{j\in J}$ be orbifold atlases for $X$ and
$Y$ respectively, let $\tf:\mathcal{U}\rightarrow\mathcal{V}$ be a compatible system (see definition
\ref{compatible-system}) for a continuous map $f:X\rightarrow Y$ and let $\groupR$ and $\groupRR$ be the groupoid
objects associated to $\mathcal{U}$ and $\mathcal{V}$ respectively. For simplicity, from now on for every point
$\tx_i\in\tU_i$ we will denote with $\ty_i$ its image in $\tV_i$ via the holomorphic function $\tf_{\tU_i,\tV_i}:
\tU_i\rightarrow\tV_i$. Now \emph{we define the set map }$\psi:U\rightarrow U'$ as:

$$\left.\psi\right|_{\tU_i}=\tf_{\tU_i,\tV_i}:\tU_i\rightarrow\tV_i\subseteq U'.$$

Let us define also a set map $\Psi:R\rightarrow R'$ as follows: let us fix any point $[\LA{ki},\tx_k,
\LA{kj}]\in R$ and a representative $(\LA{ki},\tx_k,\LA{kj})$ of it; then we set:

$$\Psi([\LA{ki},\tx_k,\LA{kj}]):=[\tf(\LA{ki}),\tf_{\tU_k,\tV_k}(\tx_k),\tf(\LA{kj})]=
[\tf(\LA{ki}),\ty_k,\tf(\LA{kj})].$$

Using remark \ref{good-definitions-on-R} and the fact that $\tf$ is a functor
(by definition \ref{compatible-system}), we get that $\Psi$ does not depend on the representative chosen for
$[\LA{ki},\tx_k,\LA{kj}]$. Then we claim that $(\psi,\Psi)$ is a morphism of groupoid objects in $\CAT{Grp}$
from $\groupR$ to $\groupRR$.
\end{definprop}

\begin{proof}
First of all, since the local liftings of $f$ are all holomorphic, so is $\psi$. Now we recall that in
proposition \ref{manifold-topology} we described a manifold atlas for $R$ where the charts are of the form
$(\tW_k^{ij},\phi_k^{ij})$; analogously, we can use similar charts of the form $(\tZ_k^{ij},\xi_k^{ij})$ on $R'$.
If we write $\Psi$ in coordinates with respect to these charts, we get that $\Psi$ coincides with $\tf_{\tU_k,
\tV_k}$, which is holomorphic by definition \ref{compatible-system}. Hence in order to prove the statement it
suffices to prove the axioms of definition \ref{morphism-groupoid}, which are easy to verify working
set-theoretically, so we omit the details.
\end{proof}

Now let us fix two atlases $\mathcal{U}$ and $\mathcal{V}$ for $X$ and $Y$ respectively, two compatible systems 
$\tf_1,\tf_2:\mathcal{U}\rightarrow\mathcal{V}$ for a continuous function $f:X\rightarrow Y$ and a natural
transformation $\delta:\tf_1\Rightarrow\tf_2$ as in definition \ref{nat-tran-orb}. Let us call $\groupR$ and 
$\groupRR$ the groupoid objects associated to the atlases $\mathcal{U}$ and $\mathcal{V}$ respectively; moreover, 
let us denote with $(\psi,\Psi)$ and $(\phi,\Phi)$ the morphisms of groupoid objects from $(\groupR)$ to
$(\groupRR)$ associated to $\tf_1$ and $\tf_2$ respectively by definition \ref{orb-to-groupoid-2}; then we give the
following:

\begin{definprop}\label{orb-to-groupoid-3}
The \emph{set map}: $\alpha:U=\coprod_{\unif{i}\in\mathcal{U}}\tU_i\rightarrow R'$ defined by:

$$\alpha(\tx_i,\tU_i):=[1_{\tV_{i}^{1}},(\tf_1)_{\tU_i,\tV_i^1}(\tx_i),\delta_{\tU_i}]$$

is a natural transformation from $(\psi,\Psi)$ to $(\phi,\Phi)$ in $\CAT{Grp}$.
\end{definprop}

\begin{proof}
First of all, \emph{we claim that} $\alpha$ \emph{is holomorphic}: if we restrict $\alpha$ to every open set 
$\tU_i$ in $U$ (with the natural chart $(\tU_i,id)$), we get that its range is contained in the open set $A:=q'
\left((\tV_i^1)^{1_{\tV_i^1},\delta_{\tU_i}}\right)$, which is biholomorphic to $\tV_i^1$ (see the proof of
proposition \ref{manifold-topology}). By composing with these biholomorphic changes of coordinates we get that
$\alpha$ coincides on $\tU_i$ with the holomorphic map $(\tf_1)_{\tU_i,\tV_i^1}$. Since the open sets of the form $\tU_i$
cover all $U$, we have proved that $\alpha$ is holomorphic on all $U$, i.e. it is a morphism in $\CAT{Manifolds}$.\\

So in order to prove that $\alpha$ is a natural transformation in $\CAT{Grp}$ it suffices to verify the axioms of 
definition \ref{nat-tran-group}. The first condition is
easy to verify, so let us consider only condition (ii); in order to prove it, let us fix any $[\LA{ki},\tx_k,
\LA{kj}]\in R$. For simplicity, let us call $\tx_i:=\LA{ki}(\tx_k)\in\tU_i$ and $\tx_j:=\LA{kj}(\tx_k)\in\tU_j$.
Moreover, we use the notations introduced in definition \ref{nat-tran-orb} and for every point $\tx_h\in\tU_h$ we
define:

$$\ty_h^m:=(\tf_m)_{\tU_h,\tV_h^m}(\tx_h)\in\tV_h^m\textrm{\quad for } m=1,2.$$

So we get that $\alpha\circ s([\LA{ki},\tx_k,\LA{kj}])=\alpha(\tx_i,\tU_i)=[1_{\tV_i^1},\ty_i^1,\delta_{\tU_i}]$
and $\Phi([\LA{ki},\tx_k,\LA{kj}])=[\lambda_{ki}^2,\ty_k^2,\lambda_{kj}^2]$. Moreover, using (\ref{eq-6}) we get
the following commutative diagram:

\[\begin{tikzpicture}[scale=0.8]
    \def\x{3}
    \def\y{-0.55}
    \def\z{-2}
    \node (A0_1) at (1*\x, 0*\y) {$\ty_i^1$};
    \node (A0_2) at (2*\x, 0*\y) {$\delta_{\tU_i}(\ty_i^1)=\ty_i^2=\lambda_{ki}^2(\ty_k)$};
    \node (A0_3) at (3*\x, 0*\y) {$\ty_k^2$};
    \node[rotate=270] (A1_1) at (1*\x, 1*\y) {$\in$};
    \node[rotate=270] (A1_2) at (2*\x, 1*\y) {$\in$};
    \node[rotate=270] (A1_3) at (3*\x, 1*\y) {$\in$};
    \node (A2_0) at (0*\x, 2*\y) {$\tV_i^1$};
    \node (A2_1) at (1*\x, 2*\y) {$\tV_i^1$};
    \node (A2_2) at (2*\x, 2*\y) {$\tV_i^2$};
    \node (A2_3) at (3*\x, 2*\y) {$\tV_k^2$};
    \node (A2_4) at (4*\x, 2*\y) {$\tV_j^2.$};
    \node (B0_0) at (2*\x, 4*\y) {$\curvearrowright$};
    \node (A3_2) at (2*\x, 2*\y+\z) {$\tV_k^1$};
    \node[rotate=90] (A4_2) at (2*\x, 3*\y+\z) {$\in$};
    \node (A5_2) at (2*\x, 4*\y+\z) {$\ty_k^1$};

    \path (A2_3) edge [->] node [auto] {$\scriptstyle{\lambda_{kj}^2}$} (A2_4);
    \path (A2_1) edge [->] node [auto,swap] {$\scriptstyle{1_{\tV_i^1}}$} (A2_0);
    \path (A2_1) edge [->] node [auto] {$\scriptstyle{\delta_{\tU_i}}$} (A2_2);
    \path (A2_3) edge [->] node [auto,swap] {$\scriptstyle{\lambda_{ki}^2}$} (A2_2);
    \path (A3_2) edge [->] node [auto] {$\scriptstyle{\lambda_{ki}^1}$} (A2_1);
    \path (A3_2) edge [->] node [auto,swap] {$\scriptstyle{\delta_{\tU_k}}$} (A2_3);
\end{tikzpicture}\]

Hence we have that:

\begin{eqnarray}
\nonumber & m'\circ(\alpha\circ s,\Phi)([\LA{ki},\tx_k,\LA{kj}])=m'\Big([1_{\tV_i^1},\ty_i^1,\delta_{\tU_i}],
   [\lambda_{ki}^2,\ty_k^2,\lambda_{kj}^2]\Big)=& \\
\label{eq-48} &=[\lambda_{ki}^1,\ty_k^1,\lambda_{kj}^2\circ\delta_{\tU_k}].&
\end{eqnarray}

On the other hand, we get that $\alpha\circ t([\LA{ki},\tx_k,\LA{kj}])=\alpha(\tx_j,\tU_j)=[1_{\tV_j^1},\ty_j^1,
\delta_{\tU_j}]$ and $\Psi([\LA{ki},\tx_k,\LA{kj}])=[\lambda_{ki}^1,\ty_k^1,\lambda_{kj}^1]$, so we can compute:

\begin{eqnarray}
\nonumber & m'\circ(\Psi,\alpha\circ t)([\LA{ki},\tx_k,\LA{kj}])=m'\Big([\lambda_{ki}^1,\ty_k^1,\lambda_{kj}^1],
   [1_{\tV_j^1},\ty_j^1,\delta_{\tU_j}]\Big)=& \\
\label{eq-49} & =[\lambda_{ki}^1,\ty_k^1,\delta_{\tU_j}\circ\lambda_{kj}^1].&
\end{eqnarray}

Using again (\ref{eq-6}) we get that $\delta_{\tU_j}\circ\lambda_{kj}^1=\lambda_{kj}^2\circ\delta_{\tU_k}$, hence
(\ref{eq-48}) and (\ref{eq-49}) are equal. Since this holds for every point $[\LA{ki},\tx_k,\LA{kj}]$ of $R$, we
get that (ii) is proved.
\end{proof}

\subsection{The 2-functor \texorpdfstring{$F$}{F}}
Until now we have described how to associate:

\begin{enumerate}[(i)]\parindent=0pt
\item to every orbifold atlas $\mathcal{U}$ a groupoid object $\groupR$, which is an object in $\CAT{Grp}$;

\item to every compatible system $\tf$ a morphism $(\psi,\Psi)$ of groupoid objects, which is in particular a
morphism in $\CAT{Grp}$;

\item to every natural transformation $\delta$ between compatible systems a natural transformation $\alpha$ in
$\CAT{Grp}$.
\end{enumerate}

A straightforward calculation proves that:

\begin{prop}
Whenever we fix a pair of objects $\mathcal{U},\mathcal{V}$ in $\CAT{Pre-Orb}$ with associated groupoid objects 
$\groupR$ and $\groupRR$ respectively, we get a functor:

$$F=F_{\mathcal{U},\mathcal{V}}:\CAT{Pre-Orb}(\mathcal{U},\mathcal{V})\rightarrow \CAT{Grp}\Big((\groupR),
(\groupRR)\Big)$$

defined by (ii) on the level of objects and by (iii) on the level of morphisms.
\end{prop}

\begin{teo}\label{2-functor-final}
The previous data define a 2-functor $F$ from $\CAT{Pre-Orb}$ to $\CAT{Grp}$.
\end{teo}

\begin{proof}
It suffices to verify axioms (a), (b) and (c) of definition \ref{2-func}.

\begin{enumerate}[(a)]\parindent=0pt
\item Let us fix any pair of compatible systems $\tf:\mathcal{U}\rightarrow\mathcal{V}$ and $\tg:\mathcal{V}
\rightarrow\mathcal{W}$. Then for simplicity, let us call:

\begin{eqnarray*}
& F(\mathcal{U})=:(\groupR),\quad F(\mathcal{V})=:(\groupRR),\quad F(\mathcal{W})=:(\groupRRR), &\\
& F(\tf)=:(\psi,\Psi),\quad F(\tg)=:(\phi,\Phi),\quad\tg\circ\tf=:\th\AND F(\th)=:(\theta,\Theta). & 
\end{eqnarray*}

So we want to prove that $\theta\stackrel{?}{=}\phi\circ\psi$ and $\Theta\stackrel{?}{=}\Phi\circ\Psi$. Using
remark \ref{groupoid-property-2}, it suffices to prove only the second equality. Indeed, if this is proved, we 
get that:

$$\theta=s''\circ\Theta\circ e=(s''\circ\Phi)\circ(\Psi\circ e)=\phi\circ s' \circ e'\circ\psi=\phi\circ 1_{U'}
\circ\psi=\phi\circ\psi.$$

Now in order to prove the second equality, let us fix any point $[\LA{ki},\tx_k,\LA{kj}]\in R$. Then we have:

\begin{eqnarray*}
& \Phi\circ\Psi([\LA{ki},\tx_k,\LA{kj}])=\Phi([\tf(\LA{ki}),\tf_{\tU_k,\tV_k}(\tx_k),\tf(\LA{kj})])= &\\
& =[\tg\circ\tf(\LA{ki}),\tg_{\tV_k,\tW_k}\circ\tf_{\tU_k,\tV_k}(\tx_k),\tg\circ\tf(\LA{kj})]=[\th
   (\LA{ki}),\th_{\tU_k,\tW_k}(\tx_k),\th(\LA{kj})]= &\\
& =\Theta([\LA{ki},\tx_k,\LA{kj}]).&
\end{eqnarray*}

\item Let us fix a diagram of compatible systems and natural transformations in $\CAT{Pre-Orb}$ of the form:

\[\begin{tikzpicture}[scale=0.8]
    \def\x{1.5}
    \def\y{-1.2}
    \node (A0_0) at (0*\x, 0*\y) {$\mathcal{U}$};
    \node (A0_1) at (1*\x, 0*\y) {$\Downarrow\delta$};
    \node (A0_2) at (2*\x, 0*\y) {$\mathcal{V}$};
    \node (A0_3) at (3*\x, 0*\y) {$\Downarrow\eta$};
    \node (A0_4) at (4*\x, 0*\y) {$\mathcal{W}.$};
    \path (A0_2) edge [->,bend left=25] node [auto] {$\scriptstyle{\tg_1}$} (A0_4);
    \path (A0_2) edge [->,bend right=25] node [auto,swap] {$\scriptstyle{\tg_2}$} (A0_4);
    \path (A0_0) edge [->,bend left=25] node [auto] {$\scriptstyle{\tf_1}$} (A0_2);
    \path (A0_0) edge [->,bend right=25] node [auto,swap] {$\scriptstyle{\tf_2}$} (A0_2);
\end{tikzpicture}\]

For simplicity, let us use the notations of (a) on the level of objects and let us call:

\begin{eqnarray*}
& F(\tf_i)=:(\psi_i,\Psi_i),\quad F(\tg_i)=:(\phi_i,\Phi_i)\quad\textrm{for\,\,} i=1,2, &\\
& F(\delta)=:\alpha:U\rightarrow R',\quad F(\eta):=\beta:U'\rightarrow R''\AND F(\eta\ast\delta):=
   \gamma:U\rightarrow R''. &
\end{eqnarray*}

By definition of $\ast$ in $\CAT{Pre-Orb}$, for every $\unif{i}\in\mathcal{U}$ we have:

$$(\eta\ast\delta)_{\tU_i}=\eta_{\tV_i^2}\circ\tg_1(\delta_{\tU_i}):\tW_i^{11}\rightarrow\tW_i^{22}$$

(where we use the notations of \S 1.3); so for every point $(\tx_i,\tU_i)\in U$ we have:

\begin{equation}\label{eq-50}
\gamma(\tx_i,\tU_i)=[1_{\tW_i^{11}},(\tg_1\circ\tf_1)_{\tU_i,\tW_i^{11}}(\tx_i),\eta_{\tV_i^2}\circ\tg_1
(\delta_{\tU_i})].
\end{equation}

On the other hand,

\begin{eqnarray}
\nonumber & \Big(F(\eta)\ast F(\delta)\Big)(\tx_i,\tU_i)=\big(\beta\ast\alpha\big)(\tx_i,\tU_i)=&\\
\nonumber & =m''\Big(\Phi_1\circ\alpha(\tx_i,\tU_i),\beta\circ\psi_2(\tx_i,\tU_i)\Big)=&\\
\nonumber & =m''\Big(\Phi_1\left([1_{\tV_i^1,},(\tf_1)_{\tU_i,\tV_i^1}(\tx_i),\delta_{\tU_i}]\right),
  \beta\left((\tf_2)_{\tU_i,\tV_i^2}(\tx_i),\tV_i^2\right)\Big)=&\\
\nonumber & =m''\Big(\left[1_{\tW_i^{11}},(\tg_1)_{\tV_i^1,\tW_i^{11}}\circ(\tf_1)_{\tU_i,\tV_i^1}(\tx_i),
  \tg_1(\delta_{\tU_i})\right],&\\
\nonumber &\left[1_{\tW_i^{21}},(\tg_1)_{\tV_i^2,\tW_i^{21}}\circ(\tf_2)_{\tU_i,\tV_i^2}(\tx_i),
   \eta_{\tV_i^2}\right]\Big)= &\\
\label{eq-51} & =\left[1_{\tW_i^{11}},(\tg_1)_{\tV_i^1,\tW_i^{11}}\circ(\tf_1)_{\tU_i,\tV_i^1}(\tx_i),
   \eta_{\tV_i^2}\circ\tg_1(\delta_{\tU_i})\right].&
\end{eqnarray}

By comparing (\ref{eq-50}) with (\ref{eq-51}), we get:

$$F(\eta)\ast F(\delta)=F(\eta\ast\delta).$$

\item Let us fix any orbifold atlas $\mathcal{U}$ with associated groupoid object $\groupR$; then a direct 
check proves that $F(1_{\mathcal{U}})=1_{F(\mathcal{U})}$ and $F(i_{\mathcal{U}})=i_{F(\mathcal{U})}$.
\end{enumerate}
\end{proof}

\section{Some properties of the 2-functor \texorpdfstring{$F$}{F}}

\subsection{Morita equivalences}
We recall the following definition:

\begin{defin} (\cite{M}, \S 2.4)\label{morita-equivalence}
A morphism $(\psi,\Psi):(\groupRR)\rightarrow(\groupR)$ between Lie groupoids is called \emph{Morita equivalence}
(or \emph{essential equivalence}) iff the following 2 conditions hold:

\begin{enumerate}[(i)]\parindent=0pt
\item let us consider the following fiber product in $\CAT{Manifolds}$:

\begin{equation}\label{eq-52}
\begin{tikzpicture}
    \def\x{1.5}
    \def\y{-1.2}
    \node (A0_0) at (0*\x, 0*\y) {$\fibra{R}{U}{U'}$};
    \node (A0_2) at (2*\x, 0*\y) {$U'$};
    \node (A2_2) at (2*\x, 2*\y) {$U$};
    \node (A2_0) at (0*\x, 2*\y) {$R$};
    \node (A1_1) at (1*\x, 1*\y) {$\square$};
    \path (A0_0) edge [->] node [auto] {$\scriptstyle{\pi_2}$} (A0_2);
    \path (A0_2) edge [->] node [auto] {$\scriptstyle{\psi}$} (A2_2);
    \path (A2_0) edge [->] node [auto,swap] {$\scriptstyle{s}$} (A2_2);
    \path (A0_0) edge [->] node [auto,swap] {$\scriptstyle{\pi_1}$} (A2_0);
\end{tikzpicture} 
\end{equation}
since $\groupR$ is a Lie groupoid, we get that the map $s$ is a submersion, so we can apply proposition 
\ref{fiber-products-1} and we get that the fiber product has a natural structure of complex manifold and that also
$\pi_2$ is a submersion. \emph{Then we require that the set map}:

$$t\circ\pi_1:\fibra{R}{U}{U'}\rightarrow U$$

\emph{is a surjective holomorphic submersion}. This request makes sense because both source and target of this map are complex
manifolds;

\item \emph{we require also that the square}:
\begin{equation}\label{eq-53}
\begin{tikzpicture}[scale=0.8]
    \def\x{1.5}
    \def\y{-1.2}
    \node (A0_0) at (0*\x, 0*\y) {$R'$};
    \node (A0_2) at (2*\x, 0*\y) {$R$};
    \node (A2_0) at (0*\x, 2*\y) {$U'\times U'$};
    \node (A2_2) at (2*\x, 2*\y) {$U\times U$};
    \path (A2_0) edge [->] node [auto,swap] {$\scriptstyle{(\psi\times\psi)}$} (A2_2);
    \path (A0_0) edge [->] node [auto,swap] {$\scriptstyle{(s',t')}$} (A2_0);
    \path (A0_2) edge [->] node [auto] {$\scriptstyle{(s,t)}$} (A2_2);
    \path (A0_0) edge [->] node [auto] {$\scriptstyle{\Psi}$} (A0_2);
\end{tikzpicture}\end{equation}
\emph{is cartesian} in $\CAT{Manifolds}$. Note that the square is always commutative because of definition
\ref{morphism-groupoid}, so we have only to check the universal property of fiber products.
\end{enumerate}
\end{defin}

\begin{defin}\label{weak-equivalence}
Two groupoid objects $R^i\rightrightarrows U^i$ (for $i=1,2$) in $\CAT{Manifolds}$ are said
to be \emph{Morita equivalent} (or \emph{weak equivalent}) iff there exist a third groupoid object
$R^3\rightrightarrows U^3$ and two Morita equivalences:

$$(R^1\rightrightarrows U^1)\stackrel{(\psi,\Psi)}{\longleftarrow}(R^3\rightrightarrows U^3)
\stackrel{(\phi,\Phi)}{\longrightarrow}(R^2\rightrightarrows U^2).$$

This is actually an equivalence relation, see for example \cite{MM}, \mbox{chapter 5} for the proof.
\end{defin}

Now let us fix any orbifold structure $\mathcal{X}$ on a topological space $X$ and let us denote with 
$\mathcal{U}=\{\unif{i}\}_{i\in I}$ the maximal atlas associated to it by definition \ref{maximal-atlas}. Then let
us fix any other atlas $\mathcal{U}'=\{\unif{i'}\}_{i'\in I'}$ in the class $\mathcal{X}$. By construction
of $\mathcal{U}$, we get that $I'\subseteq I$, so by remark \ref{orb-atlas-2} we get that $\mathcal{U}'$ is a
subcategory of $\mathcal{U}$, hence we can consider a compatible system $\widetilde{id}:\mathcal{U}'\rightarrow
\mathcal{U}$ over the identity of $X$ as follows:

\begin{itemize}
\item as a functor, $\widetilde{id}$ is just the inclusion on the level of objects and morphisms (i.e: uniformizing
systems and embeddings);

\item for every uniformizing system $\unif{i'}\in\mathcal{U}$ we set $\widetilde{id}_{\tU_{i'},\tU_{i'}}:=1_{\tU_{i'}}$.
\end{itemize}

It is easy to see that all the axioms of definition \ref{compatible-system} are satisfied; as an useful notation,
we will write an index as $i'$ if it belongs to the set $I'$ (and so also to $I$) and with $i$ if it belongs to
$I$ and we don't know whether $\unif{i}$ belongs to $\mathcal{U}'$ or not.\\

In the following pages we will use the following objects and morphisms obtained by applying the 2-functor $F$
described in the previous section:

\begin{itemize}
\item $\groupR$ is the groupoid object associated to the orbifold atlas $\mathcal{U}$;

\item $\groupRR$ is the groupoid object associated to the orbifold atlas $\mathcal{U}'$;

\item $(\psi,\Psi):(\groupRR)\rightarrow(\groupR)$ is the morphism
between groupoid objects associated to the compatible system $\widetilde{id}$.
\end{itemize}

A direct calculation shows that $\psi$ is just the inclusion of $U'=\coprod_{i'\in I'}\tU_{i'}$ in
$U=\coprod_{i\in I}\tU_i$ and that $\Psi$ is the \emph{holomorphic} map that associates to every point
$[\LA{k'i'},\tx_{k'},\LA{k'j'}]$ the same point, but considered in $R$ instead of $R'$.

Our aim is to prove the following result:

\begin{lem}
The morphism $(\psi,\Psi)$ is a Morita equivalence.
\end{lem}

\begin{proof}
We have to verify the axioms of the previous definition, so let us first focus our attention on the map $t\circ
\pi_1$ defined on the fiber product. Set-theoretically (and up to natural bijections), we have that:

\begin{eqnarray*}
& \fibra{R}{U}{U'}=\{(r,u')\in R\times U'\textrm{ s.t. }s(r)=\psi(u')\}= &\\
& =\left\{\Big([\LA{ki},\tx_k,\LA{kj}],(\tx_{i'},\tU_{i'})\Big)\textrm{ s.t. } (\LA{ki}(\tx_k),\tU_i)=
   (\tx_{i'},\tU_{i'})\right\}= &\\
& =\left\{\Big([\LA{ki'},\tx_k,\LA{kj}],(\tx_{i'},\tU_{i'})\Big)\textrm{ s.t. } \LA{ki'}(\tx_k)=
   \tx_{i'}\right\}. &
\end{eqnarray*}

Now for every point $\Big([\LA{ki'},\tx_k,\LA{kj}],(\tx_{i'},\tU_{i'})\Big)$ in this set we get that:

$$t\circ\pi_1\Big([\LA{ki'},\tx_k,\LA{kj}],(\tx_{i'},\tU_{i'})\Big)=t\Big([\LA{ki'},\tx_k,\LA{kj}]\Big)
=(\LA{kj}(\tx_k),\tU_j).$$

We want to prove that $t\circ\pi_1$ \emph{is surjective}, so let us fix any point $(\tx_j,\tU_j)\in U$ and let us
consider the point $\pi_j(\tx_j)\in X$; by hypothesis $\mathcal{U}'$ is an orbifold atlas for $X$, so there exists
a uniformizing system $\unif{i'}$ in $\mathcal{U}'$ for an open neighborhood of this point in $X$. Now
$\mathcal{U}$ contains both $\unif{i'}$ and $\unif{j}$, so by remark \ref{useful-remark} there exists a
uniformizing system $\unif{k}$ in $\mathcal{U}$, a point $\tx_k$ in $\tU_k$ and embeddings $\LA{ki'}$ and
$\LA{kj}$ such that $\LA{kj}(\tx_k)=\tx_j$. Then if we call $\tx_{i'}:=\LA{ki'}(\tx_k)$ we get that the point:

$$\Big([\LA{ki'},\tx_k,\LA{kj}],(\tx_{i'},\tU_{i'})\Big)$$

belongs to the fiber product $\fibra{R}{U}{U'}$. Moreover, $t\circ\pi_1$ applied to it is exactly the point
$(\tx_j,\tU_j)$ we have fixed. Hence we have proved that $t\circ\pi_1$ \emph{is surjective}.\\

Now by construction $\psi$ is an embedding between manifolds of the same dimension, so in particular it is \'etale; 
we recall that (\ref{eq-52}) is a cartesian diagram, hence by proposition \ref{fiber-products-1} we get that also
$\pi_1$ is \'etale. Moreover, $t$ is \'etale by lemma \ref{etale-groupoid}, hence $\pi_1\circ t$ is \'etale, hence
in particular it is a submersion. so we have completely proved condition (i).\\

Let us pass to condition (ii) and let us consider the following diagram:

\begin{equation}\label{eq-95}
\begin{tikzpicture}[scale=0.8]
    \def\x{1.5}
    \def\y{-1.2}
    \node (A0_0) at (-0.6, 0.2*\y) {$R'$};
    \node (A1_2) at (3*\x, 1.2*\y) {$\curvearrowright$};
    \node (A2_1) at (1.3*\x, 2.5*\y) {$\curvearrowright$};
    \node (A2_2) at (2*\x, 2*\y) {$\fibre{R}{(s,t)}{(\psi\times\psi)}{U'\times U'}$};
    \node (A2_4) at (5*\x, 2*\y) {$R$};
    \node (A3_3) at (3.5*\x, 3*\y) {$\square$};
    \node (A4_2) at (2*\x, 4*\y) {$U'\times U'$};
    \node (A4_4) at (5*\x, 4*\y) {$U\times U.$};
    \path (A4_2) edge [->] node [auto,swap] {$\scriptstyle{(\psi\times\psi)}$} (A4_4);
    \path (A0_0) edge [->,dashed] node [auto] {$\scriptstyle{\eta}$} (A2_2);
    \path (A2_2) edge [->] node [auto,swap] {$\scriptstyle{pr_1}$} (A2_4);
    \path (A2_4) edge [->] node [auto] {$\scriptstyle{(s,t)}$} (A4_4);
    \path (A0_0) edge [->,bend right=15] node [auto,swap] {$\scriptstyle{(s',t')}$} (A4_2);
    \path (A0_0) edge [->,bend left=15] node [auto] {$\scriptstyle{\Psi}$} (A2_4);
    \path (A2_2) edge [->] node [auto] {$\scriptstyle{pr_2}$} (A4_2);
\end{tikzpicture}
\end{equation}

where the external diagram is commutative because it is just diagram (\ref{eq-53}). Since both $\Psi$ and $(s',t')$
are holomorphic, by the universal property of the fiber product there exists a unique \emph{holomorphic}
map $\eta$ that makes the diagram commute.\\

It is easy to give an explicit description of the fiber product in (\ref{eq-95}) as the set:

\begin{equation}\label{eq-96}
\left\{\Big([\LA{ki'},\tx_k,\LA{kj'}],(\tx_{i'},\tU_{i'}),(\tx_{j'},\tU_{j'}) \Big) \textrm{ s.t }
\LA{ki'}(\tx_k)=\tx_{i'}\AND\LA{kj'}(\tx_k)=\tx_{j'}\right\}
\end{equation}

and of the set map $\eta$ as:

$$[\LA{k'i'},\tx_{k'},\LA{k'j'}]\mapsto
\Big([\LA{k'i'},\tx_{k'},\LA{k'j'}],(\tx_{i'},\tU_{i'}),(\tx_{j'},\tU_{j'}) \Big).$$

Now we want to prove that $\eta$ is surjective, so let us fix any point as in (\ref{eq-96}). Then  
$\pi_{i'}\circ\LA{ki'}(\tx_k)=\pi_k(\tx_k)=\pi_{j'}\circ\LA{kj'}(\tx_k)$; since $\mathcal{U}'$ is an orbifold atlas,
there exists a point $[\LA{l'i'},\tx_{l'},\LA{l'j'}]$ in $R'$ with source and
target coinciding with $[\LA{ki'},\tx_k,\LA{kj'}]$. Hence we can apply lemma \ref{useful-lemma-2} in $R$ 
and we get that there exists a unique $g\in G_{j'}$ such that:

$$[\LA{ki'},\tx_k,\LA{kj'}]=[\LA{l'i'},\tx_{l'},g\circ\LA{l'j'}].$$

Then we have that:

$$\Big([\LA{ki'},\tx_k,\LA{kj'}],(\tx_{i'},\tU_{i'}),(\tx_{j'},\tU_{j'}) \Big)=
\eta([\LA{l'i'},\tx_{l'},g\circ\LA{l'j'}]).$$

So $\eta$ is surjective; let us also prove injectivity, so let us suppose that 2 points of the form:

\begin{equation}\label{eq-54}
[\LA{k'i'},\tx_{k'},\LA{k'j'}]\AND[\LA{l'i'},\tx_{l'},\LA{l'j'}] 
\end{equation}

of $R'$ are identified in the fiber product by $\eta$. By applying lemma \ref{useful-lemma-2}
in $R'$ we get that there exists a unique $g\in G_{j'}$ such that:

$$[\LA{k'i'},\tx_{k'},\LA{k'j'}]=[\LA{l'i'},\tx_{l'},g\circ\LA{l'j'}]$$

in $R'$. Hence, in particular, the same relation holds in $R$; on the other side, we have assumed that in $R$
the points of (\ref{eq-54}) are identified. Hence in $R$ we have:

$$[\LA{l'i'},\tx_{l'},g\circ\LA{l'j'}]=[\LA{l'i'},\tx_{l'},\LA{l'j'}].$$

By the uniqueness part of lemma \ref{useful-lemma-2} in $R$ we get that necessarily $g=1_{\tU_{j'}}$,
so the points of (\ref{eq-54}) were already the same in $R'$. Hence we get that $\eta$ is \emph{invertible}.\\

Since it is holomorphic, we have proved that $\eta$ is a biholomorphism, so by the universal
property of fiber products we have that (\ref{eq-53}) is cartesian.
\end{proof}

\begin{prop}\label{key-point}
Suppose we have fixed two \emph{equivalent} orbifold atlases $\mathcal{U}_1$ and $\mathcal{U}_2$ on a topological
space $X$ and let us call $\group{i}$ \emph{(}for $i=1,2$\emph{)} the groupoid objects associated to them by the
2-functor $F$ described in the previous chapter. Then these two groupoid objects are Morita equivalent.
\end{prop}

\begin{proof}
It suffices to consider the maximal manifold atlas $\mathcal{U}_3$ (associated to both)
and to apply the previous lemma twice.
\end{proof}

\begin{prop}\label{quasi-injectivity}
Suppose we have fixed two orbifold atlases $\mathcal{U}$ on $X$ and $\mathcal{U}'$ on $X'$ and assume that
they are equivalent with respect to definition \ref{new-equivalence}. Then their images via the 
2-functor $F$ are Morita equivalent. 
\end{prop}

\begin{proof}
Indeed using the explicit construction of \S 3.1, one can easily see that $F$ identifies $\mathcal{U}$ and
$\varphi_{\ast}(\mathcal{U})$ for every orbifold atlas $\mathcal{U}$ on $X$ and any homeomorphism $\varphi:
X\rightarrow X'$ (actually, this is the only point where $F$ fails to be injective on objects). Then it
suffices to apply the previous proposition on $X'$.
\end{proof}

\subsection{Surjectivity up to Morita equivalences}

The aim of this subsection is to prove that the 2-functor $F$ is surjective up to Morita equivalences. In order
to do that, we divide the proof in some lemmas as follows.

\begin{lem}
Let us fix any proper \'etale and effective groupoid object $\groupR$ in $\CAT{Manifolds}$ and let us define a
relation of equivalence $\ca{R}$ on $U$ as follows:

$$\tx\,\ca{R}\,\ty\,\stackrel{\textrm{def}}{\Longleftrightarrow}\exists\,g\in R\,\textrm{ s.t. } s(g)=\tx\AND
t(g)=\ty.$$

Let us call $X:=U/\ca{R}$ (with the quotient topology) and $\pi:U\rightarrow X$ the quotient map. Since $(s,t)$
is proper, then $X$ is Hausdorff and paracompact; we claim that we can define an orbifold atlas $\mathcal{U}$
on it.
\end{lem}

\begin{proof}(adapted from \cite{MP}, last part of theorem 4.1)
Let us fix any point $\tx\in U$ and let us consider the set $R_{\tx}:=(s,t)^{-1}\{(\tx,\tx)\}$ in $R$. In definition
\ref{effective-groupoid} we have already proved that this is a finite set and that for every point $g$ in it we can
find a sufficiently small open neighborhoods $W_g$ where $s$ and $t$ are both invertible; moreover, since $R_{\tx}$
is finite, we can restrict them such that $W_g\cap W_h=\varnothing$ for all $g\neq h$ in $R_{\tx}$. Let us define
$\tA_{\tx}:=\cap_{g\in R_{\tx}}s(W_g)$, which is an open neighborhood of $\tx$ in $U$; now $R\smallsetminus
\cup_{g\in R_{\tx}}W_g$ is closed so its image via $(s,t)$ (which is proper, hence closed) is closed in $U\times U$
and does not contain $(\tx,\tx)$. Since a basis for the topology of $U\times U$ is given by products of open sets
of $U$, there exists $\tB_{\tx}$ open neighborhood of $\tx$ contained in $\tA_{\tx}$ and such that:

\begin{equation}\label{eq-200}
(\tB_{\tx}\times\tB_{\tx})\cap(s,t)\Big(R\smallsetminus\bigcup_{g\in R_{\tx}}W_g\Big)=\varnothing. 
\end{equation}

Hence whenever $h\in R$ is such that both $s(h)$ and $t(h)$ belong to $\tB_x$, then $h\in W_g$ for a \emph{unique}
$g\in R_{\tx}$. Now for every $g\in R_{\tx}$ we can define $\tg:=t\circ(s|_{W_g})^{-1}:s(W_g)\rightarrow t(W_g)$,
which is a biholomorphic map between open sets of $U$, both containing the point $\tx$. Since $\tB_{\tx}\subseteq
\tA_{\tx}$, then $\tB_{\tx}\subseteq s(W_g)$ for all $g\in R_{\tx}$, hence it makes sense to consider 
$\tC:=\bigcap_{g\in R_{\tx}}\tg(\tB_{\tx})$
which is an open neighborhood of $\tx$ (since every $\tg$ is biholomorphic and maps $\tx$ to it self). Now
let us fix any $h\in R_{\tx}$; if we apply part (a) of lemma \ref{lemma-a} we get that:

$$\th(\tC)=h\Big(\bigcap_{g\in R_{\tx}}\tg(\tB_{\tx})\Big)=\bigcap_{g\in R_{\tx}}\th\circ\tg(\tB_{\tx})=
\bigcap_{m(g,h)\in R_{\tx}}\widetilde{m(g,h)}(\tB_x)=\tC.$$

Hence if we define:

$$G_{\tx}:=\{\tg:\tC\stackrel{\sim}{\rightarrow}\tC\}_{g\in R_{\tx}},$$

this is a finite group (with multiplication given by composition of functions, or, equivalently, using lemma
\ref{lemma-a}) of holomorphic automorphisms on $\tC$. If we call $\tU_{\tx}$ the connected component of $\tC$ that
contains $\tx$, by continuity we get that $G_{\tx}$ is again a group of holomorphic automorphisms of $\tU_{\tx}$.
Moreover, using (\ref{eq-200}) one can easily prove that on $\tU_{\tx}$ the relation $\ca{R}$ 
coincides with the equivalence relation induced by the action of $G_{\tx}$. 
Hence we get that $(\tU_{\tx},G_{\tx},\pi)$ is a uniformizing system for the open set $\pi(\tU_{\tx})$ 
of the topological space $X$.\\

Now we make the same construction also \emph{for every open neighborhood} $\tB'_{\tx}\subseteq\tB_{\tx}$ \emph{of}
$\tx$ and we get uniformizing systems of the form $(\tU'_{\tx},G_{\tx},\pi)$ with natural embeddings (given by
inclusions) into the previous one. If we apply this construction for every point $\tx\in U$ we get a family
$\mathcal{U}$ of \emph{arbitrarily small} uniformizing systems for $X$ which clearly satisfies axiom (i) of
definition \ref{orb-atlas}. A direct check
proves that this family satisfies also axiom (ii), so $\mathcal{U}$ is an orbifold atlas on $X$.\\
\end{proof}

For simplicity, we rename all the charts of $\mathcal{U}$ in order to make them of the form $\unif{i}$ (since
we will not need to know the index ``$\tx$'' used to construct them).
Let us call $\groupRR$ the groupoid object associated to $\mathcal{U}$ by the 2-functor $F$ and let us fix
any point $P$ in $R'$. Since $\mathcal{U}$ contains arbitrarily small open neighborhoods of every point,
we get that $P$ has the form $[j,\tx_k,\LA{ki}]$ where $\tx_k\in\tU_k\subseteq\tU_i$ and $j:\tU_k
\rightarrow\tU_i$ is the inclusion. Now we have that $\pi_i(\tx_k)=\pi_j(\tx_j)$, so there exists 
a point $h\in R$ such that $s(h)=\tx_k$ and $t(h)=\tx_j$. Eventually by restricting $\tU_k$ we can
assume that $\tU_k\subseteq\tW_g\cap\th^{-1}(\tU_j)$,
so it makes sense to consider $\th$ as an open embedding from $\tU_k$
to $\tU_j$. Moreover, by applying the definition of this map it is easy to see that $\th$ is an embedding from
$\unif{k}$ to $\unif{j}$. Now let us consider the point $Q:=[j,\tx_k,\th]$; if we apply lemma
\ref{useful-lemma-2} (in $R'$ instead of $R$) together with lemma \ref{lemma-a} we get the following result:

\begin{lem}\label{lemma-b}
For every point $P=[\LA{ki},\tx_k,\LA{kj}]$ of $R'$ there exists a unique $g$ in $R$ such that:
\begin{itemize}
 \item $s(g)=\LA{ki}(\tx_k)$,
 \item $t(g)=\LA{kj}(\tx_k)$,
 \item $P=[j,\tx_k,\tg]$.
\end{itemize}
\end{lem}

Our aim is to prove that there is a Morita equivalence $(\psi,\Psi)$ from $\groupR$ to
$\groupRR$. First of all, let us define $\psi:U'\rightarrow U$ as $\psi(\tx_i,\tU_i):=\tx$.
Since $U'$ is the disjoint union of open sets of the form $\tU_i$, we get that in the natural coordinates this
map locally coincides with the identity, so $\psi$ is holomorphic. 

\begin{lem}
Lemma \ref{lemma-b} induces a biholomorphic map:
$$\eta:R'\rightarrow \fibre{R}{(s,t)}{(\psi\times\psi)}{U'\times U'}.$$
\end{lem}

\begin{proof}
Take any point $P$ in $R'$ and define $g\in R$ as in the previous lemma. Then define:

$$\eta([j,\tx_k,\tg]):=\Big(g,s'(P),t'(P)\Big);$$

since 

$$(s,t)(g)=\Big((\tx_k,\tg(\tx_k),(\tg(\tx_k,\tU_j)\Big)=(\psi\times\psi)(s'(P),t'(P)),$$

we get that actually $\eta$ 
has values in the fiber product. By the universal property of the fiber product, in order to prove that it
is holomorphic it is sufficient to prove that $\eta$ is so if composed with the natural projections to
$R'$ and $U'\times U'$. The second composition is just the map $(s',t')$, which is holomorphic, so let us consider 
only the first map, given by lemma \ref{lemma-b}. It is obviuos that for every $P=[j,\tx_k,\tg]$, the point $g$ can
be obtained as $(s_{|W_g})^{-1}(P)$; now for points $P$ and $Q=[j',\ty_l,\th]$ sufficiently near we get that $h$
belongs to $W_g$, so $h$ is equal to $(s_{|W_g})^{-1}(Q)$. Hence locally the map $pr_1\circ\eta$ coincides with the
local inverse of $s$, which is \'etale by hypothesis. So we have proved that $\eta$ is holomorphic. In order to prove
that it is biholomorphic, it suffices to prove that it is invertible. So let us describe explicitly its inverse:

$$\gamma:\fibre{R}{(s,t)}{(\psi\times\psi)}{U'\times U'}\rightarrow R'$$

as follows: let us take any point $\Big(g,(\tx_i,\tU_i)(\tx_j,\tU_j)\Big)$ in the fiber product, so we have
$s(g)=\psi(\tx_i,\tU_i)=\tx_i$ and $t(g)=\tx_j$. Eventually by restricting the open neighborhood $W_g$ of $g$ in $R$,
we can assume that $s(W_g)\subseteq\tU_i$ and $t(W_g)\subseteq\tU_j$. Hence it makes sense to consider $\tg$ as a map
from an open neighboorhood $\tU_k$ of $\tx_i$ in $\tU_i$ to $\tU_j$. Moreover, this will be an embedding between
uniformizing systems, so we can define:

$$\gamma\Big(g,(\tx_i,\tU_i),(\tx_j,\tU_j)\Big):=[j,\tx_i,\tg]$$

(where $j$ is the inclusion of $\tU_k$ in $\tU_i$); a direct check proves that $\gamma$ does not depend on the
choice of $\tU_k$ and that it is the inverse of $\eta$.
\end{proof}

Now we define the holomorphic map $\Psi:=pr_1\circ\eta:R'\rightarrow R$ and a direct computation shows that the pair
$(\psi,\Psi)$ is a morphism of groupoid objects from $\groupRR$ to $\groupR$. 

\begin{lem}
The morphism $(\psi,\Psi)$ is a Morita equivalence.
\end{lem}

\begin{proof}
We have to verify axioms (i) and (ii) of definition \ref{morita-equivalence}.
First of all, let us consider the map $\psi:U'\rightarrow U$; we have already said that up to an holomorphic
change of coordinates $\psi$ locally coincides with the identity, hence $\psi$ is \'etale. Moreover, it is clearly
surjective (because the domains of the charts of $\mathcal{U}$ cover all $U$. Hence $\psi$ is \'etale and
surjective, so also the map $\pi_1$ of diagram (\ref{eq-52}) is \'etale (see proposition \ref{fiber-products-1})
and surjective. Moreover, $t$ is \'etale and surjective by definition of object in $\CAT{Grp}$, hence $t\circ\pi_1$
is surjective and \'etale, so in particular it is a surjective submersion. Hence (i) is proved.\\

Now let us pass to (ii): from the previous constructions we get a commutative diagram of the form:

\[\begin{tikzpicture}[scale=0.8]
    \def\x{1.5}
    \def\y{-1.2}
    \node (A0_0) at (-0.6, 0.2*\y) {$R'$};
    \node (A1_2) at (3*\x, 1.2*\y) {$\curvearrowright$};
    \node (A2_1) at (1.3*\x, 2.5*\y) {$\curvearrowright$};
    \node (A2_2) at (2*\x, 2*\y) {$\fibre{R}{(s,t)}{(\psi\times\psi)}{U'\times U'}$};
    \node (A2_4) at (5*\x, 2*\y) {$R$};
    \node (A3_3) at (3.5*\x, 3*\y) {$\square$};
    \node (A4_2) at (2*\x, 4*\y) {$U'\times U'$};
    \node (A4_4) at (5*\x, 4*\y) {$U\times U.$};
    \path (A4_2) edge [->] node [auto,swap] {$\scriptstyle{(\psi\times\psi)}$} (A4_4);
    \path (A0_0) edge [->] node [auto] {$\scriptstyle{\eta}$} (A2_2);
    \path (A2_2) edge [->] node [auto,swap] {$\scriptstyle{pr_1}$} (A2_4);
    \path (A2_4) edge [->] node [auto] {$\scriptstyle{(s,t)}$} (A4_4);
    \path (A0_0) edge [->,bend right=15] node [auto,swap] {$\scriptstyle{(s',t')}$} (A4_2);
    \path (A0_0) edge [->,bend left=15] node [auto] {$\scriptstyle{\Psi}$} (A2_4);
    \path (A2_2) edge [->] node [auto] {$\scriptstyle{pr_2}$} (A4_2);
\end{tikzpicture}\]

The internal square is cartesian by construction; since $\eta$ is a biholomorphism, we get that also
the external diagram is cartesian, hence (ii) is proved.
\end{proof}

\begin{prop}\label{quasi-ess-surjectivity}
The 2-functor $F$ is surjective up to Morita equivalences.
\end{prop}

\begin{proof}
For every $\groupR$ in $\CAT{Grp}$, we have described an orbifold atlas $\mathcal{U}$ and
and we have proved that $F(\mathcal{U})=(\groupRR)$ is Morita equivalent to $\groupR$, so we are done.
\end{proof}

\subsection{A natural bijection on classes of objects in source and target}
Using proposition \ref{quasi-injectivity} one can induce from the 2-functor $F$ a natural set map:

$$\tilde{F}:\begin{array}{c c c} 

\left\{\begin{array}{c}
\textrm{equivalence classes of}\\
\textrm{objects of } \CAT{Pre-Orb}\\
\textrm{ w.r.t. definition \ref{new-equivalence}}
\end{array}\right\}

& \rightarrow &

\left\{\begin{array}{c}
\textrm{equivalence classes of}\\
\textrm{objects of }\CAT{Grp}\textrm{ w.r.t}\\
\textrm{Morita equivalences}
\end{array}\right\}.

\end{array}
$$

Our aim is to prove that the set map $\tilde{F}$ is a bijection. In order to do that, let us
state and prove some preliminary results.

\begin{lem}
Let us fix any pair of orbifold atlases $\mathcal{U}=\{\unif{i}\}_{i\in I}$ on $X$ and $\mathcal{U}'=
\{(\tV_j,H_j,\xi_j)\}_{j\in J}$ on $X'$, let us call $\groupR:=F(\mathcal{U})$ and $\groupRR:=F(\mathcal{U}')$ and
let us suppose we have a Morita equivalence:

$$(\psi,\Psi):(\groupRR)\rightarrow(\groupR).$$

Then $\mathcal{U}$ and $\mathcal{U'}$ are equivalent w.r.t. definition \ref{new-equivalence}.
\end{lem}

\begin{proof}
Let us recall that by definition of $F$, $U:=\coprod_{i\in I}\tU_i$, so we can define a continuous map $\pi:U
\rightarrow X$ that for every $i\in I$ coincides with the continuous map $\pi_i$ on the connected component $\tU_i$.
Since $\pi$ locally coincides with the maps $\pi_i$'s, we get that $\pi$ is not only continuous, but also open.
Analogously, we can define a continuous open map $\pi':U'\rightarrow X'$. Now let us consider the diagram:

\begin{equation}\label{eq-59}
\begin{tikzpicture}[scale=0.8]
    \def\x{1.5}
    \def\y{-1.2}
    \node (A0_0) at (0*\x, 0*\y) {$U'$};
    \node (A0_2) at (2*\x, 0*\y) {$U$};
    \node (A2_0) at (0*\x, 2*\y) {$X'$};
    \node (A2_2) at (2*\x, 2*\y) {$X$;};
    \path (A0_0) edge [->] node [auto] {$\scriptstyle{\psi}$} (A0_2);
    \path (A0_0) edge [->] node [auto,swap] {$\scriptstyle{\pi'}$} (A2_0);
    \path (A0_2) edge [->] node [auto] {$\scriptstyle{\pi}$} (A2_2);
    \path (A2_0) edge [->,dashed] node [auto,swap] {$\scriptstyle{\varphi}$} (A2_2);
\end{tikzpicture}
\end{equation}

let us fix any point $x'$ in $X'$ and let $u',\bar{u}'\in U'$ be any pair of its preimages via $\pi'$.
Then by definition of $R'$, there exists a point $r'$ in $R'$ such that $s'(r')=u'$ and $t'(r')=\bar{u}'$. Then:

$$\psi(u')=\psi(s'(r'))=s(\Psi(r'))\AND \psi(\bar{u}')=\psi(t'(r')=t(\Psi(r'))$$

So $\pi(\psi(u'))=\pi(\psi(\bar{u}'))$, hence we can induce a well defined map $\varphi:X'\rightarrow X$ as
$\varphi(\pi'(u')):=\pi(\psi(u'))$. Since $\psi$ is continuous, so is $\varphi$ by the universal property
of the quotient topology.\\

Now let us consider the fiber product (\ref{eq-52}); since $\groupR$ belongs to
$\CAT{Grp}$, both $s$ and $t$ are \'etale, so using proposition \ref{fiber-products-1} we get that also $\pi_2$ is
\'etale. Moreover, by property (i) of Morita equivalences we get that $t\circ\pi_1$ is a submersion; since $t$ is
\'etale, we get that $\pi_1$ is a submersion. Hence $s\circ\pi_1$ is also a submersion, so by commutativity of
(\ref{eq-52}), we get that $\psi\circ\pi_2$ is a submersion. We have already proved that $\pi_2$ is \'etale, hence
$\psi$ is a submersion. Moreover, it is easy to prove that $U$ and $U'$ have the same complex dimension, hence
$\psi$ is \'etale, so in particular it is open; hence using diagram \ref{eq-59} we get that $\varphi$ is also open.\\

Now we want to find an inverse $\delta$ for $\varphi$. In order to do that, let us fix any point $\pi(u)\in X$ and let 
us consider again diagram (\ref{eq-52}). Since $t\circ\pi_1$ is surjective, there exists a point $(r,u')\in
\fibra{R}{U}{U'}$ such that $t(r)=u$; the point we have chosen belongs to the fiber product, so $\psi(u')=s(r)$.
Hence by definition of $\pi$ we get that:

$$\pi(u)=\pi(t(r))=\pi(s(r))=\pi(\psi(u')).$$

So we have proved that every point in $X$ is of the form $\pi(\psi(u'))$ for some $u'\in U'$; hence we want
to define the set map $\delta$ as $\delta(\pi(\psi(u')))=\pi'(u')$. In order to prove that it is well defined,
let us suppose that $\pi(\psi(u'))=\pi(\psi(\bar{u}'))$; then there exists $r\in R$ such that $s(r)=\psi(u')$
and $t(r)=\psi(\bar{u}')$. Using the fact that diagram (\ref{eq-53}) is cartesian by definition of Morita equivalence,
we get that there exists a point $r'\in R'$ such that $s'(r')=u'$ and $t(r')=\bar{u}'$, so $\pi'(u')=\pi'(\bar{u}')$.
Hence $\delta$ is well defined; moreover, a direct check proves that it is actually the inverse of $\varphi$. In addition,
it is continuous since $\varphi$ is open, so $\varphi$ is an homeomorphism. So we have proved condition (a) of
definition \ref{new-equivalence}.\\

We want to prove also condition (b); in order to do that, it suffices to prove that for every point $x$ of
$X$ there exists a pair of uniformizing systems in $\mathcal{U}$ and $\varphi_{\ast}(\mathcal{U}')$ which
are compatible at this point. So let $x=\pi_i(\tx_i,\tU_i)$ be any such point; by what we said before, there
exists a point $(\tx'_l,\tU'_l)\in U'$ such that if we call $(\tx_j,\tU_j):=\psi(\tx'_l,\tU'_l)$ we get that:

$$\pi(\tx_i,\tU_i)=\pi(\tx_j,\tU_j);$$

hence there exist $r\in R$ such that $s(r)=(\tx_i,\tU_i)$ and $t(r)=(\tx_j,\tU_j)$. We have already proved that
$\psi,s$ and $t$ are \'etale, so there exists an open
neighboorhood $\tA$ of $(\tx'_l,\tU_l)$ in $U'$ such that the map $s\circ t^{-1}\circ\psi$ is an holomorphic embedding from
$\tA$ to $\tU_i$. Then we can apply lemma \ref{unif-sys-induced} for $\tx'_l$ and $\tA$ and we get a uniformizing
system around $\pi_l(\tx'_l)$; it is easy now to prove that this chart has natural embeddings in both $\unif{i}\in\mathcal{U}$
and $(\tV_l,H_l,\varphi\circ\phi_l)\in\varphi_{\ast}(\mathcal{U}')$. So we have proved that for every point of $X'$
the orbifold atlases $\mathcal{U}'$ and $\varphi_{\ast}(\mathcal{U})$ are equivalent at that point, so condition (b)
is satisfied.
\end{proof}

\begin{prop}\label{injectivity-up-to-equivalences}
Let us fix any pair of orbifold atlases $\mathcal{U}$ and $\mathcal{U}''$ on $X$ and $X''$ respectively and let us
suppose that $F(\mathcal{U})$ and $F(\mathcal{U}'')$ are Morita equivalent. Then $\mathcal{U}$ and $\mathcal{U}''$
are equivalent w.r.t. definition \ref{new-equivalence}.
\end{prop}

\begin{proof}
By definition \ref{weak-equivalence} we get that there exists an object $\groupR$ of $\CAT{Grp}$ and a pair of
Morita equivalences as follows:

$$F(\mathcal{U})\leftarrow(\groupR)\rightarrow F(\mathcal{U}'').$$

Using the proof of proposition \ref{quasi-ess-surjectivity} we get that there exist a space $X$, an orbifold atlas
$\mathcal{U}'$ on it and a Morita equivalence:

$$F(\mathcal{U}')\rightarrow(\groupR).$$

It is easy to prove that Morita equivalences are closed under composition, hence we get Morita equivalences:

$$F(\mathcal{U})\leftarrow F(\mathcal{U}')\rightarrow F(\mathcal{U}'').$$

Then it suffices to apply twice the previous lemma (and the fact that definition \ref{new-equivalence} gives a
relation of equivalence) in order to prove that $\mathcal{U}$ is equivalent to $\mathcal{U}''$.
\end{proof}

\begin{teo}
The set map $\tilde{F}$ is a bijection.
\end{teo}

\begin{proof}
Surjectivity of the map $\tilde{F}$ is just proposition \ref{quasi-ess-surjectivity}; injectivity is proposition
\ref{injectivity-up-to-equivalences}.
\end{proof}

Hence we have proved that the classes of complex reduced orbifold atlases (w.r.t. definition \ref{new-equivalence})
and the classes of proper \'etale effective groupoid objects (w.r.t. to Morita equivalences) are different
description of the same geometric objects. It will be very useful to prove that a similar result holds also on the
level of morphisms and 2-morphisms, but for the moment we have no idea of how this can be made. If this can be
done, we will have a very useful tool to translate results about orbifolds in results about groupoid
objects, and conversely.

\section*{Appendix - Some technical proofs}
Here are the proofs of some technical lemmas cited in the previous sections.\\

\textbf{Lemma \ref{equivalence-CR-proof}.}\emph{ The relation of definition \ref{equivalence-CR} is an equivalence relation.}

\begin{proof}
This relation is clearly symmetric and reflexive, so let us prove only transitivity. So let us suppose that we have 3 atlases
$\mathcal{U}_1,\mathcal{U}_2$ and $\mathcal{U}_3$ on $X$ with the second one equivalent to both the first and the third one.
Let us fix any $x\in X$; by definition there exist:
\begin{itemize}
 \item 2 uniformizing systems $(\tV_1,H_1,\psi_1)$ and $(\tV_3,H_3,\psi_3)$ around $x$;
 \item 2 uniformizing systems $(\tU_i,G_i,\pi_i)\in\mathcal{U}_i$ for $i=1,3$ and 2 uniformizing systems 
 $(\tU_2,G_2,\pi_2),(\tU'_2,G'_2,\pi'_2)\in\mathcal{U}_2$.
 \item 4 embeddings $\lambda_1,\lambda'_1,\lambda_3,\lambda'_3$ as follows:
\end{itemize}

\[\begin{tikzpicture}
    \def\x{1.5}
    \def\y{-1.2}
    \node (A0_0) at (0*\x, 0*\y) {$(\tU_1,G_1,\pi_1)$};
    \node (A0_2) at (2*\x, 0*\y) {$(\tU_2,G_2,\pi_2)$};
    \node (A0_4) at (4*\x, 0*\y) {$(\tU'_2,G'_2,\pi)$};
    \node (A0_6) at (6*\x, 0*\y) {$(\tU_3,G_3,\pi_3).$};
    \node (A1_1) at (1*\x, 1*\y) {$(\tV_1,H_1,\psi_1)$};
    \node (A1_3) at (3*\x, 1*\y) {$(\tV_2,H_2,\psi_2)$};
    \node (A1_5) at (5*\x, 1*\y) {$(\tV_3,H_3,\psi_3)$};
    
    \path (A1_3) edge [->,dashed] node [auto] {$\scriptstyle{\lambda_2}$} (A0_2);
    \path (A1_3) edge [->,dashed] node [auto,swap] {$\scriptstyle{\lambda'_2}$} (A0_4);
    \path (A1_1) edge [->] node [auto,swap] {$\scriptstyle{\lambda'_1}$} (A0_2);
    \path (A1_1) edge [->] node [auto] {$\scriptstyle{\lambda_1}$} (A0_0);
    \path (A1_5) edge [->] node [auto] {$\scriptstyle{\lambda_3}$} (A0_4);
    \path (A1_5) edge [->] node [auto,swap] {$\scriptstyle{\lambda'_3}$} (A0_6);
\end{tikzpicture}\]

Since both $(\tU_2,G_2,\pi_2)$ and $(\tU'_2,G'_2,\pi'_2)$ belong to the atlas $\mathcal{U}_2$, there exists a third
uniformizing system $(\tV_2,H_2,\psi_2)$ around $x$ in $\mathcal{U}_2$, together with embeddings $\lambda_2,\lambda'_2$
as in the previous diagram. If we have fixed any point $\tx_2\in\tV_2$ such that $\psi_2(\tx_2)=x$, by composing with
a suitable automorphism of $\tU_2$ there is no loss of generality in assuming that $\lambda(\tx_2)$ belongs to the image
of $\lambda'_1$; hence the set $(\lambda'_1)^{-1}(\lambda_2(\tV_2))$ is an open neighboorhood of a preimage of $x$ in $\tV_1$.
So by applying lemma \ref{unif-sys-induced} we get a uniformizing system $(\tW_1,K_1,\xi_1)$ around $x$, together with an
embedding $\mu_1$ of it into $\unif{1}$. Then the set map $\mu'_1:=\lambda_2^{-1}\circ\lambda'_1\circ\mu_1$ is a well
defined holomorphic embedding of $(\tW_1,K_1,\xi_1)$ into $(\tV_2,H_2,\psi_2)$. Analogously, we can complete also the
right part of the previous diagram and we obtain something of this form:

\[\begin{tikzpicture}
    \def\x{1.5}
    \def\y{-1.2}
    \node (A0_0) at (0*\x, 0*\y) {$(\tU_1,G_1,\pi_1)$};
    \node (A0_2) at (2*\x, 0*\y) {$(\tU_2,G_2,\pi_2)$};
    \node (A0_4) at (4*\x, 0*\y) {$(\tU'_2,G'_2,\pi)$};
    \node (A0_6) at (6*\x, 0*\y) {$(\tU_3,G_3,\pi_3).$};
    \node (A1_1) at (1*\x, 1*\y) {$(\tV_1,H_1,\psi_1)$};
    \node (A1_3) at (3*\x, 1*\y) {$(\tV_2,H_2,\psi_2)$};
    \node (A1_5) at (5*\x, 1*\y) {$(\tV_3,H_3,\psi_3)$};
    
    \path (A1_3) edge [->] node [auto] {$\scriptstyle{\lambda_2}$} (A0_2);
    \path (A1_3) edge [->] node [auto,swap] {$\scriptstyle{\lambda'_2}$} (A0_4);
    \path (A1_1) edge [->] node [auto,swap] {$\scriptstyle{\lambda'_1}$} (A0_2);
    \path (A1_1) edge [->] node [auto] {$\scriptstyle{\lambda_1}$} (A0_0);
    \path (A1_5) edge [->] node [auto] {$\scriptstyle{\lambda_3}$} (A0_4);
    \path (A1_5) edge [->] node [auto,swap] {$\scriptstyle{\lambda'_3}$} (A0_6);

    \node (A2_2) at (2*\x, 2*\y) {$(\tW_1,K_1,\xi_1)$};
    \node (A2_4) at (4*\x, 2*\y) {$(\tW_3,K_3,\xi_3)$};

    \path (A2_2) edge [->,dashed] node [auto] {$\scriptstyle{\mu_1}$} (A1_1);
    \path (A2_2) edge [->,dashed] node [auto,swap] {$\scriptstyle{\mu'_1}$} (A1_3);
    \path (A2_4) edge [->,dashed] node [auto] {$\scriptstyle{\mu_3}$} (A1_3);
    \path (A2_4) edge [->,dashed] node [auto,swap] {$\scriptstyle{\mu'_3}$} (A1_5);
\end{tikzpicture}\]

By applying the previous construction another time, we get a diagram like this:

\[\begin{tikzpicture}
    \def\x{1.5}
    \def\y{-1.2}
    \node (A0_0) at (0*\x, 0*\y) {$(\tU_1,G_1,\pi_1)$};
    \node (A0_2) at (2*\x, 0*\y) {$(\tU_2,G_2,\pi_2)$};
    \node (A0_4) at (4*\x, 0*\y) {$(\tU'_2,G'_2,\pi)$};
    \node (A0_6) at (6*\x, 0*\y) {$(\tU_3,G_3,\pi_3).$};
    \node (A1_1) at (1*\x, 1*\y) {$(\tV_1,H_1,\psi_1)$};
    \node (A1_3) at (3*\x, 1*\y) {$(\tV_2,H_2,\psi_2)$};
    \node (A1_5) at (5*\x, 1*\y) {$(\tV_3,H_3,\psi_3)$};
    
    \path (A1_3) edge [->] node [auto] {$\scriptstyle{\lambda_2}$} (A0_2);
    \path (A1_3) edge [->] node [auto,swap] {$\scriptstyle{\lambda'_2}$} (A0_4);
    \path (A1_1) edge [->] node [auto,swap] {$\scriptstyle{\lambda'_1}$} (A0_2);
    \path (A1_1) edge [->] node [auto] {$\scriptstyle{\lambda_1}$} (A0_0);
    \path (A1_5) edge [->] node [auto] {$\scriptstyle{\lambda_3}$} (A0_4);
    \path (A1_5) edge [->] node [auto,swap] {$\scriptstyle{\lambda'_3}$} (A0_6);

    \node (A2_2) at (2*\x, 2*\y) {$(\tW_1,K_1,\xi_1)$};
    \node (A2_4) at (4*\x, 2*\y) {$(\tW_3,K_3,\xi_3)$};

    \path (A2_2) edge [->] node [auto] {$\scriptstyle{\mu_1}$} (A1_1);
    \path (A2_2) edge [->] node [auto,swap] {$\scriptstyle{\mu'_1}$} (A1_3);
    \path (A2_4) edge [->] node [auto] {$\scriptstyle{\mu_3}$} (A1_3);
    \path (A2_4) edge [->] node [auto,swap] {$\scriptstyle{\mu'_3}$} (A1_5);

    \node (A3_3) at (3*\x, 3*\y) {$(\tW_2,K_2,\xi_2)$};

    \path (A3_3) edge [->,dashed] node [auto] {$\scriptstyle{\mu_2}$} (A2_2);
    \path (A3_3) edge [->,dashed] node [auto,swap] {$\scriptstyle{\mu'_2}$} (A2_4);
\end{tikzpicture}\]

By considering the embeddings $\lambda_1\circ\mu_1\circ\mu_2$ and $\lambda'_3\circ\mu'_3\circ\mu'_2$ we get that
the atlases $\mathcal{U}_1$ and $\mathcal{U}_3$ are equivalent at $x$. Since this holds for every point of $X$, we are done.
\end{proof}

\textbf{Lemma \ref{vertical-composition-groupoid}. }\emph{$\beta\odot\alpha$ is a natural transformation from $\groupR$ to
$\groupRR$.}

\begin{proof}
Using definition \ref{nat-tran-group} and property (ii) of definition \ref{groupoid-object} we get that
condition (i) of definition \ref{nat-tran-group} is easily satisfied. Now if we use condition (ii) for $\alpha$ and
$\beta$ we have:

\begin{eqnarray}
\label{eq-10} & m'\circ(\alpha\circ s,\Psi_2)=m'\circ(\Psi_1,\alpha\circ t);& \\
\label{eq-11} & m'\circ(\beta\circ s,\Psi_3)=m'\circ(\Psi_2,\beta\circ t).&
\end{eqnarray}

Hence we obtain:

\begin{eqnarray*}
& m'\circ\Big((\beta\odot\alpha)\circ s,\Psi_3\Big)=m'\circ\Big(m'\circ(\alpha,\beta)\circ s,
  \Psi_3\Big)= &\\
& =m'\circ\Big(m'\circ(\alpha\circ s,\beta\circ s),\Psi_3\Big)\stackrel{*}{=}m'\circ\Big(\alpha\circ
   s,m'(\beta\circ s,\Psi_3)\Big)\stackrel{(\ref{eq-11})}{=} &\\
& \stackrel{(\ref{eq-11})}{=}m'\circ\Big(\alpha\circ s,m'\circ(\Psi_2,\beta\circ t)\Big)\stackrel{*}{=}
   m'\circ\Big(m'\circ(\alpha\circ s,\Psi_2),\beta\circ t\Big)\stackrel{(\ref{eq-10})}{=} &\\
& \stackrel{(\ref{eq-10})}{=}m'\circ\Big(m'\circ(\Psi_1,\alpha\circ t),\beta\circ t\Big)\stackrel{*}{=}
  m'\circ\Big(\Psi_1,m'\circ(\alpha\circ t,\beta\circ t)\Big)= &\\
& =m'\circ\Big(\Psi_1,m'\circ(\alpha,\beta)\circ t\Big)=m'\circ\Big(\Psi_1,(\beta\odot\alpha)\circ
  t\Big) &
\end{eqnarray*}

where all the passages denoted with $\stackrel{*}{=}$ are just axiom (iii) of definition \ref{groupoid-object}. So
we have proved that $\beta\odot\alpha$ is a natural transformation as we claimed.
\end{proof}

\textbf{Lemma \ref{horizontal-composition-groupoid}. }\emph{$\beta\ast\alpha$ is a natural transformation
from $(\phi_1\circ\psi_1,\Phi_1\circ\Psi_1)$ to $(\phi_2\circ\psi_2,\Phi_2\circ\Psi_2)$.}

\begin{proof}
Since all the maps involved are morphisms in $\ca{C}$ we get that also $\beta\ast\alpha$ is a morphism in this 
category; moreover we have that:

\begin{eqnarray*}
& s''\circ(\beta\ast\alpha)=s''\circ m''\circ(\Phi_1\circ\alpha,\beta\circ\psi_2)=s''\circ\Phi_1\circ
  \alpha=\phi_1\circ s'\circ\alpha\stackrel{(\ref{eq-14})}{=}\phi_1\circ\psi_1& \\
& \textrm{and\quad}t''\circ(\beta\ast\alpha)=t''\circ m''\circ(\Phi_1\circ\alpha,\beta\circ\psi_2)=
   t''\circ\beta\circ\psi_2\stackrel{(\ref{eq-15})}{=}\phi_2\circ\psi_2;&
\end{eqnarray*}

hence condition (i) of definition \ref{nat-tran-group} is satisfied. Let us prove (ii); in order to do that, we
recall that by definition of natural transformation for $\alpha$ and $\beta$ we have:

\begin{eqnarray}
\label{eq-16} & m'\circ(\alpha\circ s,\Psi_2)=m'\circ(\Psi_1,\alpha\circ t);&\\
\label{eq-17} & m''\circ(\beta\circ s',\Phi_2)=m''\circ(\Phi_1,\beta\circ t').&
\end{eqnarray}

Moreover, let us state the following preliminary results:

\begin{equation}\label{eq-19}
(\beta\ast\alpha)\circ s=m''\circ(\Phi_1\circ\alpha,\beta\circ\psi_2)\circ s=m''\circ(\Phi_1\circ\alpha\circ s,
\beta\circ\psi_2\circ s);
\end{equation}

\begin{eqnarray}
\nonumber & m''\circ(\beta\circ \psi_2\circ s,\Phi_2\circ\Psi_2)=m''\circ(\beta\circ s'\circ \Psi_2,\Phi_2\circ
  \Psi_2)= &\\
\nonumber  & =m''\circ(\beta\circ s',\Phi_2)\circ\Psi_2\stackrel{(\ref{eq-17})}{=}m''\circ(\Phi_1,\beta\circ t')
   \circ\Psi_2= &\\
\label{eq-20} & =m''\circ(\Phi_1\circ\Psi_2,\beta\circ t'\circ\Psi_2)=m''\circ(\Phi_1\circ\Psi_2,\beta\circ\psi_2
  \circ t); &
\end{eqnarray}

\begin{eqnarray}
\nonumber & m''\circ(\Phi_1\circ\alpha\circ s,\Phi_1\circ\Psi_2)=m''\circ(\Phi_1\times\Phi_1)\circ(\alpha\circ s,
  \Psi_2)= &\\
\nonumber & =\Phi_1\circ m'\circ(\alpha\circ s,\Psi_2)\stackrel{(\ref{eq-16})}{=}\Phi_1\circ m'\circ(\Psi_1,\alpha
  \circ t)= &\\
\label{eq-21} & =m''\circ(\Phi_1\times\Phi_1)\circ(\Psi_1,\alpha\circ t)=m''\circ(\Phi_1\circ\Psi_1,\Phi_1\circ
  \alpha\circ t);&
\end{eqnarray}

\begin{eqnarray}\label{eq-22} 
m''\circ(\Phi_1\circ\alpha\circ t,\beta\circ\psi_2\circ t)=m''\circ(\Phi_1\circ\alpha,\beta\circ\psi_2)\circ t=
(\beta\ast\alpha)\circ t.
\end{eqnarray}

So we have that:

\begin{eqnarray*}
& m''\circ\Big((\beta\ast\alpha)\circ s,\Phi_2\circ\Psi_2\Big)\stackrel{(\ref{eq-19})}{=}
    m''\circ\Big(m''\circ(\Phi_1\circ\alpha\circ s,\beta\circ\psi_2\circ s),\Phi_2\circ\Psi_2\Big)=&\\
& =m''\circ\Big(\Phi_1\circ\alpha\circ s,m''\circ(\beta\circ\psi_2\circ s,\Phi_2\circ\Psi_2)\Big)
    \stackrel{(\ref{eq-20})}{=} &\\
& \stackrel{(\ref{eq-20})}{=}m''\circ\Big(\Phi_1\circ\alpha\circ s,m''\circ(\Phi_1\circ\Psi_2,
    \beta\circ\psi_2\circ t)\Big)=&\\
& =m''\circ\Big(m''\circ(\Phi_1\circ\alpha\circ s,\Phi_1\circ\Psi_2),\beta\circ\psi_2\circ t\Big)
    \stackrel{(\ref{eq-21})}{=} &\\
& \stackrel{(\ref{eq-21})}{=}m''\circ\Big(m''\circ(\Phi_1\circ\Psi_1,\Phi_1\circ\alpha\circ t),\beta
    \circ\psi_2\circ t\Big)=&\\
& =m''\circ\Big(\Phi_1\circ\Psi_1,m''\circ(\Phi_1\circ\alpha\circ t,\beta\circ\psi_2\circ t)\Big)
    \stackrel{(\ref{eq-22})}{=}m''\circ\Big(\Phi_1\circ\Psi_1,(\beta\ast\alpha)\circ t\Big)&
\end{eqnarray*}

where all the passages without label are just property (iii) of groupoid objects; so also property (ii) of
definition \ref{nat-tran-group} is satisfied.
\end{proof}

\textbf{Lemma \ref{m-well-defined}.}\emph{The map $m$ is well defined.}

\begin{proof}
In order to prove the statement, we have to solve 2 problems:

\begin{enumerate}[(i)]\parindent=0pt
\item first of all, let us fix representatives $(\LA{ih},\tx_i,\LA{ij})$ and $(\LA{kj},\tx_k,\LA{kl})$ for the 2 
points we have to ``multiply''. Our previous description of the multiplication map requires to \emph{choose} a 
uniformizing system $\unif{f}$, a point $\tx_f\in\tU_f$ and embeddings $\LA{fi},\LA{fk}$ making (\ref{eq-41})
commute. However, this construction uses lemma \ref{useful-lemma}, which gives only the existence of such data,
but not the uniqueness, so we have to verify that our construction does not depend on different completions of
(\ref{eq-40});

\item we have to prove that the multiplication does not depend on the 
representatives chosen for $[\LA{ih},\tx_i,\LA{ij}]$ and for $[\LA{kj},\tx_k,\LA{kl}]$.
\end{enumerate}

Let us solve these problems separately.

\begin{enumerate}[(i)]\parindent=0pt
\item Let us suppose we can ``complete'' a diagram (\ref{eq-40}) in two different ways:

\begin{equation}\label{eq-42}
\begin{tikzpicture}[scale=0.8]
    \def\x{3}
    \def\y{-0.55}
    \def\z{-2}
    \def\w{-2}
    \node (A2_0) at (0*\x, 2*\y) {$\tU_h$};
    \node (A2_1) at (1*\x, 2*\y) {$\tU_i$};
    \node (A2_2) at (2*\x, 2*\y) {$\tU_j$};
    \node (A2_3) at (3*\x, 2*\y) {$\tU_k$};
    \node (A2_4) at (4*\x, 2*\y) {$\tU_l.$};
    \node (B1_1) at (2*\x, 4*\y) {$\curvearrowright$};
    \node (C1_1) at (2*\x, 0) {$\curvearrowright$};
    \node (A3_2) at (2*\x, 2*\y+\z) {$\tU_f$};
    \node[rotate=90] (A4_2) at (2*\x, 3*\y+\z) {$\in$};
    \node (A5_2) at (2*\x, 4*\y+\z) {$\tx_f$};
    \node (B3_2) at (2*\x, 2.9+\w) {$\tU_r$};
    \node[rotate=270] (B4_2) at (2*\x, 3.35+\w) {$\in$};
    \node (B5_2) at (2*\x, 3.8+\w) {$\tx_r$};
    \path (A2_3) edge [->] node [auto] {$\scriptstyle{\LA{kl}}$} (A2_4);
    \path (A2_1) edge [->] node [auto,swap] {$\scriptstyle{\LA{ih}}$} (A2_0);
    \path (A2_1) edge [->] node [auto] {$\scriptstyle{\LA{ij}}$} (A2_2);
    \path (A2_3) edge [->] node [auto,swap] {$\scriptstyle{\LA{kj}}$} (A2_2);
    \path (A3_2) edge [->] node [auto] {$\scriptstyle{\LA{fi}}$} (A2_1);
    \path (A3_2) edge [->] node [auto,swap] {$\scriptstyle{\LA{fk}}$} (A2_3);
    \path (B3_2) edge [->] node [auto,swap] {$\scriptstyle{\LA{ri}}$} (A2_1);
    \path (B3_2) edge [->] node [auto] {$\scriptstyle{\LA{rk}}$} (A2_3);
\end{tikzpicture}
\end{equation}

Now we have that $\LA{ri}(\tx_r)=\LA{fi}(\tx_f)$, so we can apply lemma \ref{useful-lemma} and we get that
there exist a uniformizing system $\unif{s}$, a point $\tx_s\in\tU_s$ and a pair of embeddings $\LA{sr},\LA{sf}$
which make the following diagrams commute:

\begin{equation}\label{eq-43}
\begin{tikzpicture}[scale=0.8]
    \def\x{1.5}
    \def\y{-1.5}
    \node (A0_1) at (1*\x, 0*\y) {$\tU_s$};
    \node (A1_0) at (0*\x, 1*\y) {$\tU_r$};
    \node (A1_1) at (1*\x, 1*\y) {$\curvearrowright$};
    \node (A1_2) at (2*\x, 1*\y) {$\tU_f$};
    \node (A2_1) at (1*\x, 2*\y) {$\tU_i$};
    \path (A1_2) edge [->] node [auto] {$\scriptstyle{\LA{fi}}$} (A2_1);
    \path (A1_0) edge [->] node [auto,swap] {$\scriptstyle{\LA{ri}}$} (A2_1);
    \path (A0_1) edge [->] node [auto] {$\scriptstyle{\LA{sf}}$} (A1_2);
    \path (A0_1) edge [->] node [auto,swap] {$\scriptstyle{\LA{sr}}$} (A1_0);

    \def\z{4.8}
    \node (B0_1) at (1*\x+\z, 0*\y) {$\tx_s$};
    \node (B1_0) at (0*\x+\z, 1*\y) {$\tx_r$};
    \node (B1_1) at (1*\x+\z, 1*\y) {$\curvearrowright$};
    \node (B1_2) at (2*\x+\z, 1*\y) {$\tx_f.$};
    \node (B2_1) at (1*\x+\z, 2*\y) {$\tx_i$};
    \path (B1_2) edge [->] node [auto] {$\scriptstyle{\LA{fi}}$} (B2_1);
    \path (B1_0) edge [->] node [auto,swap] {$\scriptstyle{\LA{ri}}$} (B2_1);
    \path (B0_1) edge [->] node [auto] {$\scriptstyle{\LA{sf}}$} (B1_2);
    \path (B0_1) edge [->] node [auto,swap] {$\scriptstyle{\LA{sr}}$} (B1_0);
\end{tikzpicture}
\end{equation}

Now using (\ref{eq-42}) and (\ref{eq-43}) together we get that 
$\LA{kj}\circ\LA{fk}\circ\LA{sf}=\LA{kj}\circ\LA{rk}\circ\LA{sr}$
and we recall that $\LA{kj}$ is an embedding, hence in particular it is injective, so we have that:

\begin{equation}\label{eq-44}
\LA{fk}\circ\LA{sf}=\LA{rk}\circ\LA{sr}.
\end{equation}

Now if we combine together diagram (\ref{eq-43}) and equation (\ref{eq-44}), we get commutative diagrams:

\begin{equation}\label{eq-45}
\begin{tikzpicture}[scale=0.8]
    \def\x{1.3}
    \def\y{-1.2}
    \node (A2_0) at (2*\x, 0*\y) {$\tU_r$};
    \node (A0_2) at (0*\x, 2*\y) {$\tU_h$};
    \node (A2_2) at (2*\x, 2*\y) {$\tU_s$};
    \node (A4_2) at (4*\x, 2*\y) {$\tU_l$};
    \node (A2_4) at (2*\x, 4*\y) {$\tU_f$};
    \node (A1_2) at (1*\x, 2*\y) {$\curvearrowright$};
    \node (A3_2) at (3*\x, 2*\y) {$\curvearrowright$};
    \path (A2_0) edge [->] node [auto,swap] {$\scriptstyle{\LA{ih}\circ\LA{ri}}$} (A0_2);
    \path (A2_0) edge [->] node [auto] {$\scriptstyle{\LA{kl}\circ\LA{rk}}$} (A4_2);
    \path (A2_4) edge [->] node [auto] {$\scriptstyle{\LA{ih}\circ\LA{fi}}$} (A0_2);
    \path (A2_4) edge [->] node [auto,swap] {$\scriptstyle{\LA{kl}\circ\LA{fk}}$} (A4_2);
    \path (A2_2) edge [->] node [auto,swap] {$\scriptstyle{\LA{sr}}$} (A2_0);
    \path (A2_2) edge [->] node [auto] {$\scriptstyle{\LA{sf}}$} (A2_4);

    \def\z{7}
    \node (B2_0) at (2*\x+\z, 0*\y) {$\tx_r$};
    \node (B0_2) at (0*\x+\z, 2*\y) {$\tx_h:=\LA{hi}(\tx_i)$};
    \node (B2_2) at (2*\x+\z, 2*\y) {$\tx_s$};
    \node (B4_2) at (4*\x+\z, 2*\y) {$\tx_l:=\LA{kl}(\tx_k).$};
    \node (B2_4) at (2*\x+\z, 4*\y) {$\tx_f$};
    \node (B1_2) at (1.2*\x+\z, 2*\y) {$\curvearrowright$};
    \node (B3_2) at (2.8*\x+\z, 2*\y) {$\curvearrowright$};
    \path (B2_0) edge [->] node [auto,swap] {$\scriptstyle{\LA{ih}\circ\LA{ri}}$} (B0_2);
    \path (B2_0) edge [->] node [auto] {$\scriptstyle{\LA{kl}\circ\LA{rk}}$} (B4_2);
    \path (B2_4) edge [->] node [auto] {$\scriptstyle{\LA{ih}\circ\LA{fi}}$} (B0_2);
    \path (B2_4) edge [->] node [auto,swap] {$\scriptstyle{\LA{kl}\circ\LA{fk}}$} (B4_2);
    \path (B2_2) edge [->] node [auto,swap] {$\scriptstyle{\LA{sr}}$} (B2_0);
    \path (B2_2) edge [->] node [auto] {$\scriptstyle{\LA{sf}}$} (B2_4);
\end{tikzpicture}
\end{equation}

This means that $(\LA{ih}\circ\LA{fi},\tx_f,\LA{kl}\circ\LA{fk})\sim
(\LA{ih}\circ\LA{ri},\tx_r,\LA{kl}\circ\LA{rk})$, hence (i) is solved.

\item Let us suppose we have chosen another representative $(\LA{sh},\tx_s,\LA{sj})$ for $[\LA{ih},\tx_i,\LA{ij}]$.
Using remark \ref{good-definitions-on-R} it suffices to consider the case when the two representatives are
related by a diagram of the form (\ref{eq-38}); in other words, we can assume there exists an embedding
$\LA{si}$ such that:

$$\LA{sh}=\LA{ih}\circ\LA{si},\quad\LA{sj}=\LA{ij}\circ\LA{si}\AND\LA{si}(\tx_s)=\tx_i.$$

Now if we want to compute $m([\LA{ih},\tx_i,\LA{ij}],[\LA{kj},\tx_k,\LA{kl}])$ using this new representative for
the first point, we have to use lemma \ref{useful-lemma} in order to choose a uniformizing system $\unif{r}$
together with a point $\tx_r\in\tU_r$ and a pair of embeddings $\LA{rs},\LA{rk}$ such that:

$$\LA{sj}\circ\LA{rs}=\LA{kj}\circ\LA{rk},\quad\LA{rs}(\tx_r)=\tx_s\AND\LA{rk}(\tx_r)=\tx_k.$$ 

Note that there are no problems in choosing all these data, since we have already proved (i). In other words, we
are using commutative diagrams of the form:

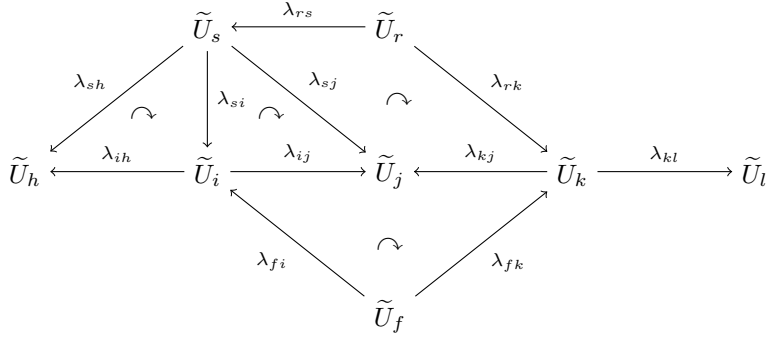
\[\begin{tikzpicture}[scale=0.8]
    \def\x{1.5}
    \def\y{-1.2}
    \node (A0_2) at (2*\x, 0*\y) {$\tU_s$};
    \node (A0_4) at (4*\x, 0*\y) {$\tU_r$};
    \node (A1_1) at (1.3*\x, 1.2*\y) {$\curvearrowright$};
    \node (A1_3) at (2.7*\x, 1.2*\y) {$\curvearrowright$};
    \node (A1_5) at (4.1*\x, 1*\y) {$\curvearrowright$};
    \node (A2_0) at (0*\x, 2*\y) {$\tU_h$};
    \node (A2_2) at (2*\x, 2*\y) {$\tU_i$};
    \node (A2_4) at (4*\x, 2*\y) {$\tU_j$};
    \node (A2_6) at (6*\x, 2*\y) {$\tU_k$};
    \node (A2_8) at (8*\x, 2*\y) {$\tU_l$};
    \node (A3_4) at (4*\x, 3*\y) {$\curvearrowright$};
    \node (A4_4) at (4*\x, 4*\y) {$\tU_f$};
    \path (A0_2) edge [->] node [auto,swap] {$\scriptstyle{\LA{sh}}$} (A2_0);
    \path (A0_4) edge [->] node [auto,swap] {$\scriptstyle{\LA{rs}}$} (A0_2);
    \path (A0_2) edge [->] node [auto] {$\scriptstyle{\LA{si}}$} (A2_2);
    \path (A0_2) edge [->] node [auto] {$\scriptstyle{\LA{sj}}$} (A2_4);
    \path (A4_4) edge [->] node [auto] {$\scriptstyle{\LA{fi}}$} (A2_2);
    \path (A2_2) edge [->] node [auto] {$\scriptstyle{\LA{ij}}$} (A2_4);
    \path (A4_4) edge [->] node [auto,swap] {$\scriptstyle{\LA{fk}}$} (A2_6);
    \path (A2_2) edge [->] node [auto,swap] {$\scriptstyle{\LA{ih}}$} (A2_0);
    \path (A0_4) edge [->] node [auto] {$\scriptstyle{\LA{rk}}$} (A2_6);
    \path (A2_6) edge [->] node [auto] {$\scriptstyle{\LA{kl}}$} (A2_8);
    \path (A2_6) edge [->] node [auto,swap] {$\scriptstyle{\LA{kj}}$} (A2_4);
\end{tikzpicture}\]

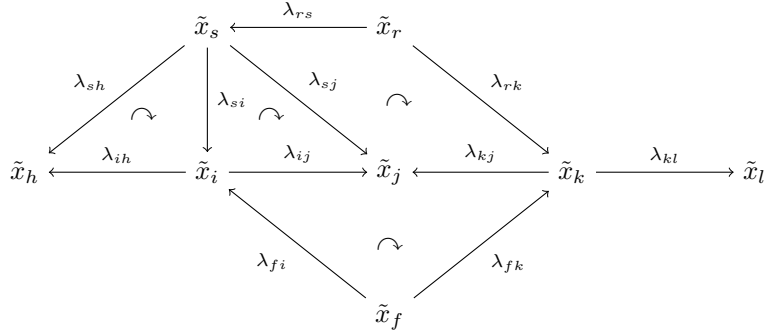
\[\begin{tikzpicture}[scale=0.8]
    \def\x{1.5}
    \def\y{-1.2}
    \node (A0_2) at (2*\x, 0*\y) {$\tx_s$};
    \node (A0_4) at (4*\x, 0*\y) {$\tx_r$};
    \node (A1_1) at (1.3*\x, 1.2*\y) {$\curvearrowright$};
    \node (A1_3) at (2.7*\x, 1.2*\y) {$\curvearrowright$};
    \node (A1_5) at (4.1*\x, 1*\y) {$\curvearrowright$};
    \node (A2_0) at (0*\x, 2*\y) {$\tx_h$};
    \node (A2_2) at (2*\x, 2*\y) {$\tx_i$};
    \node (A2_4) at (4*\x, 2*\y) {$\tx_j$};
    \node (A2_6) at (6*\x, 2*\y) {$\tx_k$};
    \node (A2_8) at (8*\x, 2*\y) {$\tx_l$};
    \node (A3_4) at (4*\x, 3*\y) {$\curvearrowright$};
    \node (A4_4) at (4*\x, 4*\y) {$\tx_f$};
    \path (A0_2) edge [->] node [auto,swap] {$\scriptstyle{\LA{sh}}$} (A2_0);
    \path (A0_4) edge [->] node [auto,swap] {$\scriptstyle{\LA{rs}}$} (A0_2);
    \path (A0_2) edge [->] node [auto] {$\scriptstyle{\LA{si}}$} (A2_2);
    \path (A0_2) edge [->] node [auto] {$\scriptstyle{\LA{sj}}$} (A2_4);
    \path (A4_4) edge [->] node [auto] {$\scriptstyle{\LA{fi}}$} (A2_2);
    \path (A2_2) edge [->] node [auto] {$\scriptstyle{\LA{ij}}$} (A2_4);
    \path (A4_4) edge [->] node [auto,swap] {$\scriptstyle{\LA{fk}}$} (A2_6);
    \path (A2_2) edge [->] node [auto,swap] {$\scriptstyle{\LA{ih}}$} (A2_0);
    \path (A0_4) edge [->] node [auto] {$\scriptstyle{\LA{rk}}$} (A2_6);
    \path (A2_6) edge [->] node [auto] {$\scriptstyle{\LA{kl}}$} (A2_8);
    \path (A2_6) edge [->] node [auto,swap] {$\scriptstyle{\LA{kj}}$} (A2_4);
\end{tikzpicture}\]

where for simplicity we have used the following notations: $\tx_h:=\LA{sh}(\tx_s)=\LA{ih}\circ\LA{si}(\tx_s)$,
$\tx_l:=\LA{kl}(\tx_k)$ and $\tx_j:=\LA{ij}(\tx_i)$. So we get the diagram:

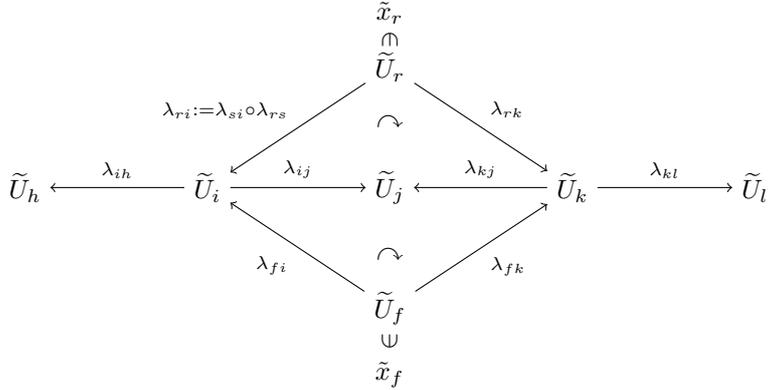
\[\begin{tikzpicture}[scale=0.8]
    \def\x{3}
    \def\y{-0.55}
    \def\z{-2}    
    \def\w{-2}
     \node (A2_0) at (0*\x, 2*\y) {$\tU_h$};
    \node (A2_1) at (1*\x, 2*\y) {$\tU_i$};
    \node (A2_2) at (2*\x, 2*\y) {$\tU_j$};
    \node (A2_3) at (3*\x, 2*\y) {$\tU_k$};
    \node (A2_4) at (4*\x, 2*\y) {$\tU_l$};
    \node (B1_1) at (2*\x, 4*\y) {$\curvearrowright$};
    \node (C1_1) at (2*\x, 0) {$\curvearrowright$};
    \node (A3_2) at (2*\x, 2*\y+\z) {$\tU_f$};
    \node[rotate=90] (A4_2) at (2*\x, 3*\y+\z) {$\in$};
    \node (A5_2) at (2*\x, 4*\y+\z) {$\tx_f$};
    \node (B3_2) at (2*\x, 2.9+\w) {$\tU_r$};
    \node[rotate=270] (B4_2) at (2*\x, 3.35+\w) {$\in$};
    \node (B5_2) at (2*\x, 3.8+\w) {$\tx_r$};
    \path (A2_3) edge [->] node [auto] {$\scriptstyle{\LA{kl}}$} (A2_4);
    \path (A2_1) edge [->] node [auto,swap] {$\scriptstyle{\LA{ih}}$} (A2_0);
    \path (A2_1) edge [->] node [auto] {$\scriptstyle{\LA{ij}}$} (A2_2);
    \path (A2_3) edge [->] node [auto,swap] {$\scriptstyle{\LA{kj}}$} (A2_2);
    \path (A3_2) edge [->] node [auto] {$\scriptstyle{\LA{fi}}$} (A2_1);
    \path (A3_2) edge [->] node [auto,swap] {$\scriptstyle{\LA{fk}}$} (A2_3);
    \path (B3_2) edge [->] node [auto,swap] {$\scriptstyle{\LA{ri}:=\LA{si}\circ\LA{rs}}$} (A2_1);
    \path (B3_2) edge [->] node [auto] {$\scriptstyle{\LA{rk}}$} (A2_3);
\end{tikzpicture}\]

with $\LA{ri}(\tx_r)=\LA{fi}(\tx_f)$, so we can repeat the same construction of (i) in order to get a diagram of
the form (\ref{eq-45}), so we obtain:

$$(\LA{ih}\circ\LA{ri},\tx_r,\LA{kl}\circ\LA{rk})\sim(\LA{ih}\circ\LA{fi},\tx_f,\LA{kl}\circ\LA{fk}).$$

Now by definition of $\LA{ri}$ we have $\LA{ih}\circ\LA{ri}=\LA{ih}\circ\LA{si}\circ\LA{rs}=\LA{sh}\circ\LA{rs}$, 
hence:

$$(\LA{sh}\circ\LA{rs},\tx_r,\LA{kl}\circ\LA{rk})\sim(\LA{ih}\circ\LA{fi},\tx_f,\LA{kl}\circ\LA{fk}).$$

These two points are the multiplication obtained when we choose representatives $(\LA{ih},\tx_i,\LA{ij})$
and $(\LA{sh},\tx_s,\LA{sj})$ for the same point, so the multiplication does not depend on the representative
chosen for $[\LA{ih},\tx_i,\LA{ij}]$. In the same way one can also prove that the multiplication doesn't
depend on the representative chosen for the point $[\LA{kj},\tx_k,\LA{kl}]$.
\end{enumerate}
\end{proof}

\end{document}